\documentclass[11pt]{compositio}

\pdfoutput=1

\usepackage[table]{xcolor}
\usepackage{
  amssymb,
  amsmath,
  amsthm,
  thmtools,
  stmaryrd,
  amscd,
  verbatim,
  microtype,
  hyperref,
  tikz,
  tikz-cd,
  subcaption,
  paralist
}

\usetikzlibrary{patterns}
\tikzset{
  every picture/.style={thick, font=\scriptsize}
}
\tikzstyle{curve}=[thick, font=\scriptsize, draw, circle, minimum width=5pt, inner sep=1pt, semithick]

\usepackage[charter]{mathdesign}
\usepackage{eucal}
\usepackage[final]{showlabels}

\newcommand{\todo}[1]{(#1) \marginpar{ToDo}}

\newcommand{\Z}{{\mathbb Z}}

\newcommand{\C}{{\mathbb C}}
\newcommand{\Q}{{\mathbb Q}}
\newcommand{\A}{{\mathbb A}}
\newcommand{\Gm}{{\mathbb G}_m}
\DeclareMathOperator{\PGL}{PGL}

\renewcommand{\O}{\mathcal O}
\newcommand{\from}{\colon}
\newcommand{\PP}{{\mathbb{P}}}
\newcommand{\F}{{\mathbb{F}}}
\DeclareMathOperator{\id}{id}

\DeclareMathOperator{\spec}{spec}
\DeclareMathOperator{\Proj}{Proj}
\DeclareMathOperator{\Pic}{Pic}
\renewcommand{\k}{{\mathbb K}}
\DeclareMathOperator{\br}{br}
\DeclareMathOperator{\mult}{mult}
\DeclareMathOperator{\Aut}{Aut}
\renewcommand{\H}{\mathcal H}

\DeclareMathOperator{\NE}{NE}
\DeclareMathOperator{\Bl}{Bl} %Blowup
\DeclareMathOperator{\Ext}{Ext}

\DeclareMathOperator{\SExt}{\mathcal{E}\!{\it xt}}
\DeclareMathOperator{\length}{length}

\newtheorem{introtheorem}{Theorem}
\newtheorem{theorem}{Theorem}[section]
\newtheorem{lemma}[theorem]{Lemma}
\newtheorem{proposition}[theorem]{Proposition}

\newtheorem{corollary}[theorem]{Corollary}
\newtheorem{question}[theorem]{Question}

\theoremstyle{definition}
\newtheorem{definition}[theorem]{Definition}
\newtheorem{example}[theorem]{Example}

\theoremstyle{remark}
\newtheorem{remark}[theorem]{Remark}

\title[Log surfaces, covers, and curves of genus 4]{Stable log surfaces, admissible covers, and canonical curves of genus 4}

\author{Anand Deopurkar}
\email{anand.deopurkar@anu.edu.au}
\address{Anand Deopurkar, Mathematical Sciences Institute, The Australian National University, Acton, ACT, Australia}

\author{Changho Han}
\email{Changho.Han@uga.edu}
\address{Changho Han, Department of Mathematics, University of Georgia, Athens, GA, USA}

\thanks{A.D. is supported by the Australian Research Council Award DE180101360 and the AMS Simons Travel Grant.\\ \indent C.H. is supported by the National Sciences and Engineering Research Council of Canada (NSERC), [PGSD3-487436-2016].}

\begin{document}
	
\begin{abstract}
  We explicitly describe the KSBA/Hacking compactification of a moduli space of log surfaces of Picard rank 2.
  The space parametrizes log pairs $(S, D)$ where $S$ is a degeneration of $\PP^1 \times \PP^1$ and $D \subset S$ is a degeneration of a curve of class $(3,3)$.
  We prove that the compactified moduli space is a smooth Deligne--Mumford stack with 4 boundary components.
  We relate it to the moduli space of genus 4 curves; we show that it compactifies the blow-up of the hyperelliptic locus.
  We also relate it to a compactification of the Hurwitz space of triple coverings of $\PP^1$ by genus 4 curves.
\end{abstract}

\maketitle

\section{Introduction}\label{sec:intro}
% Outline:

% * Recall the situation for curves.

% * Higher dimensions, surfaces: there exists a moduli space, but its geometry is not very well understood.

% * Only a handful of cases where we know the geometry. Recall Hassett and Hacking's work.

% * Introduce our case.

% * State the main theorem.

% * Ingredients of the proof: 1. The weighted Hurwitz space (State the theorem that there exists a morphism from the weighted Hurwitz space).

% * Ingredients of the proof: 2. Stable replacements -- describes the points of the moduli space

% * Ingredients of the proof: 3. Q-Gorenstein deformation theory.

% * Geometry of the moduli spcae: Divisors, Picard rank, open subset isomorphic to the blow up of the hyperelliptic locus.

% * Further questions/problems: Spaces when the weight is higher. Curves on other del Pezzo surfaces.

The goal of this paper is to describe a compact moduli space $\mathfrak X$ that lies at the cusp of three different areas of study of moduli spaces in algebraic geometry: (1) the study of compact moduli spaces of surfaces of log general type, (2) the study of the birational geometry of the moduli space of curves, and (3) the study of alternative compactifications of Hurwitz spaces parametrizing branched coverings.

Before we explain why the space $\mathfrak X$ is remarkable from these three points of view, let us first define it.
Consider a pair $(S, D)$, where $S \cong \PP^1 \times \PP^1$ and $D \subset S$ is a smooth divisor of class $(3,3)$.
Observe that for all $w > 2/3$, the pair $(S, wD)$ is a surface of log general type.
Set $w = 2/3 + \epsilon$, where $0 < \epsilon \ll 1$.
Then $\mathfrak X$ is the compactification of the space of pairs $(S, wD)$ constructed by Koll\'ar--Shepherd-Barron/Alexeev \cite{kol.she:88,ale:96}
(the idea of taking $\epsilon$ close to 0 is inspired by the work of Hacking \cite{hac:04}). 
The compactification $\mathfrak X$ parametrizes stable semi log canonical pairs of log general type.
We recall the definition in detail in the main text.
For now, it suffices to say that it is the analogue in higher dimensions of the Deligne--Mumford compactification $\overline {\mathcal M}_{g,n}$.

Having defined $\mathfrak X$, let us explain its significance, beginning with the first aspect.
In sharp contrast to the case of curves, it is rare to have a complete description of the boundary of the KSBA compactification of surfaces.
Furthermore, and again in contrast to the case of curves, the KSBA compactification is usually highly singular, even reducible with components of unexpected dimensions.
Nevertheless, bucking the general expectations, we are able to give an explicit description of all the boundary points of $\mathfrak X$, and show that $\mathfrak X$ is quite well-behaved.

In the pairs $(S, D)$ parameterized by $\mathfrak X$, the surface $S$ is a del Pezzo, and $D$ is a rational multiple of the canonical class.
Extending Hacking's approach in \cite{hac:04}, one can uniformly treat some general aspects of the moduli theory of such pairs, which we call `almost K3 pairs' (see \autoref{sec:stable-pair}).
However, the only instance of a reasonably explicit description of the boundary of such moduli spaces has been achieved by Hacking for $S = \PP^2$.
This description relies heavily on our complete understanding of degenerations of $\PP^2$, thanks to the work of Manetti \cite{man:95} and Hacking--Prokhorov  \cite{hac.pro:10}, which in turn crucially uses that the Picard rank of $\PP^2$ is $1$.
To our knowledge, the moduli space $\mathfrak X$ is the first explicit example for higher Picard rank.
It seems a difficult problem to achieve a uniform, yet explicit description of the boundary for all del Pezzos.
We are currently investigating this direction.
We refer the reader to some other explicit compactifications of log surfaces that consider special kinds of divisors in their linear series, for example the work of Moon--Schaffler \cite{moo.sch:18} and Ascher--Bejleri \cite{asc.bej:18}.

Let us now give the explicit description of the boundary. 
Denote by  $\mathfrak X^\circ$ the open subset of $\mathfrak X$ that parametrizes $(S, wD)$ with $S \cong \PP^1 \times \PP^1$ and $D \subset S$ smooth of type $(3,3)$.
\begin{introtheorem}\label{thm:smooth}
  The weighted KSBA compactification $\mathfrak X$ is an irreducible and smooth Deligne--Mumford stack over $\k$.
  The closed substack $\mathfrak X \setminus \mathfrak X^\circ$ is the union of 4 irreducible divisors.
\end{introtheorem}
We label the 4 boundary components $\mathfrak Z_0$, $\mathfrak Z_2$, $\mathfrak Z_4$, and $\mathfrak Z_{3,3}.$
The log surfaces parameterized by their generic points are as follows.
\begin{enumerate}
\item $\mathfrak Z_0$: The surface $S$ is a smooth quadric surface in $\PP^3$ and $D \subset S$ is a generic singular curve of class $(3,3)$.
\item $\mathfrak Z_2$: The surface $S$ is an irreducible singular quadric surface in $\PP^3$ and $D$ is a complete intersection of $S$ and a cubic surface in $\PP^3$.
\item $\mathfrak Z_4$: The surface $S$ is a $\Q$-Gorenstein smoothing of the $A_1$ singularity of $\PP(9,1,2)$ and $D$ is a smooth hyperelliptic curve away from the singular point of type $\frac{1}{9}(1,2)$.
\item $\mathfrak Z_{3,3}$: The surface $S$ is the union $\mathrm{Bl}_u\PP(3,1,1) \cup \mathrm{Bl}_v\PP(3,1,1)$ along a $\PP^1$ and $D$ is a nodal union of two non-Weierstrass genus 2 tails.
  Here, the blowups are along curvilinear subschemes $u,v$ of length 3 (see \autoref{subsec:mainthm} for a more precise description).
\end{enumerate}
We highlight some facts about the surfaces and the curves appearing at the boundary, referring the reader to \autoref{subsec:mainthm} for the complete list.
There are only 8 isomorphism classes of surfaces $S$, 4 of which are irreducible, and among the irreducible surfaces, 1 is non-toric (the general one parametrized by $\mathfrak Z_4$).
The curves $D$ are reduced, and only have $A_n$ singularities for $n \leq 4$.

We now come to the second aspect of $\mathfrak X$, namely its relationship to the birational geometry of $\mathcal M_g$.
Let $\mathfrak X_0 \subset \mathfrak X$ be the open subset that parametrizes pairs $(S, wD)$ with smooth $D$.
We have a forgetful morphism
\[ \mu \from \mathfrak X_0 \to \mathcal M_4.\]
Denote by $\mathcal H \subset \mathcal M_4$ the closed substack that parametrizes hyperelliptic curves.
\begin{introtheorem}\label{thm:blowup}
  The forgetful map $\mu \from \mathfrak X_0 \to \mathcal M_4$ induces an isomorphism
  \[ \mathfrak X_0 \cong \Bl_{\mathcal H} \mathcal M_4.\]
\end{introtheorem}
Thus, $\mathfrak X_0$ provides a modular interpretation of the blowup of the hyperelliptic locus in $\mathcal M_4$.
The map $\mu \from \mathfrak X_0 \to \mathcal M_4$ does not extend to a regular map from $\mathfrak X_0$ to any known modular compactification of $\mathcal M_4$ (see \autoref{prop:newcpct} for a more precise statement).
It does, however, extend to a morphism from $\mathfrak X$ to the (non-separated) moduli stack of Gorenstein curves.
It would be interesting to know if the image of $\mathfrak X$ in this stack is a modular compactification of $\mathcal M_4$ in the sense of \cite{fed.smy:13}, or in some other sense.

We recall another instance of a KSBA compactification that is closely related to the moduli of curves: the space of log pairs $(\PP^2, Q)$, where $Q \subset \PP^2$ is a quartic.
In this case, the KSBA compactification turns out to be isomorphic to the Deligne--Mumford compactification $\overline{\mathcal M}_3$ (see \cite{has:99}), and the KSBA/Hacking compactification is the first divisorial contraction of $\overline{\mathcal M}_3$ (see \cite[\S~11]{hac:04}).
In genus 4, as our theorem shows, the KSBA/Hacking compactification is something new.

We now discuss the connection of $\mathfrak X$ with the third aspect mentioned before, the alternate compactifications of Hurwitz spaces.
Recall that the Hurwitz space $\mathcal H^d_g$ is the moduli space of maps $\phi \from C \to P$, where $C$ is a smooth genus $g$ curve, $P$ is isomorphic to $\PP^1$, and $\phi$ is a finite map of degree $d$ with simple branching.
From general structure theorems of finite coverings, we know that the map $\phi$ gives an embedding $C \subset \PP E$, where $E$ is the so-called Tschirnhausen bundle of $\phi$ defined by $E^\vee = \phi_* \O_C / \O_P$.
For a general $[\phi] \in \mathcal H^3_4$, we have $E \cong \O(3) \oplus \O(3)$, and hence $\PP E \cong \PP^1 \times \PP^1$.
We thus get a rational map $\mathcal H^3_4 \dashrightarrow \mathfrak X$ defined by the rule $\phi \mapsto (S, C)$, where $S = \PP E$.
It is not too difficult to see that this rational map extends to a regular map $\mathcal H^3_4 \to \mathfrak X$.

At the heart of our analysis of $\mathfrak X$ is to find a compactification of $\mathcal H^3_4$ on which the map $\mathcal H^3_4 \to \mathfrak X$ extends to a regular, and hence surjective, map.
Unfortunately, the standard admissible cover compactification $\overline{\mathcal H}^3_4$ of $\mathcal H^3_4$ lacks this property.
We appeal to an alternate compactification $\overline{\mathcal H}^3_4(1/6 + \epsilon)$ constructed in \cite{deo:13}.
This compactification parametrizes weighted admissible covers $\phi \from C \to P$.
Roughly speaking, these are finite maps from a reduced stacky curve $C$ of arithmetic genus $4$ to a stacky nodal curve $P$ of arithmetic genus $0$ which are admissible over the nodes in the sense of Abramovich--Corti--Vistoli \cite{abr.cor.vis:03} and where the pointed curve $(P, \br \phi)$ is $(1/6+\epsilon)$-stable in the sense of Hassett \cite{has:03}.
The following theorem is the  major step towards understanding $\mathfrak X$.
\begin{introtheorem}\label{thm:HX}
  The map $\mathcal H^3_4 \to \mathfrak X$ extends to a regular map $\overline{\mathcal H}^3_4(1/6+\epsilon) \to \mathfrak X$.
\end{introtheorem}
The existence of the regular map $\overline{\mathcal H}^3_4(1/6+\epsilon) \to \mathfrak X$ is crucial for our analysis of $\mathfrak X$, and occupies the technical heart of the paper.
Thanks to this map, we obtain an explicit description of pairs parametrized by $\mathfrak X$ using the knowledge of the points of $\overline{\mathcal H}^3_4(1/6+\epsilon)$.
This description allows us to understand the connection between $\mathfrak X_0$ and $\mathcal M_4$, leading to \autoref{thm:blowup}.
It also allows us to directly verify that the $\Q$-Gorenstein deformations of the pairs we encounter are unobstructed, leading to \autoref{thm:smooth}.

\subsection*{Techniques}

For the proof of \autoref{thm:smooth}, we use a broad range of techniques from moduli theory and birational geometry, such as the theory of twisted admissible covers of Abramovich--Corti--Vistoli \cite{abr.cor.vis:03}, its extension to weighted admissible covers \cite{deo:13}, explicit Mori theory of 3-folds, and $\Q$-Gorenstein deformation theory \cite{hac:04}.

In broad strokes, the proof of \autoref{thm:smooth} goes as follows.
As mentioned before, our explicit description of $\mathfrak X$ fundamentally derives from showing that it is related to another, better understood moduli space $\overline {\mathcal H}^3_4(1/6 + \epsilon)$.
Via the Tschirnhausen construction, the points of $\overline {\mathcal H}^3_4(1/6 + \epsilon)$ can be interpreted as certain stacky log surfaces (see \autoref{sec:trigonal-surface}).
To go from such a stacky log surface to a stable surface, we prove a stable reduction theorem (\autoref{prop:stab}).
Roughly speaking, this theorem says that in a one parameter smoothing of the stacky surface, the unique stable limit depends only on the central fiber.
We prove this theorem by explicitly running a minimal model program on the total space of the one-parameter family.
As a result, we also get an explicit description of the stable limits.
We believe that the explicit Mori-theoretic transformation that are the steps of this minimal model program may be of independent interest (see \autoref{sec:flips}).
Finally, we compute the $\Q$-Gorenstein deformations and obstructions of the stable log surfaces obtained above, and conclude that $\mathfrak X$ is smooth.

\subsection*{Outline}
The paper is organized as follows.
\autoref{sec:stable-pair} recalls fundamental results about the moduli of stable log surfaces. We focus particularly on the case of almost K3 log surfaces, where it is possible to give a functorial description of the moduli stack.
\autoref{sec:trigonal-surface} describes the construction of a log surface from a triple covering of curves.
It culminates in an explicit description of the pairs obtained from triple coverings $C \to \PP^1$ where $C$ is a genus 4 curve.
\autoref{sec:flips} is devoted to a construction of two kinds of threefold flips that are necessary for the stable reduction of the surface pairs obtained in \autoref{sec:trigonal-surface}.
\autoref{sec:stable-replacement} uses the flips of \autoref{sec:flips} to carry out stable reduction for the unstable pairs.
As a result, by the end of this section, we obtain a list of the log surfaces parametrized by $\mathfrak X$.
\autoref{sec:deformation} shows that the $\Q$-Gorenstein deformation space of the pairs parametrized by $\mathfrak X$ is smooth.
\autoref{sec:geometry} relates $\mathfrak X$ to $\mathcal M_4$ and $\mathcal H^3_4$.

\subsection*{Conventions}
All schemes and stacks are locally of finite type over an algebraically closed field $\k$ of characteristic 0.
The projectivization of a vector bundle is the space of one dimensional quotients.
We go back and forth between Weil divisors and the associated divisorial sheaves without comment.

\subsection*{Acknowledgements}
We thank Valery Alexeev, Dori Bejleri, Maksym Fedorchuk, Paul Hacking, Joe Harris, J\'anos Koll\'ar, and Chenyang Xu for enlightening discussions and encouragement. We also thank the referee for reading the draft carefully and giving insightful feedbacks.

%%% Local Variables:
%%% mode: latex
%%% TeX-master: "main"
%%% End:

\section{Moduli spaces of `almost K3' stable log surfaces}\label{sec:stable-pair}
% Mostly expository.
% The discussion of the functors should go here.

In this section, we collect fundamental results on moduli of stable log surfaces of a particular kind that are used throughout this paper.
These log surfaces consist of a rational surface and a divisor whose class is proportional to the canonical class, and which is taken with a weight such that the log canonical divisor is just barely ample (hence the name `almost K3').
Hacking pioneered the study of such surfaces in \cite{hac:01} and \cite{hac:04}, hence those surfaces are called Hacking stable in other places in the literature \cite{gol:18,dev:19}.
Our treatment closely follows his work and benefits greatly from the subsequent enhancements due to Abramovich and Hassett \cite{abr.has:11}.

All objects are over an algebraically closed field $\k$ of characteristic 0.
We fix a pair of relatively prime positive integers $(m, n)$ with $m \leq n$.
After the general foundations in \autoref{sec:sls}--\autoref{sec:modstack}, we take $(m,n) = (2,3)$.

\subsection{Stable log surfaces}\label{sec:sls}

The following definition is motivated by \cite{hac:04}, where a similar object in the context of plane curves is called a stable pair.
\begin{definition}[(Stable log surface)]
  \label{def:sls}
  An \emph{almost K3 semi-stable log surface} over $\k$ is a pair $(S, D)$ where $S$ is a projective, reduced, connected, Cohen-Macaulay surface over $\k$ and $D$ is an effective Weil divisor on $S$ such that
  \begin{enumerate}
  \item no component of $D$ is contained in the singular locus of $S$;
  \item the pair $(S, m/n \cdot D)$ is semi log canonical;
  \item the divisor class $n K_S + mD$ is linearly equivalent to zero;
  \item we have $\chi(\O_S) = 1$.
  \end{enumerate}
  An \emph{almost K3 stable log surface} is an almost K3 semi-stable log surface $(S, D)$ such that for some $\epsilon > 0$
  \begin{enumerate}
  \item the pair $(S, (m/n+\epsilon) \cdot D)$ is semi log canonical (slc for short);
  \item $K_S + (m/n + \epsilon) \cdot D$ is ample.
  \end{enumerate}
\end{definition}
For brevity, from now on we refer to an almost K3 stable log surface simply as a \emph{stable log surface}. This is also called 
We also suppress the choice of $(m,n)$, which remains fixed throughout this section, and equal to $(2,3)$ after \autoref{sec:modstack}.

\begin{remark}\label{rmk:history}
  If $S$ is smooth, then $S$ is a del Pezzo surface.
  The case of $S \cong \PP^2$ (and its degenerations) was studied by Hacking in \cite{hac:01} and \cite{hac:04}.
  Our interest in this paper is the case of $S \cong \PP^1 \times \PP^1$ (and its degenerations) and $(m, n) = (2, 3)$.
\end{remark}

\begin{remark}
  Recently, DeVleming extended \autoref{def:sls} to pairs of arbitrary dimension, called $H$-stable pairs \cite[Definition 3.1]{dev:19}.
  Our definition \autoref{def:sls} is almost a special case where the ambient variety is $\PP^1 \times \PP^1$.
  The only difference is that we use a topological condition as in (iv) instead of a numerical condition coming from the top self-intersection of the canonical divisor.
\end{remark}

We recall some terms in the definition above, mainly to set the conventions.
A \emph{Weil divisor} $D$ on $S$ is a formal $\Z$-linear combination of irreducible pure codimension 1 subvarieties of $S$.
An \emph{effective} Weil divisor is one where all the coefficients are non-negative.
We assume throughout that our Weil divisors are Cartier in codimension 1.
That is, there exists an open subset $U \subset S$ whose complement is of codimension at least 2 such that the restriction of the divisor to $U$ is Cartier.
In \autoref{def:sls}, this is guaranteed by the first requirement.
A (generically Cartier) Weil divisor $D$ defines a reflexive sheaf $\O_S(D)$ by the formula
\[ \O_S(D) = i_* \O_U\left(D|_U\right),\]
where $U \subset S$ is an open set whose complement is of codimension at least 2 on which $D$ is Cartier and $i \from U \to S$ is the inclusion.
The divisor $D$ is Cartier if $\O_S(D)$ is invertible.
We say that $D$ is \emph{$\Q$-Cartier} if some multiple of $D$ is Cartier.

A coherent sheaf $F$ on $S$ is \emph{divisorial} if there exists an open inclusion $i \from U \to S$ with complement of codimension at least 2 such that $i^*F$ is invertible and
\[ F = i_*(i^*F),\]
A divisorial sheaf is isomorphic to $\O_S(D)$ for some Weil divisor $D$ on $S$.
Indeed, if $i^*F \cong \O_U(D^\circ)$, where $D^\circ$ is a Cartier divisor on $U$, then we may take $D = \overline{D^\circ}$.
Two Weil divisors $D_1$ and $D_2$ are linearly equivalent if and only if the sheaves $\O_S(D_1)$ and $\O_S(D_2)$ are isomorphic.
The divisor class $K_S$ is the linear equivalence class corresponding to the divisorial sheaf $\omega_S$.
For a divisorial sheaf $F$ and $n \in \Z$, denote by $F^{[n]}$ the divisorial sheaf $i_* \left( i^*F^{\otimes n}\right)$.
This operation corresponds to multiplication by $n$ on the associated divisors.

The semi log canonical condition in \autoref{def:sls} entails the following:
\begin{enumerate}
\item $S$ has at worst normal crossings singularities in codimension 1.
\item Let $K_S$ be the Weil divisor associated to the dualizing sheaf $\omega_S$.
  Then the $\Q$-Weil divisor $K_S + (m/n + \epsilon) \cdot D$ is $\Q$-Cartier (some integer multiple of it is Cartier).
\item Let $S^\nu \to S$ be the normalization, $B \subset S^\nu$ the pre-image of the double curve (the divisor defined by the different ideal), and $D^\nu \subset S^\nu$ the pre-image of $D$.
  Then the pair $(S^\nu, (m/n+\epsilon) \cdot D^\nu + B)$ is log canonical.
\end{enumerate}
Since $n K_S + m \cdot D$ is linearly equivalent to $0$, if $K_S + (m/n + \epsilon) \cdot D$ is $\Q$-Cartier, then both $K_S$ and $D$ are $\Q$-Cartier.
Note that if $(S,D)$ satisfies the conditions of \autoref{def:sls} for a particular $\epsilon$, then it also satisfies the definitions for all $\epsilon' < \epsilon$.

\subsection{Families of stable log surfaces}
Having defined stable log surfaces, we turn to families of them.
A n\"aive guess is that a family of stable log surfaces should be a flat morphism whose fibers are stable log surfaces.
However, to ensure a well-behaved moduli space---one in which numerical invariants are locally constant---additional conditions are needed.

Let $B$ be a $\k$-scheme, and $\pi \from S \to B$ a flat, Cohen-Macaulay morphism of relative dimension 2 with geometrically reduced fibers.
An \emph{effective relative Weil divisor} on $S$ is a subscheme $D \subset S$ such that there exists an open subset $U \subset S$ satisfying the following conditions:
\begin{enumerate}
\item for every geometric point $b \to B$, the complement of $U_b$ in $S_b$ is of codimension at least 2;
\item $D|_U \subset U$ is Cartier (its ideal sheaf is invertible) and flat over $B$;
\item $D$ is the scheme-theoretic closure of $D|_U$.
\end{enumerate}
A \emph{relative Weil divisor} is a formal difference of effective relative Weil divisors.
A \emph{divisorial sheaf} is a coherent sheaf $F$ on $S$ such that $i^* F$ is locally free and $F = i_*i^*F$, where $i \from U \to S$ is the inclusion of an open set as above.
A relative Weil divisor $D$ gives a divisorial sheaf $\O_S(D)$, and every divisorial sheaf is of this form.
Given a divisorial sheaf $F$ and $n \in \Z$, we have a divisorial sheaf $F^{[n]}$ defined as before.
If the geometric fibers $S_b$ are slc, then $\omega_{S/B}$ is a divisorial sheaf \cite[Example~8.18]{hac:01}.

Let $A$ be a $\k$-scheme with a map $A \to B$.
Let $\pi \from S \to B$ be as before.
Let $D$ be a effective relative Weil divisor on $S$.
Set $S_A = S \times_A B$.
The \emph{divisorial pullback} of $D$ to $A$, denoted by $D_{(A)}$, is the divisor given by the closure of $D|_U \times_B A$ in $S_A$.
Note that $D_{(A)}$ may not be equal to the subscheme $D \times_A B$ of $S_A$.
The divisorial pull-back of a non-effective relative divisor is defined by linearity.
Likewise, given a divisorial sheaf $F$ on $S$, its divisorial pull-back $F_{(A)}$ is defined by
\[ F_{(A)} = {i_A}_* i_A^* F,\]
where $i_A \from U \times_B A \to S_A$ is the open inclusion pulled back from $U \to S$.
Again, the divisorial pull-back $F_{(A)}$ may not be equal to the usual pullback $F_A = F \times_B A$.
To compare the two, observe that we always have a map
\begin{equation}\label{eqn:pullbacks}
  F_A \to F_{(A)}.
\end{equation}
This map is an isomorphism if $F_{A}$ is divisorial.
We say that $F$ \emph{commutes with base change} if for every $\k$-scheme $A$ with a map $A \to B$, the map in \eqref{eqn:pullbacks} is an isomorphism, or equivalently, the usual pullback $F_A$ is divisorial.
To check that $F$ commutes with base change, it suffices to check that it commutes with the base change for the inclusions of closed points into $S$ \cite[Lemma~8.7]{hac:01}.
Furthermore, if $F$ commutes with base change, then $F$ is flat over $B$ \cite[Lemma~8.6]{hac:01}.
Plainly, if $F$ is locally free, then it commutes with base change.
Furthermore, by Nakayama's lemma, it is easy to see that if $F$ commutes with base change, and $F_b$ is invertible for all $b \in B$, then $F$ is invertible.

Following \cite[Definition~2.14]{hac:01}, we make the following definition.
\begin{definition}[($\Q$-Gorenstein family)]\label{def:sls_family}
  Let $B$ be a $\k$-scheme.
  A \emph{$\Q$-Gorenstein family of log surfaces} over $B$ is a pair $(\pi \from S \to B, D \subset S)$ where $\pi$ is a flat Cohen--Macaulay morphism with geometric fibers of dimension 2 with slc singularities, and $D \subset S$ is a relative effective Weil divisor such that the following hold:
  \begin{enumerate}
  \item $\omega_{\pi}^{[i]}$ commutes with base change for every $i \in \Z$, and for every geometric point $b \to B$, there exists an $n$ such that $\omega_{S_b}^{[n]}$ is invertible;
  \item $\O_S(D)^{[i]}$ commutes with base change for every $i \in \Z$.
  \end{enumerate}
  A $\Q$-Gorenstein family of \emph{stable} log surfaces is a family as above with $\pi$ proper where all geometric fibers are stable log surfaces.
\end{definition}
By \cite[Lemma~8.19]{hac:01}, if $\O_S(-D)$ commutes with base change, then for every $A \to B$, the divisor $D$ is flat over $B$ and the divisorial pullback $D_{(A)}$ agrees with the usual pullback $D_A = D \times_B A$.
In particular, the two possible notions of the fiber of $(S, D)$ over $b \in B$ agree.

\subsection{The canonical covering stack and the index condition}
The analogue of \autoref{def:sls_family} without the divisor is called a \emph{Koll\'ar family}.
Explicitly, a \emph{Koll\'ar family} of surfaces is a flat, Cohen--Macaulay morphism $\pi \from S \to B$ with slc fibers satisfying the following conditions
\begin{enumerate}
\item $\omega_{\pi}^{[i]}$ commutes with arbitrary base change for all $i \in \Z$;
\item for every geometric point $b \to B$, there exists an $n$ such that $\omega_{S_b}^{[n]}$ is invertible.
\end{enumerate}

Let $\pi \from S \to B$ be a Koll\'ar family of surfaces.
The \emph{canonical covering stack} of $S/B$ is the stack
\[ \mathcal S = \left[\spec \left(\bigoplus_{n \in \Z} \omega_{\pi}^{[n]}\right) \big / \Gm\right],\]
where the $\Gm$ action is given by the grading.
By construction, $\mathcal S \to B$ is flat and Gorenstein.
Furthermore, by \cite[Theorem~5.3.6]{abr.has:11}, the natural map $p \from \mathcal S \to S$ is the coarse space map; it is an isomorphism over the locus where $\omega_{\pi}$ is invertible; and we have $p_* \omega^n_{\mathcal S/B} = \omega_{S/B}^{[n]}$.
Furthermore, if $U \subset S$ is an open subset such that $\omega_{\pi}^{[N]}\big|_U$ is invertible, then we have
\[ \mathcal S \times_S U \cong \left[\spec \left(\bigoplus_{n = 0}^{N-1} \omega_{\pi}^{[n]}\big|_U\right) \big / \mu_N \right].\]
Thus, $\mathcal S$ is a cyclotomic Deligne--Mumford stack in the language of \cite{abr.has:11}.

The canonical covering stack provides a convenient conceptual and technical framework to deal with the Koll\'ar condition that $\omega_{\pi}^{[i]}$ commute with base change.
It becomes very convenient if it \emph{also} takes care of the second condition in \autoref{def:sls_family}.
This motivates the following discussion.

Let $(S, D)$ be a stable log surface over $\k$.
Let $\mathcal S \to S$ be the canonical covering stack and $\mathcal D \subset S$ the divisorial pullback of $D$, namely the divisor obtained by taking the closure of $D|_U \times_S \mathcal S$ where $U \subset S$ is an open subset with complement of codimension at least 2 on which $D$ is Cartier.
\begin{definition}[(Index condition)]
  \label{def:index}
  We say that a stable log surface $(S, D)$ satisfies the \emph{index condition} if $\mathcal D \subset \mathcal S$ is a Cartier divisor.
\end{definition}
The reason for the term ``index condition'' is as follows.
Let $s \in S$ be a point.
The \emph{index} of $S$ at $s$ is the smallest positive integer $N$ such that $\omega_S^{[N]}$ is invertible at $s$.
Likewise, the \emph{index} of $D$ at $s$ is the smallest positive integer $M$ such that $\O_S(D)^{[M]}$ is invertible at $s$.
The linear equivalence $nK_S + mD \sim 0$ implies that we have an isomorphism
\[\omega_{\mathcal S}^{-n} \cong \O_{\mathcal S}(\mathcal D)^{[m]}.\]
The condition in \autoref{def:index} holds if and only if $\gcd(m, M) = 1$, as $\mathcal D$ Cartier at $p$ on $\mathcal S$ implies that the index of $D$ divides the index of $\omega_S$.
Thus, \autoref{def:index} is a condition on the index of $D$.

\subsection{The moduli stack}\label{sec:modstack}
Let $\mathfrak F$ be the category fibered in groupoids over the category of $\k$-schemes whose objects over $B$ are $\Q$-Gorenstein families of stable log surfaces over $B$ such that all geometric fibers satisfy the index condition.
Morphisms in $\mathfrak F$ are isomorphisms over $B$.

\begin{theorem}[(Existence of the moduli stack)]
  \label{thm:Fstack}
  $\mathfrak F$ is a Deligne--Mumford stack, locally of finite type over $\k$.
\end{theorem}
Before we prove the theorem, we recast $\mathfrak F$ in a more amenable form.
Let $\mathfrak G$ be the category fibered in groupoids over the category of $\k$-schemes whose objects over $B$ are pairs $(\pi \from S \to B, \mathcal D \subset \mathcal S)$, where
\begin{enumerate}
\item $\pi$ is a flat, proper, Koll\'ar family of surfaces,
\item $\mathcal S \to S$ is the canonical covering stack,
\item $\mathcal D \subset \mathcal S$ is an effective Cartier divisor flat over $B$,
\end{enumerate}
such that, for every geometric point $b \to B$, the pair $(S, D)$ is a stable log surface, where $D$ is the coarse space of $\mathcal D$.

\begin{proposition}\label{prop:GF}
  The categories $\mathfrak G$ and $\mathfrak F$ are equivalent as fibered categories over the category of $\k$-schemes.
\end{proposition}
\begin{proof}
  We have a natural transformation $\mathfrak G \to \mathfrak F$, defined as follows.
  Consider an object $(\pi \from S \to B, \mathcal D \subset \mathcal S)$ of $\mathfrak G$ over $B$.
  Let $D$ be the coarse space of $\mathcal D$.
  Using that $\mathcal D$ is a Cartier divisor and that $\mathcal S \to S$, we can check that $\O_S(D)^{[n]}$ commutes with base change for all $n \in \Z$ (see \cite[Theorem~5.3.6]{abr.has:11}).
  Therefore, $(\pi \from S \to B, D \subset S)$ is an object of $\mathfrak F$ over $B$.
  
  We now show that the transformation $\mathfrak G \to \mathfrak F$ defined above is an isomorphism.
  To do so, let us construct an inverse.
  Let $(\pi \from S \to B, D \subset S)$ be an object of $\mathfrak F$ over $B$.
  Let $\mathcal S \to S$ be the canonical covering stack, and $\mathcal D \subset \mathcal S$ the divisorial pullback.
  Since $D \subset S$ is a $\Q$-Cartier divisor, so is $\mathcal D \subset \mathcal S$.
  Furthermore, by the index condition, for every geometric point $b \to B$, the divisor $\mathcal D_{(b)}$ is Cartier.
  By \cite[Lemma~8.25]{hac:01}, it follows that $\mathcal D$ is Cartier.
  Thus, $(\pi \from S \to B, \mathcal D \subset \mathcal S)$ is an object of $\mathfrak G$ over $B$.
  This transformation provides the required inverse.
\end{proof}

\begin{remark}\label{rem:cohomology}
  Let $(S,D)$ be a stable log surface.
  Then $-K_S$ is ample, so $h^0(K_S) = h^2(\O_S) = 0$.
  Since $\chi(\O_S) = 1$ and $h^0(\O_S) = 1$, we also have $h^1(\O_S) = 0$.
  Thus, $h^i(\O_S) = 0$ for all $ i > 0$.
\end{remark}

\begin{proof}[Proof of \autoref{thm:Fstack}]
  By \autoref{prop:GF}, we may work with $\mathfrak G$ instead of $\mathfrak F$.
  We first show that $\mathfrak G$ is an algebraic stack, locally of finite type.

  Let $\mathrm{Orb}^{\lambda}$ be the moduli of polarized orbispaces defined in \cite[Section~3]{abr.has:11} (called $Sta^{\lambda}$ in loc. cit.).
  We have a map $\mathfrak G \to \mathrm{Orb}^{\lambda}$ given by
  \[ (S \to B, \mathcal D \subset \mathcal S) \mapsto (\mathcal S \to B, \omega_{\mathcal S \to B}^{-1}).\]
  Since $\mathrm{Orb}^{\lambda}$ is an algebraic stack locally of finite type by \cite[Theorem 3.1.4]{abr.has:11}, it suffices to show that for every scheme $B$ with a map $\phi \from B \to \mathrm{Orb}^{\lambda}$, the fiber product $\mathfrak G \times_\phi B$ is an algebraic stack.
  
  Let $B$ be a scheme with a map $\phi \from B \to \mathrm{Orb}^{\lambda}$ corresponding to a family of polarized orbispaces $(\pi \from \mathcal S \to B, \lambda)$.
  After passing to an \'etale cover, we may assume that the polarization $\lambda$ comes from a line bundle $\mathcal L$ on $\mathcal S$.
  Let $\mathfrak H \to B$ be the Hilbert stack of $\pi$.
  This is the stack whose objects over a $B$-scheme $A$ are substacks $\mathcal D \subset \mathcal S_A$ flat over $A$.
  By \cite[Theorem~1.1]{ols.sta:03}, $\mathfrak H \to B$ is an algebraic space locally of finite type.
  We show that $\mathfrak G \times_\phi B$ is isomorphic to a locally closed substack of $\mathfrak H$.

  There exists an open substack $U \subset \mathfrak H$ with the property that a map $A \to \mathfrak H$ given by $(\pi \from \mathcal S_A \to A, \mathcal D \subset \mathcal S_A)$ factors through $U$ if and only if 
  \begin{enumerate}
  \item $\mathcal D \subset \mathcal S_A$ is a Cartier divisor (its ideal sheaf is invertible);
  \item $\pi$ is Gorenstein;
  \item we have $\chi(\O_{S_a}) = 1$, and for every geometric point $a \to A$, there exists an $\epsilon > 0$ such that $(S_a, (m/n + \epsilon) \cdot D_a)$ is semi log canonical, where $(S_a, D_a)$ is the coarse space of $(\mathcal S_a, \mathcal D_a)$.
  \item the locus of points in $\mathcal S_a$ with non-trivial automorphism groups has codimension at least 2.
  \end{enumerate}
  The openness of the first condition follows by Nakayama's lemma.
  The openness of the Gorenstein condition follows from \cite[Lemma~4.4.1]{abr.has:11}, and of the semi log canonical property from \cite[A.1.1]{abr.has:11}.
  The openness of the last property follows from semi-continuity of fiber dimensions in the inertia stack $I \mathcal S \to B$.
  
  There exists a closed substack $V \subset U$ with the property that a map $B \to U$ factors through $V$ if and only if, in addition to the conditions above, we have 
  \begin{enumerate}
    \setcounter{enumi}{4}
  \item for every geometric point $b \to B$, the line bundles $\mathcal L_b \otimes \mathcal \omega_{\mathcal S_b}$ and $\O_{\mathcal S_b}(\mathcal D_b)^m \otimes \omega_{\mathcal S_b}^{n}$ are trivial.
  \end{enumerate}
  Since $h^0(\O_{\mathcal S_b}) = 1$ and $h^i(\O_{\mathcal S_b}) = 0$  for all $i > 0$, this condition is equivalent to saying that the line bundles $\mathcal L \otimes \mathcal \omega_\pi$ and $\O_{\mathcal S}(\mathcal D)^m \otimes \omega_\pi^n$ are pull-backs of line bundles from $B$.
  That this is a closed condition follows from \cite[III.10]{mum:08}.

  It is now easy to see that $\mathfrak G \times_\phi B$ is isomorphic to $V$.
    
  Since the automorphism group of a stable pair is finite \cite[Theorem~11.12]{iit:84}, the stack $\mathfrak G$ is Deligne--Mumford.
\end{proof}

It is not clear that $\mathfrak F$ is of finite type for two reasons.
Firstly, we have not put any numerical conditions on $(S, D)$.
Secondly, and more seriously, there is no \emph{a priori} lower bound on the $\epsilon$ in \autoref{def:sls}.
The problem goes away if we define away these two reasons.

Fix an $\epsilon > 0$ and a positive rational number $N$. 
Denote by $\mathfrak F_{\epsilon, N}$ the open substack of $\mathfrak F$ that parametrizes stable log surfaces that satisfy the definitions of \autoref{def:sls} with the given $\epsilon$ and have $K_S^2 \leq N$.
\begin{proposition}\label{prop:fintype}
  $\mathfrak F_{\epsilon, N}$ is an open substack of $\mathfrak F$ of finite type.
  If it is proper, then the coarse moduli space is projective.
\end{proposition}
\begin{proof}
  Note that $\mathfrak F_{\epsilon, N}$ is an open substack of $\mathfrak F$, and hence locally of finite type.
  The fact that it is bounded (admits a surjective morphism from a scheme of finite type) follows from \cite[\S~7]{ale:94} (and \cite[Theorem 1.1]{hac.mck.xu:18} in higher dimensional cases).
  Assuming properness, the projectivity of the coarse space follows from \cite[\S~4]{ale:96}.
\end{proof}

Deferring the considerations of finite type, we turn to the valuative criteria for separatedness and properness for $\mathfrak F$.
To do so, we must understand $\Q$-Gorenstein families of stable log surfaces over DVRs.
The following lemma gives a useful characterization of such families.

Let $\Delta$ be the spectrum of a DVR with generic point $\eta$ and special point $0$.
Let $\pi \from S \to \Delta$ be a flat, Cohen--Macaulay morphism with reduced geometric fibers of dimension 2 with slc singularities and $D \subset S$ a relative effective Weil divisor.
\begin{lemma}\label{lem:QgorDVR}
  In the setup above, assume that $S_\eta$ has canonical singularities and $\left(S_0, D_{(0)}\right)$ satisfies the index condition.
  Then $\pi \from (S, D) \to \Delta$ is a $\Q$-Gorenstein family of log surfaces if and only if both $K_{S/\Delta}$ and $D$ are $\Q$-Cartier.
\end{lemma}
\begin{proof}
  See \cite[Proposition~11.7]{hac:01}.
  The proof goes through verbatim.
\end{proof}

\begin{proposition}[(Valuative criterion of separatedness)]
  \label{lem:valsep}
  Let $(S_i, D_i) \to \Delta$ for $i = 1, 2$ be $\Q$-Gorenstein families of stable log surfaces satisfying the index condition.
  Suppose the geometric generic fibers of $S_i \to \Delta$ is isomorphic to $\PP^1 \times \PP^1$ for $i = 1, 2$.
  Then an isomorphism between $(S_i, D_i)$ over the generic fiber extends to an isomorphism over $\Delta$.  
\end{proposition}
\begin{proof}
  The proof is analogous to the proof of \cite[Theorem~2.24]{hac:01}.
  We recall the salient points.

  Possibly after a base change, there exists a common semistable log resolution $(\widetilde S, \widetilde D)$ of $(S_i, D_i)$ for $i = 1, 2$ that is an isomorphism over the generic fiber.
  Recall that a semistable log resolution is a projective morphism $\widetilde S \to S_i$ with the following properties:
  \begin{enumerate}
  \item $\widetilde S$ is non-singular;
  \item the exceptional locus of $\widetilde S \to S_i$ is a divisor;
  \item the central fiber $\widetilde S_0$ of $\widetilde S \to \Delta$ is reduced;
  \item the sum of $\widetilde S_0$, the proper transform of $D_i$, and the 
exceptional divisors dominating $T$ is a simple normal crossings divisor.
\end{enumerate}

The isomorphism between $(S_i, D_i)$ over the generic fiber implies that the proper transforms of $D_i$ are equal for $i = 1, 2$; call this proper transform $\widetilde D$.
  Let $\epsilon > 0$ be such that the central fibers of $(S_i, D_i) \to \Delta$ satisfy \autoref{def:sls} with this $\epsilon$.
  Then $i = 1, 2$, the pair $(S_i, D_i)$ is the $K_{\widetilde S} + \widetilde S_0 + (m/n + \epsilon) \cdot \widetilde D$ canonical model of $(\widetilde S, \widetilde D)$.
  The uniqueness of the log canonical model implies that the isomorphism between $(S_i, D_i)$ over the generic fiber extends over $\Delta$.
\end{proof}

From general principles, we get the following result that partially verifies the valuative criterion of properness for $\mathfrak F$.
\begin{proposition}[(A partial valuative criterion of properness)]
  \label{lem:valprop}
  Let $\Delta$ be a DVR with generic point $\eta$.
  Let $(S_\eta, D_\eta) \to \eta$ be a log surface with $S_\eta \cong (\PP^1 \times \PP^1) _\eta$ and $D_\eta \subset S_\eta$ a smooth curve of bi-degree $(\frac{2n}{m}, \frac{2n}{m})$.
  Possibly after a base change, there exists a (flat, proper) extension $(S, D) \to \Delta$ of $(S_\eta, D_\eta) \to \eta$ such that the central fiber $(S_0, D_{(0)})$ is a stable log surface and both $K_{S/\Delta}$ and $D$ are $\Q$-Cartier.
\end{proposition}
The key missing ingredient in \autoref{lem:valprop} is the assertion that $(S_0, D_{(0)})$ satisfies the index condition, and as a result (thanks to \autoref{lem:QgorDVR}) that $(S, D) \to \Delta$ is a $\Q$-Gorenstein family.
We do not know an \emph{a priori} reason for the index condition to hold.
In the work of Hacking and the present paper, a separate analysis is needed to confirm that it holds in cases of interest.

In subsequent sections, we develop methods to construct $(S, D)$ that yield an explicit description of $(S_0, D_{(0)})$ (see \autoref{prop:stab} and \autoref{subsec:mainthm}) for stable log quadric surfaces (defined in \autoref{sec:slq}).
Thus, for stable log quadrics, \autoref{prop:stab} subsumes \autoref{lem:valprop} and also verifies the index condition.
Nevertheless, we outline the proof of \autoref{lem:valprop} in general, following the proofs of \cite[Theorem~2.6]{hac:04} and \cite[Theorem~2.12]{hac:04}.
\begin{proof}[Outline of the proof of \autoref{lem:valprop}]
  First, complete $(S_\eta, D_\eta)$ to a flat family $(\PP^1 \times \PP^1 \times \Delta, D)$ over $\Delta$.
  Possibly after a base change on $\Delta$, take a semistable log resolution $(\widetilde S, \widetilde D) \to (\PP^1 \times \PP^1 \times \Delta, D)$.
  Run a $K_{\widetilde S} + (m/n) \widetilde D$ MMP on $(\widetilde S, \widetilde D)$ over $\Delta$, resulting in $(S_1, D_1)$.
  Then run a $K_{X_1}$ MMP on $(S_1, D_1)$ over $\Delta$, resulting in $(S_2, D_2)$.
  One can show that $(S_2, D_2) \to \Delta$ is a family of semistable log surfaces extending the original family where both $K_{S_2}$ and $D_2$ are $\Q$-Cartier and $nK_{S_2} + m D_2 \sim 0$.
  We note one difference at this step from \cite[Theorem~2.6]{hac:04}.
  Since the Picard rank of our generic fiber may not be 1 (unlike the case in \cite{hac:04}), the central fiber of $(S_2, D_2) \to \Delta$ may not be irreducible.

  Second, take a maximal crepant blowup $(S_3, D_3) \to (S_2, D_2)$,
  namely a partial semistable resolution such that the $nK_{S_3} + mD_3$ is the divisorial pullback of $nK_{S_2}+ mD_2$, and hence linearly equivalent to $0$, and $(S_3, S_3|_0 + (m/n+\epsilon)D_3)$ is dlt for small enough $\epsilon > 0$.
  Let $(S, D)$ be the $K_{S_3} + (m/n + \epsilon) D_3$ canonical model of $(S_3, D_3)$.
  Then $(S, D)$ is the required extension.  
\end{proof}

\subsection{Stable log quadrics}\label{sec:slq}

Henceforth, we fix $(m, n) = (2, 3)$.
Let $\mathfrak F_{K^2 = 8}$ be the open and closed substack of $\mathfrak F$ parametrizing stable log surfaces $(S, D)$ with $K_S^2 = 8$.
If $S$ is smooth, then it is a del Pezzo surface with $K_S^2 = 8$ and $D$ is a divisor such that $3K_S + 2D \sim 0$.
In particular, $K_S$ is an even divisor class, and hence $S$ is isomorphic to $\PP^1 \times \PP^1$ (a smooth quadric in $\PP^3$) and $D$ is a divisor of bi-degree $(3,3)$ on $S$.
Let $\mathfrak U \subset \mathfrak F_{K^2 = 8}$ be the open substack that parametrizes stable log surfaces $(S, D)$ with $S$ and $D$ smooth.
It is easy to see that $\mathfrak U$ is a smooth and irreducible stack of finite type.
Indeed, let $U \subset \PP H^0(\PP^1 \times \PP^1, \O(3,3))$ be the open subset of the linear series of $(3,3)$ curves on $\PP^1 \times \PP^1$ parametrizing $D \subset S$ such that $D$ is smooth.
Then $\mathfrak U$ is the quotient stack $\left[U / \Aut(\PP^1 \times \PP^1)\right]$.
\begin{definition}[(Stable log quadric)]
  We set $\mathfrak X$ as the closure of $\mathfrak U$ in $\mathfrak F_{K^2 = 8}$, with the reduced stack structure.
  We call the points of $\mathfrak X$ \emph{stable log quadrics}.
\end{definition}
Equivalently, a stable log quadric over $\k$ is a pair $(S, D)$ (satisfying the index condition) such that there exists a DVR $\Delta$ and a $\Q$-Gorenstein family of stable log surfaces (in the sense of \autoref{def:sls_family}) whose geometric generic fiber is isomorphic to $(\PP^1 \times \PP^1, \mathcal D)$, where $\mathcal D \subset \PP^1 \times \PP^1$ is a smooth curve of bi-degree $(3,3)$, and whose central fiber is isomorphic to $(S, D)$.
By the end of \autoref{sec:stable-replacement}, we obtain an explicit description of the stable log quadrics.
Using this description, we will also see that $\mathfrak X \subset \mathfrak F_{\epsilon, K^2 = 8}$ for a particular $\epsilon$, and hence it is of finite type.

%%% Local Variables:
%%% mode: latex
%%% TeX-master: "main"
%%% End:

\section{Trigonal curves and stable log surfaces}\label{sec:trigonal-surface}
% How to associating a log stable surface to a trigonal curve.
% Includes the discussion of the non-Gorenstein case and a list of unstable pairs.

The goal of this section is to describe the Tschirnhausen construction, which constructs a semi log canonical surface pair from a degree 3 covering of curves.
We use \cite{cas.eke:96} as a general reference for the structure of finite flat morphisms, particularly of degree 3.

Let $X$ and $Y$ be schemes and $\phi \from X \to Y$ a finite flat morphism of degree 3.
Let $E = E_\phi$ be the Tschirnhausen bundle of $\phi$.
We recall its construction from \cite[\S~1]{cas.eke:96}.
It is the vector bundle on $Y$ defined by the exact sequence
\begin{equation}\label{eqn:defE}
  0 \to \O_Y \to \phi_* \O_X \to E^\vee \to 0.
\end{equation}
We can associate to $\phi$ a Cartier divisor $D(\phi) \subset \PP E$ whose associated line bundle is $\O_{\PP E}(3) \otimes \det E^\vee$.
If $\phi$ is Gorenstein, then $D(\phi)$ is defined as follows.
The dual of the quotient map in \eqref{eqn:defE} is a map $E \to \phi_* \omega_{X/Y}$, or equivalently a map $\phi^* E \to \omega_{X/Y}$.
This map yields an embedding $X \to \PP E$ \cite[Theorem~1.3]{cas.eke:96}.
The divisor $D(\phi)$ is the image of $X$ under this embedding.
The construction of $D(\phi)$ extends by continuity to the case where $\phi$ is not Gorenstein \cite[\S~4.1]{deo:13*1}.
If $p \in Y$ is a point over which $\phi$ is not Gorenstein, then $D(\phi)$ contains the entire fiber of $\PP E \to Y$ over $p$.
The construction $\phi \leadsto D(\phi)$ is compatible with arbitrary base-change.
Furthermore, it extends to the case where $\phi \from X \to Y$ is a representable finite flat morphism of degree 3 between algebraic stacks.

Let $Y$ be a reduced stacky curve, and let $\phi \from X \to Y$ be a representable finite flat morphism of degree 3, \'etale over the generic points and the singular points of $Y$.
Write
\[ D(\phi) = D_H + \pi^* Z,\]
where $D_H$ is finite over $Y$ and $Z \subset Y$ is a divisor.
Note that $Z \subset Y$ is supported on the non-Gorenstein locus of $\phi$, and in particular on the smooth locus of $Y$.
As we have $X \cong D_H$ over $Y \setminus Z$, we see that $D_H$ is reduced.
Let $\phi_H \from D_H \to Y$ be the natural projection.
\begin{proposition}\label{prop:compare_br}
  We have the equality $\br \phi = \br \phi_H + 4Z$.
\end{proposition}
\begin{proof}
  It suffices to check the equality of divisors \'etale locally at a point $y \in Y$. Therefore, we may assume that $Y$ is a scheme.
  Choose a trivialization $\langle S, T \rangle$ of $E$ around $y$.
  We can write $D(\phi)$ as the vanishing locus of a homogeneous cubic
  \[ f = a S^3 + b S^2T + c ST^2 + d T^3,\]
  where $a, b, c, d \in \O_{Y,y}$.
  The discriminant divisor $\br \phi$ is cut out by the function
  \[ \Delta(f) = b^2c^2 - 4ac^3 - 4b^3d - 27a^2d^2+18abcd.\]
  Let $t$ be a uniformizer of $Y$ at $y$ and let $t^n$ be the highest power of $t$ that divides $a, b, c$ and $d$.
  Then $Z$ is the zero locus of $t^n$ and $D_H$ of the cubic $f_H = f/t^n$.
  We see that $\Delta(f) = \Delta(f_H) \cdot t^{4n}$, and hence $\br \phi = \br \phi_H + 4Z$.
\end{proof}

Let $P$ be an orbi-nodal curve and let $\phi \from C \to P$ be an admissible triple cover.
Let $S$ be the coarse-space of the surface $\PP E_\phi$ and $D$ the coarse space of the divisor $D(\phi) \subset \PP E$.
\begin{proposition}\label{prop:slc}
  Suppose $\mult_p \br \phi \leq 5$ for all $p \in P$.
  Then the pair $(S, c D)$ is slc for all $c \leq 7/10$.
\end{proposition}
\begin{proof}
  Locally, the pair $(S, D)$ is obtained from the pair $(\PP E, D(\phi))$ by taking the quotient by a finite group.
  Since the property of being slc is preserved under finite group quotients, it suffices to show that $(\PP E, c D(\phi))$ is slc.
  
  We first check the slc condition at the singular points of $\PP E$.
  Since $\PP E \to P$ is a $\PP^1$ bundle, the singular locus of $\PP E$ is the pre-image of the singular locus of $P$.
  Let $s \in D(\phi) \subset \PP E$ lie over a node $p \in P$.
  Since $C \to P$ is \'etale over $p$, \'etale locally near $s$ the pair $(\PP E, D(\phi))$ has the form
  \begin{equation}\label{eqn:stacklocal}
    (\spec \k[x,y,t]/(xy), t = 0 ).
  \end{equation}
  We see that $(\PP E, D(\phi))$ is slc at $p$.

  We now check the slc condition at the smooth points of $\PP E$.
  Let $s \in D(\phi) \subset \PP E$ lie over a smooth point $p \in P$.
  Choose a local coordinate $t$ on $P$ at $p$ and coordinates $(y,t)$ on $\PP E$ at $s$.
  Recall that we have the decomposition $D(\phi) = D_H + \pi^* Z$.
  Since $\mult_p \br \phi \leq 5$, \autoref{prop:compare_br} implies that $\mult_p Z \leq 1$.

  First, suppose $\mult_p Z = 1$.
  Then $\mult_p \br \phi_H \leq 1$; that is, $D_H$ is smooth at $s$ and $D_H \to P$ has at most a simple ramification point at $s$.
  In other words, $D$ has the local equation $ty = 0$, which has log canonical threshold 1, or $t(y^2-t) = 0$, which has log canonical threshold $3/4$; both $1$ and $3/4$ are bigger than $7/10$.

  Next, suppose  $\mult_p Z = 0$.
  Then $D = D(\phi)$ is flat over $p$.
  Let $D^\nu \to D$ be the normalization and $\delta = \length (\O_{D^\nu}/\O_D)$ the delta invariant.
  It is easy to check (\cite[Remark~7.4]{deo:13}) that 
  \[ \mult_p \br (D \to P) = \mult_p \br (D^\nu \to P) + 2 \delta.\]
  Since $\mult_p \br (D \to P) \leq 5$, we get $\delta \leq 2$.
  Hence the only possible singularities of $D$ are the $A_n$ singularities for $n \leq 4$.
  We conclude that the log canonical threshold of $D$ is at most $7/10$, achieved for an $A_4$ singularity, namely for a $D$ whose equation over $P$ locally over $p$ is $(y^2 - x^5)(y-1) = 0$.
\end{proof}

\begin{remark}\label{rem:sing}
  We record the observation from the proof of \autoref{prop:slc} that the only possible singularities of $D$ are $A_n$ for $n \leq 4$.
\end{remark}

\begin{remark}\label{rem:Ssing}
  We also record that $(S, D)$ satisfies the index condition.
  To see this, it suffices to check the singular points of $S$ contained in $C$.
  If $p \in S$ is a singular point contained in $C$, then $p$ is the image of $(0,0,0)$ coming from the \'etale local neighborhood in \eqref{eqn:stacklocal}.
  In particular, $\PP E$ is the canonical covering stack of $S$ at $p$, implying that $(S,D)$ satisfies the index condition as $D(\phi) \subset \PP E$ is Cartier as in \eqref{eqn:stacklocal}.
\end{remark}

Let $g \geq 4$.
Denote by $\H^3_g$ the Hurwitz space of genus $g$ triple covers of the projective line.
More precisely, it is the moduli space whose $S$-points are given by
\[(P \to S, C \to S, \phi \from C \to P),\]
where $P \to S$ is a conic bundle, $C \to S$ is a smooth, proper, and connected curve of genus $g$, and $\phi$ is a finite flat morphism of degree 3 with simple branching.
Let $\epsilon$ be a positive number less than $1/30$.
Consider the compactification $\overline \H^3_g(1/6 + \epsilon)$ of $\H^3_g$ constructed in \cite{deo:13}.
We quickly recall its definition and salient properties.
The $S$-points of $\overline \H^3_g(1/6 + \epsilon)$ are given by
\[(P \to S, C \to S, \phi \from C \to P),\]
where $P \to S$ is a (balanced) orbi-nodal curve of genus $0$, $C \to S$ is a flat, proper, connected, orbi-nodal curve of genus $g$, and $\phi \from C \to P$ is a finite flat morphism satisfying three conditions: (1) $\phi$ is \'etale over the nodes of $P$, (2) the monodromy map $P \setminus \br \phi \to BS_3$ given by $\phi$ is representable, and (3) the coarse space of $P$ along with the divisor $\br \phi$ with the weight $(1/6+\epsilon)$ is a weighted stable curve over $S$ in the sense of Hassett \cite{has:03}.
In this definition, a (balanced) orbi-nodal curve over $S$ is a Deligne--Mumford stack whose coarse space is a nodal curve over $S$, and whose stack structure around a node is given \'etale locally over the coarse space by
\[ \left[\spec \O_S[x,y]/(xy-s) / \mu_n \right],\]
for some $s \in \O_S$, where the cyclic group $\mu_n$ acts by $\zeta \from (x,y) \mapsto (\zeta x, \zeta^{-1}y)$ (see \cite[2.1.2]{abr.cor.vis:03}).
The main results of \cite{deo:13} imply that $\overline \H^3_g(1/6 + \epsilon)$ is a proper and smooth Deligne--Mumford stack over $\k$ \cite[Theorem~5.5, Corollary~6.6]{deo:13}.

Consider a point $[\phi \from C \to P]$ of $\overline \H^3_g(1/6 + \epsilon)$.
Let $(S, D)$ be the pair associated to $C \to P$ by the Tschirnhausen construction.
We call $(S, D)$ the \emph{Tschirnhausen pair} of $\phi$.
\begin{proposition}\label{prop:unstable}
  The divisor $K_S + (2/3 + \epsilon) D$ is ample for all sufficiently small and positive $\epsilon$ except in the following cases.
  \begin{enumerate}
  \item\label{maroni} $P = \PP^1$, and $C$ is a Maroni special curve of genus $4$,
  \item\label{hyperell} $P = \PP^1$, and $C = \PP^1 \cup H$, where $H$ is a hyperelliptic curve of genus 4 attached nodally to $\PP^1$ at one point.
  \item There is a component $L \cong \PP^1$ of $P$ meeting $\overline{P \setminus L}$ in a unique point such that $C \times_P L$ is either
    \begin{enumerate}
    \item \label{ell_tail} a connected curve of arithmetic genus 1, or
    \item \label{g2_tail} a disjoint union of $L$ and a connected curve of arithmetic genus 2.
    \end{enumerate}
  \end{enumerate}
\end{proposition}
Recall  that a smooth curve $C$ of genus 4 is \emph{Maroni special} if it satisfies the following equivalent conditions (see, for example, \cite[Page 298]{gri.har:94}): $C$ is not hyperelliptic and lies on a singular quadric in its canonical embedding in $\PP^3$, (b) $C$ has a unique $g^1_3$, (c) there is a degree 3 map $\phi \from C \to \PP^1$ such that the Tschirnhausen bundle $\left(\phi_* \O_C/\O_{\PP^1}\right)^\vee$ is isomorphic to $\O(2) \oplus \O(4)$.
In contrast, a \emph{Maroni general} $C$ of genus 4 (a) is non-hyperelliptic and lies on on a smooth quadric in its canonical embedding in $\PP^3$, (b) has two distinct $g^1_3$'s, and (c) has Tschirnhausen bundle isomorphic to $\O(3) \oplus \O(3)$.

\begin{proof}[Proof of \autoref{prop:unstable}]
  The numerical criteria of ampleness may be checked on the stack, rather than the coarse space.
  Therefore, in the rest of the proof, let $S$ denote the stack $\PP E_\phi$ and $D(\phi) \subset S$ the Tschirnhausen divisor associated to $\phi$.
  As the coarse space map of $\PP E_\phi$ is unramified in codimension one, the divisor classes remain unchanged.

  It suffices to check ampleness on each irreducible component of $S$.
  Let $L$ be an irreducible component of $P$.
  Set $C_L = L \times_P C$, let $\phi_L \from C_L \to L$ be the restriction of $\phi$, and let $E_L$ be the Tschirnhausen bundle of $\phi_L$.
  Set $S_L = \PP E_L$ and $D_L = D \cap S_L$.
  Let $n = \deg E_L$, so that $2n = \deg \br \phi_L$.

  We know that the Neron-Severi group of $S_L$ is spanned by the class $F$ of a fiber and the class $\zeta$ of $\O_{\PP E}(1)$.
  The intersection form is determined by $F^2 = 0$, $\zeta F = 1$, and $\zeta^2 = n$.
  The cone of curves on $S_L$ is spanned by $F$ and the class of a section $\sigma$.  
  Let $m$ be the number of points in $L \cap \overline {(P \setminus L)}$ counted without any multiplicity.
  Then, it is easy to check that $\deg K_{P}|_L = -2 + m$.
  Therefore, we obtain that
  \[ K_S |_{S_L} \sim (m+n-2)F - 2 \zeta. \]
  We also have $D_L \sim 3 \zeta - nF$, and hence
  \[ K_S + (2/3 + \epsilon) D \big|_{S_L} \sim (m + n/3 - 2)F  + \epsilon (3 \zeta - nF).\]
  We see immediately that $\left(K_S + (2/3 + \epsilon) D \right) \cdot F = 3\epsilon > 0$.
  Thus, it remains to check that $\left(K_S + (2/3 + \epsilon) D \right) \cdot \sigma > 0$ for the extremal section $\sigma$.

  If $m + n/3 > 2$, then it is clear that $\left(K_S + (2/3 + \epsilon) D \right) \cdot \sigma > 0$ for small enough $\epsilon$.
  As a result, we only need to consider the cases where $m \leq 2$.
  If $m = 2$, then the ampleness of $K_P + (1/6 + \epsilon)\br\phi$ implies that $n > 0$, and hence $m + n/3 > 2$.

  It remains to consider $m = 0$ and $m = 1$.
  First, suppose $m = 0$.
  Then $n = g+2 \geq 6$, so $m + n/3 \geq 2$, with equality only if $g = 4$. %
  If $g = 4$, then $E$ is isomorphic to either $\O(3) \oplus \O(3)$, or $\O(2) \oplus \O(4)$, or $\O(1) \oplus \O(5)$.
  For $E \cong \O(3) \oplus \O(3)$, it is easy to check that $K_S + (2/3 + \epsilon)D$ is ample.
  The cases $E \cong \O(2) \oplus \O(4)$ and $E \cong \O(1) \oplus \O(5)$ yield the possibilities \eqref{maroni} and \eqref{hyperell}, respectively, in the statement of \autoref{prop:unstable}.

  Next, suppose $m = 1$.
  The ampleness of $K_P + (1/6 + \epsilon)\br\phi$ implies that $n = 3$.
  Let $p \in L$ be the unique point of intersection of $L$ with $\overline{P \setminus L}$.
  We know that vector bundles on $L$ split as direct sums of line bundles  \cite{mar.tha:12}.
  Also, line bundles on $L$ are classified by their degree, which is an element of $\frac{1}{d}\Z$, where $d$ is the order of $\Aut_pL$.
  Suppose
 $E_L \cong \O_L(a) \oplus \O_L(b)$,
  where $a, b \in \frac{1}{d}\Z$ with $0 \leq a \leq b$ and $a + b = n$.
  The extremal section $\sigma$ is given by $\sigma \sim \zeta - bF$.
  Since $C \to P$ is an admissible triple cover, $d$ is either 1, 2, or 3.
  If $d = 1$, then $(a, b)$ is either $(1,2)$ or $(0,3)$.
  These two cases yield the possibilities \eqref{ell_tail} and \eqref{g2_tail}, respectively, in the statement of \autoref{prop:unstable}.

  It remains to consider the cases $d = 2$ and $d = 3$.
  Consider the map $\underline{\phi_L} \from \underline {C}_L \to \underline L$ on coarse spaces associated to $\phi_L \from C_L \to L$.
  Since $d > 1$, we know that $C_L$ is not isomorphic to its coarse space $\underline {C}_L$, and hence $E_L$ is not pulled back from $\underline L$.
  Said differently, $a$ and $b$ are not both integers.
  We have $\deg \br \underline{\phi_L} = \deg \br \phi_L + (d-1) = 2n+d-1$, and $\deg \br \underline{\phi_L}$ must be even.
  So we cannot have $d = 2$.
  For $d=3$, observe that $\underline{C}_L$ must be totally ramified over $p$.
  We compute that
  \begin{align*}
    K_S + (2/3 + \epsilon)D\big|_{S_L} \cdot \sigma = \epsilon(3 \zeta - nF) \cdot \sigma = \epsilon (2a - b).
  \end{align*}
  Since $\underline{C}_L$ is triply ramified over $p$, it is locally irreducible over $p$.
  As a result, $D_L$ does not contain $\sigma$ as a component.
  We conclude that $D_L \cdot \sigma = 2a-b \geq 0$.
  The further constraints that $a + b =n$ and that not both $a$ and $b$ are integers force $2a -b > 0$.
  Therefore, we get that $K_S + (2/3 + \epsilon)D$ is ample.
\end{proof}

\subsection{Stable and unstable pairs in genus 4}\label{sec:g4list}
Let $g = 4$, and $[\phi \from C \to P] \in \overline \H^3_g(1/6 + \epsilon)$.
Let $(S, D)$ be the pair associated to $C \to P$ by the Tschirnhausen construction.
\begin{proposition}\label{prop:genus4}
  The pair $(S, D)$ is a semi-stable log quadric surface.
  It is also stable except in the cases enumerated in \autoref{prop:unstable}
\end{proposition}
\begin{proof}
  By \autoref{prop:slc}, $(S, 2/3 \cdot D)$ is slc.
  By \autoref{prop:unstable}, there exists $\epsilon > 0$ such that $K_S + (2/3 + \epsilon) D$ is ample, except in the listed cases.
  It remains to show that $3K_S + 2D$ is linearly equivalent to 0.
  It suffices to show this on the stack $\PP E_\phi$.
  We have
  \begin{equation}\label{eq:Kpp}
    K_{\PP E_\phi} \cong \O(-2) \otimes \pi^*\det E_\phi \otimes \pi^*K_{P}
  \end{equation}
  where $\pi \from \PP E_\phi \to P$ is the natural projection.
  By construction, we have
  \begin{equation}\label{eq:od}
    \O(D) \cong \O(3) \otimes \pi^* \det E_\phi^\vee.
  \end{equation}
  Recall that the branch divisor of $\phi \from C \to P$ is defined by the discriminant, which is a section of $(\det E_\phi)^{\otimes 2}$ (see \cite[Tag 0BVH]{sta:20}).
  Hence $2 \det E_\phi$ is the class of the branch divisor $B$.
  Furthermore, we always have
  \begin{equation}\label{eq:kp}
    K_P + 1/6 \cdot B \sim 0.
  \end{equation}
  To check this, note that we either have $P \cong \PP^1$ or $P \cong P_1 \cup P_2$ with the 12 points of $B$ separated as 6+6 on the two components.
  In both cases, \eqref{eq:kp} holds.
  From \eqref{eq:Kpp}, \eqref{eq:od}, and \eqref{eq:kp}, we get that $3K_{\PP E_\phi} + 2 D \sim 0$.
\end{proof}

We enumerate the strictly semi-stable and stable cases for genus 4.
Recall that $\epsilon$ is such that $0 < \epsilon < 1/30$.
\begin{description}
\item [Stable pairs]
  A $(1/6 + \epsilon)$-admissible cover $\phi \from C \to P$ yields a stable log quadric surface $(S, D)$ in the following cases.
  \begin{enumerate}
  \item\label{f0}
    $P \cong \PP^1$ and $\phi \from C \rightarrow P$ is Maroni general in the sense that $E_\phi \cong \O(3) \oplus \O(3)$.
    In this case, we see that $S \cong \PP^1 \times \PP^1$ and $D \subset S$ is a divisor of bi-degree $(3,3)$.
  \item\label{f_1/3_1/3}
    $P=P_1 \cup_s P_2$ is a twisted curve with two smooth irreducible components $P_1$ and $P_2$ attached nodally at $s$.
    Both components are rational (their coarse spaces are $\PP^1$), and the only point with a non-trivial automorphism group on $P$ is the node $s$ with $\Aut_s P = \mu_3$.
    The curve $C$ is schematic, and is isomorphic to $C_1 \cup_p C_2$, where $p \in C$ is a node and $C_i$ has arithmetic genus 2 for each $i$.
    The map $\phi$ restricts to a degree 3 map $C_i \to P_i$, \'etale over $s$, and $p$ is the unique pre-image of $s$.
    In this case, we see that $S$ is the coarse space of a projective bundle $\PP(\O(5/3, 4/3) \oplus \O(4/3,5/3))$, where $\O(a_1,a_2)$ is a line bundle on $P$ whose restriction to $P_i$ is $\O(a_i)$.
    Let $S_1$ and $S_2$ be the two components of $S$ over coarse spaces of $P_1$ and $P_2$, respectively.
    Then $D \cap S_i \subset S_i$ is a divisor of class $3 \sigma_i + 2F$ where $\sigma_i \subset S_i$ is the image of the unique section of $\PP(\O(5/3)\oplus\O(4/3))$ of self-intersection $(-1/3)$.
    Furthermore, $D$ intersects the double curve $S_1 \cap S_2$ transversely.
  \end{enumerate}

\item[Unstable pairs]
  A $(1/6 + \epsilon)$-admissible cover $\phi \from C \to P$ yields a semi-stable but not stable log quadric surface $(S, D)$ in the following cases.
  \begin{enumerate}
  \item \label{maroni4} $P \cong \PP^1$, and $\phi \from C \to P$ is Maroni special.
    In this case, $S \cong \F_2$ and $D \subset S$ is a divisor of class $3 \sigma + 6 F$, where $\sigma \subset S$ is the directrix.
  \item \label{hyperell4} $P \cong \PP^1$ and $C \cong H \cup_p L$, where $L \cong \PP^1$, and $H$ is a curve of arithmetic genus 4 attached nodally to $L$ at one point $p$.
    The map $\phi$ restricts to a degree 2 map $H \to P$ and to a degree 1 map $L \to P$.
    In this case, $S \cong \F_4$ and $D$ is the union of $\sigma$ and a divisor of class $2 \sigma + 9F$.
  \item $P \cong P_1 \cup_s P_2$, with $P_i \cong \PP^1$ attached nodally at a point $s$;
    and $C \cong C_1 \cup_{p,q,r} C_2$, where $p,q,r \in C$ are nodes and $C_i$ has arithmetic genus one for each $i$.
    The map $\phi$ restricts to a degree 3 map $C_i \to P_i$, \'etale over $s$, and $\{p,q,r\}$ is the pre-image of $s$.
    These cases break into three further subcases.
    In all three subcases, we have $S = S_1 \cup S_2$ and $D = D_1 \cup D_2$.
    The subcases are:
    \begin{enumerate}
    \item\label{f33}
      For $i = 1, 2$, we have $C_i = H_i \sqcup L_i$, where $L_i \cong \PP^1$ and $H_i$ is a connected curve of genus 2.
      The map $\phi \from C_i \to P_i$ restricts to a degree 2 map on $H_i$ and to a degree 1 map on $L_i$. $L_1$ and $L_2$ do not intersect as $C$ is connected.
      In this case, we have $S_i \cong \F_3$; and $D_i = \sigma_i \sqcup H_i \subset S_i$, where $\sigma_i \subset S_i$ is the unique section of self-intersection $(-3)$ and $H_i \subset S_i$ is a divisor of class $2 \sigma_i + 6F$ intersecting the fiber $S_1 \cap S_2$ transversely.
    \item\label{f11} For $i = 1,2$, the curve $C_i$ is a connected curve of arithmetic genus 1.
      In this case, we have $S_i \cong \F_1$, and $D_i \subset S_i$ is a divisor of class $3 \sigma_i + 3F$ intersecting the fiber $S_1 \cap S_2$ transversely, where $\sigma_i$ is the unique section of self-intersection $(-1)$ in $S_i$.
    \item\label{f13} $C_1$, $S_1$, $D_1$ are as in case \eqref{f33} and $C_2$, $S_2$, $D_2$  are as in case \eqref{f11}.
    \end{enumerate}
  \end{enumerate}
\end{description}

%%% Local Variables:
%%% mode: latex
%%% TeX-master: "main"
%%% End:

\section{Flips}\label{sec:flips}
% Technical section that describes the major flips that we need.
The goal of this section is to describe two kinds of flips that are necessary for the stable reduction of log surfaces arising from trigonal curves.
The first involves flipping a $-4$ curve and the second a $-3$ curve on the central fiber in a family of surfaces.

\subsection{Flipping a $(-4)$ curve (Type I flip)}\label{sec:flip1}
Let $\Delta$ be the spectrum of a DVR.
Let $\mathcal X \to \Delta$ be a smooth, but not necessarily proper, family of surfaces.
Let $\mathcal D \subset \mathcal X$ an effective divisor flat over $\Delta$ with a non-singular general fiber.
Denote by $(X, D)$ the special fiber of $(\mathcal X, \mathcal D) \to \Delta$.
Suppose $(X, D)$ has the following form.
We have $D = \sigma \cup C$, where $\sigma \subset X$ is a smooth rational curve of self-intersection $-4$ and $C \subset X$ is a non-singular curve (not necessarily proper) that intersects $\sigma$ transversely at one point $p$.

The leftmost quadrilateral in \autoref{fig:type1flip} is our diagrammatic representation of $X$ along with the configuration of curves the $C$ and $\sigma$ on it.
In general, we represent surfaces by plane polygons, and depict curves lying on the surface along the edges or on the interior.
A number next to an edge, if any, is the self-intersection of the curve represented by the edge.
Text next to a point is the description of the singularity at that point.

Construct $(X', D')$ from $(X, D)$ as follows (see \autoref{fig:type1flip}).
Let $\widetilde X \to X$ be the blow up of $X$ two times, first at $p$ (the intersection point of $C$ and $\sigma$), and second at the intersection point of the exceptional divisor $E_1$ of the first blow-up with the proper transform of $C$.
Equivalently, $\widetilde X$ is the minimal resolution of the blow-up of $X$ at the unique subscheme of $C$ of length 2 supported at $p$.
Denote by $\widetilde C \subset \widetilde X$ and $\widetilde \sigma \subset \widetilde X$ the proper transforms of $\sigma$ and $C$, and by $E_{i} \subset X$ the proper transform of the exceptional divisor of the $i$th blow up, for $i = 1, 2$.
On $\widetilde X$, the curves $(E_1, \widetilde \sigma)$ form a chain of rational curves of self-intersections $(-2,-5)$.
Let $\widetilde X \to X'$ be the contraction of this chain.
Then the surface $X'$ is smooth everywhere except at the image point of the rational chain, where it has the quotient singularity $\frac{1}{9}(1,2)$.
Let $C' \subset X'$ be the image of $\widetilde C$.
Set $D' = C'$.

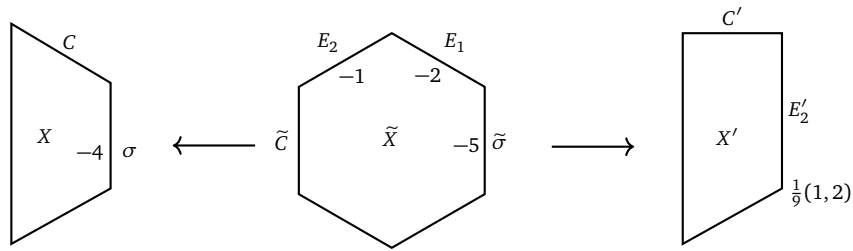
\begin{figure}[hb]
  \centering
  \begin{tikzpicture}[y=0.80pt, x=0.80pt, yscale=-1.000000, xscale=1.000000, inner sep=0pt, outer sep=0pt]
\begin{scope}[shift={(534.9029,-395.55249)}]
  \begin{scope}[cm={{1.37549,0.0,0.0,1.37388,(-6.9078,-804.35828)}}]
      \begin{scope}[cm={{0.72702,0.0,0.0,0.72787,(-359.36641,542.74565)}}]
      \end{scope}
      \begin{scope}[cm={{0.52625,0.0,0.0,0.52558,(-154.06407,500.39016)}}]
          \path[xscale=1.000,yscale=1.000,fill=black] (-398.2454,737.2799) node
            (text4137-2) {$C$};
          \path[xscale=1.000,yscale=1.000,fill=black] (-359.2418,808.1982) node
            (text4137-3-1) {$\sigma$};
          \begin{scope}[cm={{2.86589,0.0,0.0,2.89693,(-177.11425,-439.26347)}},draw=black]
            \path[xscale=1.106,yscale=0.904,draw=black,fill=black] (-65.6408,476.5813) node
              (text4193-0-5-6) {$-4$};
            \path[xscale=1.106,yscale=0.904,draw=black,fill=black] (-74.6635,472.5353) node
              (text4193-62-0-0-6) {$X$};
            \path[draw=black] (-67.9415,415.1732) -- (-67.9415,438.7927) --
              (-90.3255,451.2002) -- (-90.3255,401.9412) -- cycle;
            \path[draw=black,<-] (-53.9298,429.1793) -- (-35.3362,429.1793);
            \path[draw=black,->] (31.5905,429.4505) -- (50.1802,429.4505);
            \path[xscale=1.005,yscale=0.995,fill=black] (70.9618,429.3405) node (text4137-9)
              {$X'$};
            \path[xscale=1.106,yscale=0.904,draw=black,fill=black] (83.3252,486.5515) node
              (text4193-62-0-0-6-6-8) {$\frac19(1,2)$};
            \path[draw=black] (83.5167,404.0567) -- (83.5167,438.7927) -- (61.1326,451.2002)
              -- (61.1326,404.0567) -- cycle;
          \end{scope}
          \path[xscale=1.000,yscale=1.000,fill=black] (30.1635,720.1034) node
            (text4137-2-7-5) {$C'$};
          \path[cm={{0.79722,0.38455,-0.3836,0.79918,(186.5955,869.9338)}},draw=black]
            (-373.2097,49.5735) -- (-339.2150,120.2334) -- (-383.4108,185.0037) --
            (-461.6015,179.1140) -- (-495.5962,108.4542) -- (-451.4003,43.6839) -- cycle;
          \begin{scope}[cm={{0.0,-3.07046,2.53176,0.0,(-1593.2379,508.02141)}},draw=black]
            \path[cm={{0.0,1.09757,-0.9111,0.0,(0.0,0.0)}},draw=black,fill=black]
              (495.7994,89.8043) node (text4193-6) {$-1$};
            \path[cm={{0.0,1.09757,-0.9111,0.0,(0.0,0.0)}},draw=black,fill=black]
              (513.3980,89.3965) node (text4193-62-0) {$-2$};
          \end{scope}
          \path[xscale=1.100,yscale=0.909,draw=black,fill=black] (-128.8045,884.1939) node
            (text4193-0-5-6-7) {$-5$};
          \path[xscale=1.100,yscale=0.909,draw=black,fill=black] (-173.1562,878.1454) node
            (text4193-0-5-6-9) {$\widetilde X$};
          \path[xscale=1.000,yscale=1.000,fill=black] (-261.1903,800.7129) node
            (text4137-2-7) {$\widetilde C$};
          \path[xscale=1.000,yscale=1.000,fill=black] (-120.1984,801.4191) node
            (text4137-3-1-3) {$\widetilde\sigma$};
          \path[xscale=1.059,yscale=0.945,fill=black] (-218.3039,780.3044) node
            (text4137-3-1-3-3) {$E_2$};
          \path[xscale=1.000,yscale=1.000,fill=black] (-148.9543,737.4889) node
            (text4137-3-1-3-3-7) {$E_1$};
          \path[xscale=1.000,yscale=1.000,fill=black] (73.2954,781.4898) node
            (text4137-3-1-3-3-3) {$E_2'$};
      \end{scope}
  \end{scope}
\end{scope}

\end{tikzpicture}
  \caption{The central fiber $X$ is replaced by $X'$ in a type 1 flip.}
  \label{fig:type1flip}
\end{figure}

\begin{proposition}
\label{prop:type1flip}
  Let $(\mathcal X, \mathcal D) \to \Delta$ be a family of log surfaces as described above.
  Then there exists a flat family $(\mathcal X', \mathcal D') \to \Delta$ isomorphic to $(\mathcal X, \mathcal D)$ over $\Delta^\circ$ such that the central fiber of $(\mathcal X', \mathcal D') \to \Delta$ is $(X',D')$.
  Furthermore, the threefold $\mathcal X'$ is $\Q$-factorial and has canonical singularities.
\end{proposition}
\begin{remark}
\label{rmk:indep_of_family}
Note that $(X', D')$ is log canonical.
Also, it is important to observe that it depends only on $(X, D)$, not on the family $(\mathcal X, \mathcal D) \to \Delta$.  
\end{remark}

The rest of \autoref{sec:flip1} is devoted to the proof of \autoref{prop:type1flip}.
In the proof, we construct $\mathcal X'$ from $\mathcal X$ by an explicit sequence of birational transformations.
We divide these birational transformations into two stages.
The first stage consists of a sequence of blow-ups along $-4$ curves.
The second stage consists of a sequence of a particular kind of flip, which we call a \emph{topple}.
We begin by studying blow ups and topples.

\subsubsection{A $(-4)$-blow up}

Let $(\mathcal X, \mathcal D) \to \Delta$ be as in the statement of \autoref{prop:type1flip}.
Let $\beta \from \widetilde {\mathcal X} \to \mathcal X$ be the blow up along $\sigma$.
Let $\widetilde {\mathcal D}$ be the proper transform of $\mathcal D$ in $\widetilde {\mathcal X}$ and $E \subset \widetilde{\mathcal X}$ the exceptional divisor.
The central fiber of $\widetilde{\mathcal X} \to \Delta$ is the union of $E$ and the proper transform of $X$, which is an isomorphic copy of $X$.
We know that $E$ is the projectivization of the conormal bundle of $\sigma$ in $\mathcal X$.
The next lemma identifies the normal bundle.
\begin{lemma}
  \label{prop:normalSigma}
  The normal bundle $N_{\sigma/\mathcal X}$ is given by
  \[
    N_{\sigma/\mathcal X} \cong
    \begin{cases}
      \O(-1) \oplus \O(-3) &\text{if $\mathcal D$ is non-singular,} \\
      \O \oplus \O(-4) &\text{otherwise}.
    \end{cases}
  \]
  In the first case, we have $E \cong \F_2$, and $E \cap \widetilde {\mathcal D}$ is the unique $-2$ curve on $E$.
  In the second case, we have $E \cong \F_4$, and $E \cap \widetilde {\mathcal D}$ is the union of the unique $-4$ curve on $E$ and a fiber $F$ of $E \to \PP^1$.
\end{lemma}
\begin{proof}
  We have the exact sequence of bundles
  \[ 0 \to N_{\sigma/X} \to N_{\sigma/\mathcal X} \to N_{X/\mathcal X}\big|_\sigma \to 0.\]
  In this sequence, the kernel is $\O(-4)$ and the cokernel is $\O$.
  Therefore, the only possibilities for $N_{\sigma/\mathcal X}$ are $\O(-i) \oplus \O(-4+i)$ for $i = 0,1,2$.
  We must now rule out $i = 2$, and characterize the remaining two.

  We have $[\widetilde {\mathcal D}] = [\beta^* \mathcal D] - [E]$.
  Since $\widetilde {\mathcal D}$ intersects $E$ properly, the restriction $\widetilde{\mathcal D}|_E$ must be effective.
  Let us find the self-intersection of this divisor on $E$.
  Recall that $E = \PP N^{\vee}_{\sigma/\mathcal{X}}$.
  Denote $\zeta =[c_1(\mathcal{O}_{\PP N^{\vee}_{\sigma/\mathcal{X}}}(1))]$, and let $f$ be the class of the fiber of $E \rightarrow \sigma$.
  Then, $\zeta^2= \deg N_{\sigma/X}^{\vee} = 4$. Observe that $-[E|_{E}]=\zeta$, and $[(\beta^*\mathcal{D})|_E]=[\beta^*(\mathcal{D}|_{\sigma})]$ by the push-pull formula. As  $[\mathcal{D}|_{\sigma}]=[\mathcal{D}|_X]|_{\sigma}=(\sigma +[C])|_{\sigma}$, which is $-3$ times a point of $\sigma$, the class $[(\beta^*\mathcal{D})|_E]$ must be $-3f$. We conclude that $[\widetilde{\mathcal D}|_E]^2=-2$. 

  Note that $\PP\left(\O(-2)\oplus\O(-2)\right) = \PP^1 \times \PP^1$ contains no effective classes of self-intersection $-2$.
  Therefore, we can rule out the possibility of $i = 2$, namely the possibility that $N_{\sigma / \mathcal X} \cong \O(-2) \oplus \O(-2)$.

  For the remainder, we examine the map $\widetilde {\mathcal D} \to \mathcal D$, which is the blow-up along $\sigma$, and the curve $E \cap \widetilde {\mathcal D}$.
  Since the central fiber of $\mathcal D \to \Delta$ is a nodal curve with the node at $p$, the only possible singularity of $\mathcal D$ is at $p$.
  Hence, the curve $E \cap \widetilde {\mathcal D}$ contains a unique reduced component $\widetilde \sigma$ mapping isomorphically to $\sigma$, and possibly some other components that are contracted to $p$.
  As a divisor on $E$, we may write
  \[ E \cap \widetilde {\mathcal D} = s + m \cdot f,\]
  for some $m \geq 0$,  where $s$ is a section of $E \to \sigma$ and $f$ is the fiber of $E \to \sigma$ over $p$.
    
  Suppose $\mathcal D$ is non-singular.
  Then the blow-up $\widetilde {\mathcal D} \to \mathcal D$ is an isomorphism, and therefore we have $m = 0$.
  As a result, we see that $E \to \sigma$ has a section of self-intersection $(-2)$.
  We conclude that $N_{\sigma/\mathcal X} \cong \O(-1) \oplus \O(-3)$, and $E \cap \widetilde {\mathcal D}$ is the unique section of self-intersection $(-2)$.

  Suppose $\mathcal D$ is singular.
  Then it has an $A_n$-singularity at $p$ for some $n \geq 1$.
  In that case, $\widetilde D \to D$ contracts a $\PP^1$.
  Therefore, we must have $m > 0$.
  Since $\F_2$ does not contain a class of the form $s + m \cdot f$ of self-intersection $(-2)$, we can rule out $N_{\sigma / \mathcal X} \cong \O(-1) \oplus \O(-3)$, and get $N_{\sigma / \mathcal X} \cong \O \oplus \O(-4)$.
  The unique effective class of the form $s + m \cdot f$ on $E \cong \F_4$ is the union of the section of self-intersection $(-4)$ and a fiber.
\end{proof}

\subsubsection{A topple}\label{sec:topple}
Let $\mathcal Z \to \Delta$ be a flat and generically smooth family of surfaces with central fiber $\mathcal Z_0 = S \cup T$.
Let $\mathcal D \subset \mathcal Z$ be a non-singular surface such that the central fiber $\mathcal D_0$ of $\mathcal D \to \Delta$ has the form $\mathcal D_0 = C \cup \sigma$, where $C$ lies on $S$ and $\sigma$ lies on $T$.
We require that the configuration of $S$, $T$, $C$, and $\sigma$ is as shown in the leftmost diagram in \autoref{fig:topple}.
\begin{figure}[hb]
  \centering
  \begin{tikzpicture}[y=0.80pt, x=0.80pt, yscale=-1.000000, xscale=1.000000, inner sep=0pt, outer sep=0pt]
\begin{scope}[shift={(994.42286,-453.82308)}]
  \begin{scope}[cm={{2.07449,0.0,0.0,2.09181,(-806.54575,-335.80273)}},draw=black]
      \path[xscale=1.106,yscale=0.904,draw=black,fill=black] (-58.2875,481.3282) node
        (text4193-0-5-0) {$4$};
      \path[xscale=1.106,yscale=0.904,draw=black,fill=black] (-65.2007,481.3363) node
        (text4193-0-5-6) {$-4$};
      \path[xscale=1.106,yscale=0.904,draw=black,fill=black] (-48.9097,461.9172) node
        (text4193-62-0-0) {$-2$};
      \path[xscale=1.106,yscale=0.904,draw=black,fill=black] (-71.6094,475.7000) node
        (text4193-62-0-0-6) {$S$};
      \path[xscale=1.106,yscale=0.904,draw=black,fill=black] (-50.7650,474.8610) node
        (text4193-62-0-0-6-1) {$T$};
      \path[draw=black] (-90.3255,426.5707) .. controls (-83.4563,417.2298) and
        (-76.1976,427.0893) .. (-67.8242,418.8672);
      \path[draw=black] (-67.8242,418.8672) .. controls (-61.0968,422.0050) and
        (-56.7495,408.9694) .. (-45.5574,413.2945);
      \path[draw=black] (-67.9415,415.1732) -- (-67.9415,438.7927) --
        (-90.3255,451.2002) -- (-90.3255,401.9412) -- cycle;
      \path[draw=black] (-67.9415,415.1732) -- (-67.9415,438.7927) --
        (-45.5574,451.2003) -- (-45.5574,402.7657) -- cycle;
      \path[draw=black,<-] (-36.3898,428.6664) -- (-12.7585,428.6664);
      \path[draw=black,->] (77.8670,428.9457) -- (101.3926,428.9457);
      \begin{scope}[shift={(11.56912,0)}]
        \path[cm={{0.27818,0.13274,-0.13385,0.27587,(147.15578,451.92539)}},draw=black]
          (-373.2097,49.5735) -- (-339.2150,120.2334) -- (-383.4108,185.0037) --
          (-461.6015,179.1140) -- (-495.5962,108.4542) -- (-451.4003,43.6839) -- cycle;
        \begin{scope}[shift={(-27.50762,-50.9555)}]
          \path[draw=black] (43.2476,455.0122) -- (43.2476,428.6795) --
            (87.0982,454.6010);
          \begin{scope}[cm={{0.0,-1.0599,0.88341,0.0,(-446.37668,378.67573)}}]
            \path[draw=black] (-110.1407,539.1965) .. controls (-99.1754,538.1962) and
              (-90.9910,556.1246) .. (-75.4976,546.9589) .. controls (-64.2279,545.2900) and
              (-67.2787,538.5863) .. (-55.8337,536.1019);
            \path[cm={{0.82019,0.0,0.0,1.0,(-87.56477,120.31112)}},draw=black]
              (18.9498,433.9309) -- (49.2409,433.9309) -- (20.5321,384.5841) --
              (5.1440,410.2005);
            \path[cm={{0.0,1.09757,-0.9111,0.0,(0.0,0.0)}},draw=black,fill=black]
              (490.8278,91.7845) node (text4193-6) {$-1$};
            \path[cm={{0.0,1.09757,-0.9111,0.0,(0.0,0.0)}},draw=black,fill=black]
              (517.6854,83.2013) node (text4193-62) {$0$};
            \path[cm={{0.0,1.09757,-0.9111,0.0,(0.0,0.0)}},draw=black,fill=black]
              (508.8542,67.1640) node (text4193-62-8) {$-1$};
            \path[cm={{0.0,1.09757,-0.9111,0.0,(0.0,0.0)}},draw=black,fill=black]
              (474.3701,82.9142) node (text4193-62-4) {$0$};
            \begin{scope}[cm={{0.82019,0.0,0.0,1.0,(-138.79201,138.91195)}},draw=black]
              \path[cm={{0.0,0.99401,-1.00603,0.0,(0.0,0.0)}},draw=black,fill=black]
                (399.9917,-73.8287) node (text4193-62-24) {$0$};
            \end{scope}
            \path[cm={{0.0,1.09757,-0.9111,0.0,(0.0,0.0)}},draw=black,fill=black]
              (535.4553,83.5464) node (text4193-62-9) {$1$};
            \path[cm={{0.0,1.09757,-0.9111,0.0,(0.0,0.0)}},draw=black,fill=black]
              (513.3980,89.3965) node (text4193-62-0) {$-2$};
            \path[cm={{0.0,1.09757,-0.9111,0.0,(0.0,0.0)}},draw=black,fill=black]
              (502.0664,67.2804) node (text4193-62-6-0) {$0$};
          \end{scope}
        \end{scope}
        \path[draw=black] (36.7024,416.0607) -- (36.7010,440.0663) -- (59.4164,452.1767)
          -- (59.5906,403.6455) -- cycle;
        \path[xscale=1.106,yscale=0.904,draw=black,fill=black] (44.4920,476.3286) node
          (text4193-62-0-0-6-1-5) {$T$};
        \path[xscale=1.106,yscale=0.904,draw=black,fill=black] (43.7676,458.5285) node
          (text4193-62-0-0-2) {$-2$};
        \path[xscale=1.106,yscale=0.904,draw=black,fill=black] (36.0335,474.8215) node
          (text4193-0-5-0-1) {$4$};
        \path[xscale=1.106,yscale=0.904,draw=black,fill=black] (29.4519,474.6345) node
          (text4193-0-5-6-7) {$-5$};
        \path[xscale=1.106,yscale=0.904,draw=black,fill=black] (14.0593,476.7668) node
          (text4193-0-5-6-9) {$\widetilde S$};
      \end{scope}
      \path[xscale=1.106,yscale=0.904,draw=black,fill=black] (111.2843,475.3276) node
        (text4193-62-0-0-6-6) {$S'$};
      \path[xscale=1.106,yscale=0.904,draw=black,fill=black] (129.9749,485.9495) node
        (text4193-62-0-0-6-6-8) {$\frac19(1,2)$};
      \path[draw=black] (134.6135,404.0567) -- (134.6135,438.7927) --
        (112.2295,451.2002) -- (112.2295,404.0567) -- cycle;
      \path[draw=black] (112.2294,416.3222) .. controls (119.0987,406.9812) and
        (126.2400,417.0097) .. (134.6135,408.7877);
  \end{scope}
\end{scope}

\end{tikzpicture}
  \caption{The central fibers in a topple.}
  \label{fig:topple}
\end{figure}
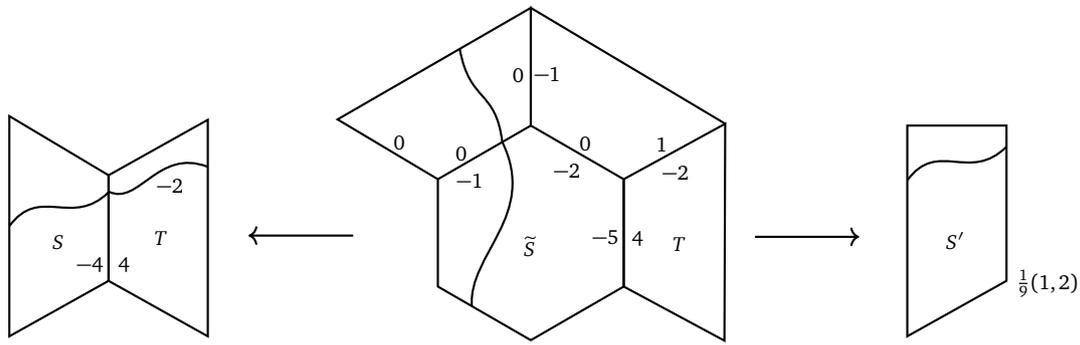
More precisely, we assume the following.
\begin{enumerate}
\item The surfaces $S$ and $T$ meet transversely along a curve $B \cong \PP^1$.
  In particular, $S$ and $T$ are non-singular along $B$.
\item Both $C$ and $\sigma$ are non-singular, and $T$ is non-singular along $\sigma$.
\item $B$ has self-intersection $(-4)$ on $S$ and $4$ on $T$.
\item $\sigma \cong \PP^1$ has self-intersection $(-2)$ on $T$.
\item On $S$, the curves $C$ and $B$ intersect transversely at a unique point $p$.
  Similarly, on $T$, the curves $\sigma$ and $B$ intersect transversely at the same point $p$.
\item The Neron-Severi group $NS(T)$ is spanned by $B$ and $\sigma$.
\end{enumerate}

We make two additional assumptions on the threefold $\mathcal Z$.
First, assume that we have a projective morphism $\pi \from \mathcal Z \to \mathcal Y$ that is an isomorphism on the general fiber and contracts $T$ to a point.
Second, assume that $\mathcal Z$ is non-singular along $B$, $C$, and $\sigma$, and has canonical singularities elsewhere.

\begin{lemma}\label{prop:topple}
  In the setup above, there exists a family of log surfaces $(\mathcal Z', \Delta') \to \Delta$ isomorphic to $(\mathcal Z, \Delta) \to \Delta$ on the generic fiber, whose central fiber $(S', C')$ is obtained from $(S, C)$ by the procedure $(X, C) \leadsto (X', C')$ described in \autoref{fig:type1flip} with the role of $\sigma$ played by $B$.
  Furthermore, the threefold $\mathcal Z'$ is $\Q$-factorial, and the surface $\mathcal D'$ is non-singular.
\end{lemma}
We say that the transformation $\mathcal Z \dashrightarrow \mathcal Z'$ is a \emph{topple} along $T$.
\begin{proof}
  We construct $\mathcal Z'$ from $\mathcal Z$ by two blow ups and two blow downs.
  
  Let $\mathcal Z^{1} \to \mathcal Z$ be the blow up of $\mathcal Z$ along $\sigma$; let $E^{(1)} \subset \mathcal Z^{(1)}$ be the exceptional divisor; and let $\sigma^{(1)} \subset E^{(1)}$ be the intersection of $E^{(1)}$ with the proper transform $\mathcal D^{(1)}$ of $\mathcal D$.
To describe $E^{(1)}$, consider the normal bundle $N_{\sigma/\mathcal{Z}}$. Similar to the proof of \autoref{prop:normalSigma}, we have the following short exact sequence
\[
0 \to N_{\sigma/T} \to N_{\sigma/\mathcal Z} \to N_{T/\mathcal Z}\big|_\sigma \to 0.
\]
Since $N_{\sigma/T} \cong \mathcal{O}(-2)$ and $N_{T/\mathcal Z}\big|_\sigma \cong \mathcal{O}(-1)$, we get $N_{\sigma/\mathcal Z} \cong \O(-1) \oplus \O(-2)$. Hence $E^{(1)} \cong \F_1$, and $\sigma^{(1)} \subset E^{(1)}$ is the directrix.

Let $\mathcal Z^{(2)} \to \mathcal Z^{(1)}$ be the blow up of $\mathcal Z^{(1)}$ along $\sigma^{(1)}$.
Define $E^{(2)}$, $\mathcal D^{(2)}$, and $\sigma^{(2)}$ as before.
By similar computation as above, it follows that $E^{(2)} \cong \PP^1 \times \PP^1$ and $\sigma^{(2)} \subset E^{(2)}$ is a ruling line, more precisely, a line of the ruling opposite to the fibers of $E^{(2)} \to \sigma^{(1)}$.
The middle picture in \autoref{fig:topple} shows a sketch of the central fiber $\mathcal Z^{(2)}_0$ of $\mathcal Z^{(2)} \to \Delta$.

Let $\mathcal Z^{(2)} \to \mathcal Z^{(3)}$ be the contraction in which the lines of the ruling $\left[\sigma^{(2)}\right]$ are contracted.
Note that this contractions contracts $E^{(2)}$ to a $\PP^1$, but in the opposite way compared to the contraction $\mathcal Z^{(2)} \to \mathcal Z^{(1)}$.
We can show that the contraction $\mathcal Z^{(2)} \to \mathcal Z^{(3)}$ exists by appealing to the contraction theorem.
Indeed, it is easy to check that the curve $\sigma^{(2)}$ spans a $K_{\mathcal Z^{(2)}}$ negative ray in $\NE(\pi)$, and hence can be contracted by the contraction theorem.
This contraction must contract all the ruling lines in the same class as $\sigma^{(2)}$, and therefore must contract $E$ to a $\PP^1$.
This is a divisorial contraction, and hence $\mathcal Z^{(3)}$ is $\Q$-factorial with canonical singularities.
It is easy to identify this divisorial contraction: since $E$ is isomorphic to $\PP^1 \times \PP^1$ and the contracted extremal curve is a ruling, this divisorial contraction is a case of \cite[3.3.1]{mor:82}, and hence $\mathcal Z^{(3)}$ is in fact non-singular along the image of $E$.
Let $\mathcal D^{(3)} \subset \mathcal Z^{(3)}$ be the image of $\mathcal D^{(2)}$.
The images of $E^{(1)}$ and $T$ in $\mathcal Z^{(3)}$ lie away from $\mathcal D^{(3)}$.
The image $\overline E^{(1)}$ of $E^{(1)}$ is isomorphic to $\PP^2$.
The image of $T$ is isomorphic to $T$; we denote it by the same letter.
Let $\mathcal Z^{(3)} \to \mathcal Z^{(4)}$ be the contraction that maps $\overline E^{(1)}$ to a point.
This is the contraction of the $K_{\mathcal Z^{(4)}}$-negative extremal ray of $\NE(\pi)$ spanned by a line in $\overline E^{(1)}$.
The image $\overline T$ of $T$ in $\mathcal Z^{(4)}$ is a surface of Picard rank 1; the only curve class on it is $[B]$.
Since the contraction is divisorial, $\mathcal Z^{(4)}$ is $\Q$-factorial with canonical singularities.
As before, it is easy to identify this divisorial contraction: since $\overline E^{(1)}$ is isomorphic to $\PP^2$, the threefold $\mathcal Z^{(3)}$ is non-singular along $\overline E^{(1)}$, and the normal bundle of $\overline E^{(1)}$ is isomorphic to $\O(-2)$, the divisorial contraction is a case of \cite[3.3.5]{mor:82}, and $\mathcal Z^{(4)}$ has the quotient singularity $\A^3/\Z_2$ at the image point of $\overline E^{(1)}$ (here the generator of $\Z_2$ acts on $\A^3$ by $(x,y,z) \mapsto (-x,-y,-z)$).
Finally, let $\mathcal Z^{(4)} \to \mathcal Z'$ be the contraction that maps $\overline T$ to a point.
It is the contraction of the $K_{\mathcal Z^{(4)}}$-negative extremal curve $[B]$ of $\NE(\pi)$.
The rightmost picture in \autoref{fig:topple} shows a sketch of the central fiber $\mathcal Z'_0$ of $\mathcal Z' \to \Delta$.

Set $S' = \mathcal Z'_0$ and $C' = \mathcal D'_0$.
The threefold $\mathcal Z'_0$ is canonical, and hence Cohen--Macaulay.
Theerfore, $S'$ is also Cohen-Macaulay, and since it is non-singular in codimension 1, it is normal.
Observe that the transformation from $S$ to $S'$ is exactly as described in \autoref{fig:type1flip}---two blow ups on $C$ followed by the contraction of a $(-2,-5)$ chain of $\PP^1$s, resulting in a $\frac19(1,2)$ singularity.
\end{proof}

\begin{remark}
  In the notation of \autoref{prop:topple}, note that the Picard rank of $S'$ is the same as the Picard rank of $S$, and the self intersection of $C'$ on $S'$ is given in terms of the self-intersection of $C$ on $S$ by
  \[ C'^2 = C^2 - 2.\]
\end{remark}

\subsubsection{Proof of \autoref{prop:type1flip}}
Having described the two required birational transformations, we take up the proof of \autoref{prop:type1flip}.

The transformation from $(\mathcal X, \mathcal D)$ to $(\mathcal X', \mathcal D')$ goes through a number of intermediate steps $(\mathcal X^{(i)}, \mathcal D^{(i)})$, which can be divided into two stages.
Throughout, $\mathcal D^{(i)} \subset \mathcal X^{(i)}$ denotes the closure of $\mathcal D \big|_{\Delta^{\circ}}$ in $\mathcal X^{(i)}$.

Let $\pi \from \mathcal X \to \mathcal Y$ be the contraction of the curve $\sigma$.
All the intermediate steps $\mathcal X^{(i)}$ will be projective over $\mathcal Y$.
We use the letter $\pi$ to denote the obvious map from various spaces to $\mathcal Y$.

\begin{asparaenum}
\item [\emph {Stage 1 (Blowups):}]

Set $\mathcal X^{(0)} = \mathcal X$, $E_0 = X$, $\mathcal D^{(0)} = \mathcal D$, and $\sigma^{(0)} = \sigma$.
For a Hirzebruch surface $E \cong \F_k$  for $k \geq 1$, denote by $\sigma^E$ the directrix, namely the unique section of self-intersection $(-k)$.

Suppose $\mathcal D^{(0)}$ has an $A_n$ singularity at $p$, where $n \geq 1$.
Let $\mathcal X^{(1)} = \Bl_\sigma\mathcal X$.
Denote by $E_1 \subset \mathcal X^{(1)}$ the exceptional divisor and $\mathcal D^{(1)}$ the proper transform of $\mathcal D^{(0)}$.
By \autoref{prop:normalSigma}, we have $E_1 \cong \F_4$ and ${\mathcal D^{(1)}} \cap E_1 = \sigma^{E_1} \cup F$.
Set $\sigma^{(1)} = \sigma^{E_1}$.
Note that \autoref{prop:normalSigma} applies to $\sigma^{(1)} \subset \mathcal X^{(1)}$ and its blow up.
Indeed, the conditions on the central fiber of $(\mathcal X, \mathcal D)$ hold for $(\mathcal X^{(1)}, \mathcal D^{(1)})$ in an open subset around the $-4$ curve $\sigma^{(1)}$ (the role of $\sigma$ is played by $\sigma^{(1)}$, and the role of $C$ by $F$).
Note that after the blowup, $\mathcal D^{(1)}$ has an $A_{n-1}$ singularity.
Continue blowing up the $-4$ curves in this way, obtaining a sequence
\[ \mathcal X^{(n)} \to \dots \to \mathcal X^{(1)} \to \mathcal X^{(0)}.\]

\autoref{fig:blowup1} shows the central fiber of $\mathcal X^{(n)} \to \Delta$.
In this figure, some curves are labelled with two numbers.
Note that these curves lie on two surfaces; the two numbers are the self-intersection numbers of the curve on either surface.
\begin{figure}[ht]
  \begin{subfigure}{\textwidth}
    \centering
    \begin{tikzpicture}[y=0.80pt, x=0.80pt, yscale=-1.000000, xscale=1.000000, inner sep=0pt, outer sep=0pt]
\begin{scope}[shift={(-91.81513,-353.56212)}]
  \begin{scope}[shift={(274.30189,193.13323)}]
    \begin{scope}[cm={{1.79394,0.0,0.0,1.36208,(16.93306,-97.59296)}}]
        \begin{scope}[cm={{0.74615,0.0,0.0,0.74615,(-28.23095,69.7684)}}]
          \path[draw=black] (-110.4787,201.3485) -- (-78.3858,192.9652) --
            (-78.3858,263.0547) -- (-110.5149,272.7917);
          \path[draw=black] (-110.4456,244.4468) .. controls (-95.7826,253.4400) and
            (-89.4984,229.7187) .. (-78.3858,231.4663);
          \path[draw=black] (-78.3858,192.9652) -- (-29.3277,203.8351);
          \path[draw=black] (-78.3858,263.0547) -- (-29.3277,273.9246);
          \path[draw=black] (-29.3277,203.8351) -- (-29.3277,273.9246);
          \path[draw=black] (-78.3858,231.4663) -- (-29.3277,242.3362);
          \path[xscale=0.908,yscale=1.101,fill=black] (-11.0707,211.6966) node (text5232)
            {$\dots$};
          \path[draw=black] (14.0989,192.9567) -- (63.1570,203.8267) -- (63.1570,273.9161)
            -- (14.0989,263.0462) -- (14.0989,192.9567);
          \path[draw=black] (14.0989,231.4578) -- (63.1570,242.3278);
          \path[draw=black] (63.1570,203.8267) -- (112.2152,191.5148);
          \path[draw=black] (112.2152,191.5148) -- (112.2152,261.6043) --
            (63.1570,273.9162);
          \path[draw=black] (63.1570,242.3278) -- (112.2152,230.0159);
          \path[draw=black] (94.3644,195.9947) .. controls (83.7058,214.1435) and
            (103.6168,243.0026) .. (94.8337,265.9665);
          \path[xscale=0.908,yscale=1.101,fill=black] (-108.7324,230.6359) node
            (text5232-6) {$C$};
          \path[xscale=0.908,yscale=1.101,fill=black] (-97.1916,190.4761) node
            (text5232-6-5) {$-4$};
          \path[xscale=0.908,yscale=1.101,fill=black] (-79.4637,190.4988) node
            (text5232-6-5-5) {$4$};
          \path[xscale=0.908,yscale=1.101,fill=black] (-39.5729,196.1679) node
            (text5232-6-5-2) {$-4$};
          \path[xscale=0.908,yscale=1.101,fill=black] (20.1465,190.0832) node
            (text5232-6-5-5-2) {$4$};
          \path[xscale=0.908,yscale=1.101,fill=black] (59.1000,197.0007) node
            (text5232-6-5-1) {$-4$};
          \path[xscale=0.908,yscale=1.101,fill=black] (76.8279,197.0233) node
            (text5232-6-5-5-7) {$4$};
          \path[xscale=0.908,yscale=1.101,fill=black] (108.2202,194.4997) node
            (text5232-6-5-1-2) {$-4$};
          \path[xscale=0.908,yscale=1.101,fill=black] (-59.8749,223.6533) node
            (text5232-6-5-5-23) {$0$};
          \path[xscale=0.908,yscale=1.101,fill=black] (38.7005,224.3508) node
            (text5232-6-5-5-23-1) {$0$};
          \path[xscale=0.908,yscale=1.101,fill=black] (89.1441,224.9425) node
            (text5232-6-5-5-23-9) {$0$};
        \end{scope}
    \end{scope}
  \end{scope}
\end{scope}

\end{tikzpicture}
    \caption{The central fiber of $\mathcal X^{(n)} \to \Delta$}
    \label{fig:blowup1}
  \end{subfigure}

  \bigskip
  
  \begin{subfigure}{\textwidth}
    \centering
    \begin{tikzpicture}[y=0.80pt, x=0.80pt, yscale=-1.000000, xscale=1.000000, inner sep=0pt, outer sep=0pt]
\begin{scope}[shift={(837.49005,900.30273)}]
  \begin{scope}[cm={{1.79394,0.0,0.0,1.36208,(-637.65428,-1189.806)}}]
    \begin{scope}[cm={{0.7366,0.0,0.0,0.7366,(-29.29289,72.39284)}}]
      \begin{scope}[cm={{0.73854,0.0,0.0,0.95841,(358.62299,1042.2281)}}]
        \path[draw=black] (-333.6424,-887.6329) -- (-266.1794,-876.3252) --
          (-266.1794,-803.4129) -- (-333.6424,-814.7206) -- (-333.6424,-887.6329);
        \path[draw=black] (-333.6424,-847.5445) .. controls (-308.6458,-856.1274) and
          (-293.0254,-821.8463) .. (-266.1794,-836.2368);
        \path[xscale=1.044,yscale=0.957,fill=black] (-294.2544,-866.9770) node
          (text5232-6-5-5-7-70) {$-2$};
        \path[xscale=1.044,yscale=0.957,fill=black] (-311.7205,-906.4714) node
          (text5232-6-5-5-7-7-7) {$4$};
      \end{scope}
        \path[draw=black] (-110.4787,201.3485) -- (-78.3858,192.9652) --
          (-78.3858,263.0547) -- (-110.5149,272.7917);
        \path[draw=black] (-110.4456,244.4468) .. controls (-95.7826,253.4400) and
          (-89.4984,229.7187) .. (-78.3858,231.4663);
        \path[draw=black] (-78.3858,192.9652) -- (-29.3277,203.8351);
        \path[draw=black] (-78.3858,263.0547) -- (-29.3277,273.9246);
        \path[draw=black] (-29.3277,203.8351) -- (-29.3277,273.9246);
        \path[draw=black] (-78.3858,231.4663) -- (-29.3277,242.3362);
        \path[xscale=0.908,yscale=1.101,fill=black] (-8.6162,211.6966) node (text5232)
          {$\dots$};
        \path[draw=black] (14.0989,192.9567) -- (63.1570,203.8267) -- (63.1570,273.9161)
          -- (14.0989,263.0462) -- (14.0989,192.9567);
        \path[draw=black] (14.0989,231.4578) -- (63.1570,242.3278);
        \path[draw=black] (63.1570,203.8267) -- (112.2152,191.5148);
        \path[draw=black] (112.2152,191.5148) -- (112.2152,261.6043) --
          (63.1570,273.9162);
        \path[draw=black] (63.1570,242.3278) -- (112.2152,230.0159);
        \path[xscale=0.908,yscale=1.101,fill=black] (-108.7324,230.6359) node
          (text5232-6) {$C$};
        \path[xscale=0.908,yscale=1.101,fill=black] (-95.9643,190.4761) node
          (text5232-6-5) {$-4$};
        \path[xscale=0.908,yscale=1.101,fill=black] (-80.6909,190.4988) node
          (text5232-6-5-5) {$4$};
        \path[xscale=0.908,yscale=1.101,fill=black] (-40.8001,197.5018) node
          (text5232-6-5-2) {$-4$};
        \path[xscale=0.908,yscale=1.101,fill=black] (21.3738,192.7509) node
          (text5232-6-5-5-2) {$4$};
        \path[xscale=0.908,yscale=1.101,fill=black] (60.3272,197.0007) node
          (text5232-6-5-1) {$-4$};
        \path[xscale=0.908,yscale=1.101,fill=black] (75.6006,197.0233) node
          (text5232-6-5-5-7-2) {$4$};
        \path[xscale=0.908,yscale=1.101,fill=black] (114.5874,191.3062) node
          (text5232-6-5-1-2) {$-4$};
        \path[xscale=0.908,yscale=1.101,fill=black] (-61.1022,224.9872) node
          (text5232-6-5-5-23) {$0$};
        \path[xscale=0.908,yscale=1.101,fill=black] (38.7005,224.3508) node
          (text5232-6-5-5-23-1) {$0$};
        \path[xscale=0.908,yscale=1.101,fill=black] (84.2350,226.2764) node
          (text5232-6-5-5-23-9) {$0$};
    \end{scope}
  \end{scope}
\end{scope}

\end{tikzpicture}
    \caption{The central fiber of $\mathcal X^{(n+1)} \to \Delta$}
    \label{fig:blowup2}
\end{subfigure}

\caption{The central fibers of the $n$th and the $(n+1)$th blow up}
\end{figure}
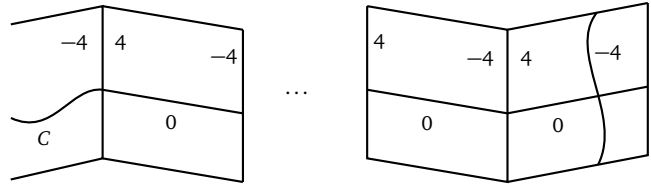
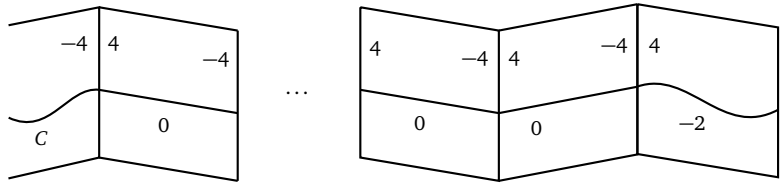

We now continue from $\mathcal X^{(n)}$ and the non-singular surface $\mathcal D^{(n)}$, where $n \geq 0$.
Let $\mathcal X^{(n+1)} \to \mathcal X^{(n)}$ be the blow up of the $-4$ curve $\sigma^{(n)} \subset \mathcal X^{(n)}$.
By \autoref{prop:normalSigma}, the exceptional divisor $E_{n+1}$ is isomorphic to $\F_2$ and it intersects the proper transform $\mathcal D^{(n+1)}$ of $\mathcal D^{(n)}$ in the unique $(-2)$ curve $\sigma^{E_{n+1}}$.
Set $\sigma^{(n+1)} = \sigma^{E_{n+1}}$.
\autoref{fig:blowup2} shows the central fiber of $\mathcal X^{(n+1)} \to \Delta$.

\item[\emph{Stage 2 (Topples):}]
We now continue with $\mathcal X^{(n+1)}$, whose central fiber is the union
\[ \mathcal X^{(n+1)}_0 = E^{(0)} \cup \dots \cup E^{(n)} \cup E^{(n+1)}.\]
After restricting to an open set containing $E^{(n+1)}$, we see that we can topple $\mathcal X^{(n+1)}$ along $E^{(n+1)}$.
That is, the family $(\mathcal X^{(n+1)}, \mathcal D^{(n+1)}) \to \Delta$ satisfies the assumptions of \autoref{prop:topple}.
Let $\mathcal X^{(n+1)} \dashrightarrow \mathcal X^{(n+2)}$ be the topple along $E^{(n+1)}$.
Denote by $E^{(n+2)}$ (resp. $\mathcal D^{(n+2)}$) the image of $E^{(n)}$  (resp. $\mathcal D^{(n+1)}$) under the topple.
Then the central fiber of $\mathcal X^{(n+2)}$ is the union
\[ \mathcal X^{(n+2)}_0 = E^{(0)} \cup \dots \cup E^{(n-1)} \cup E^{(n+2)}.\]
See \autoref{fig:topples} for a sketch of this configuration.

\begin{figure}[hb]
  \centering
  \begin{tikzpicture}[y=0.80pt, x=0.80pt, yscale=-1.000000, xscale=1.000000, inner sep=0pt, outer sep=0pt]
\begin{scope}[shift={(837.49005,900.30273)}]
  \begin{scope}[cm={{1.79394,0.0,0.0,1.36208,(-637.65428,-1189.806)}}]
    \begin{scope}[shift={(-91.4191,0)}]
      \begin{scope}[shift={(-15.60814,0)}]
        \begin{scope}[cm={{0.54401,0.0,0.0,0.70596,(234.86963,840.10046)}}]
          \path[draw=black] (-333.6424,-887.6329) -- (-266.1794,-876.3252) --
            (-266.1794,-803.4129) -- (-333.6424,-814.7206) -- (-333.6424,-887.6329);
          \path[draw=black] (-333.6424,-847.5445) .. controls (-308.6458,-856.1274) and
            (-293.0254,-821.8463) .. (-266.1794,-836.2368);
          \path[xscale=1.044,yscale=0.957,fill=black] (-294.2544,-866.9770) node
            (text5232-6-5-5-7-70) {$-2$};
          \path[xscale=1.044,yscale=0.957,fill=black] (-311.7205,-906.4714) node
            (text5232-6-5-5-7-7-7) {$4$};
        \end{scope}
        \path[draw=black] (-3.2241,217.9242) -- (17.2287,222.5320) -- (17.2287,274.1601)
          -- (-3.2241,269.5523);
        \path[draw=black] (-3.2241,246.2842) -- (17.2287,250.8920);
        \path[draw=black] (17.2287,222.5320) -- (53.3651,213.4631);
        \path[draw=black] (53.3651,213.4631) -- (53.3651,265.0912) --
          (17.2287,274.1601);
        \path[draw=black] (17.2287,250.8920) -- (53.3651,241.8231);
        \path[xscale=0.908,yscale=1.101,fill=black] (23.4414,210.8906) node
          (text5232-6-5-5-7-2) {$4$};
        \path[xscale=0.908,yscale=1.101,fill=black] (52.1592,206.6793) node
          (text5232-6-5-1-2) {$-4$};
        \path[xscale=0.908,yscale=1.101,fill=black] (39.6197,231.1046) node
          (text5232-6-5-5-23-9) {$0$};
        \path[xscale=0.908,yscale=1.101,fill=black] (12.9292,211.1614) node
          (text5232-6-5-1-2-6) {$-4$};
        \path[xscale=0.908,yscale=1.101,fill=black] (162.5727,210.6968) node
          (text5232-6-5-1-2-6-8) {$-4$};
        \path[draw=black,->,dash pattern=on 2.56pt off 1.28pt] (100.7246,245.9496) --
          (121.7961,245.9496);
      \end{scope}
      \begin{scope}[shift={(97.20175,-0.33468)}]
        \begin{scope}[shift={(23.41221,0)}]
          \path[draw=black] (-3.2241,217.9242) -- (17.2287,222.5320) -- (17.2287,274.1601)
            -- (-3.2241,269.5523);
          \path[draw=black] (-3.2241,246.2842) -- (17.2287,250.8920);
          \path[draw=black] (17.2287,222.5320) -- (53.3651,213.4631);
          \path[draw=black] (53.3651,213.4631) -- (53.3651,265.0912) --
            (17.2287,274.1601);
          \path[draw=black] (17.2287,250.8920) .. controls (27.9273,237.1869) and
            (39.3555,264.1828) .. (53.3651,247.0145);
          \path[xscale=0.908,yscale=1.101,fill=black] (23.4414,210.8906) node
            (text5232-6-5-5-7-2-9) {$4$};
          \path[xscale=0.908,yscale=1.101,fill=black] (34.7106,233.7723) node
            (text5232-6-5-5-23-9-6) {$-2$};
        \end{scope}
      \end{scope}
    \end{scope}
  \end{scope}
\end{scope}

\end{tikzpicture}
  \caption{The central fibers in one step of the sequence of topples}
  \label{fig:topples}
\end{figure}
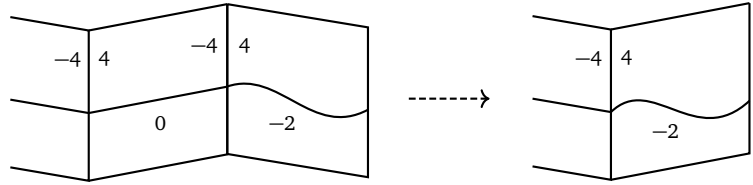
We observe again that an open subset containing $E^{(n+2)}$ satisfies the assumptions of \autoref{prop:topple}, and we continue the process by toppling $\mathcal X^{(n+2)}$ along $E^{(n+2)}$.
After $(n+1)$ topples, we arrive at a pair $(\mathcal X', \mathcal D') = (\mathcal X^{(2n+2)}, \mathcal D^{(2n+2)})$.
Note that in the very first topple, the toppled surface $E^{(n+1)}$ is isomorphic to $\F_2$.
In the subsequent topples, however, the toppled surface is different---it is a rational surface of Picard rank 2 with a $\frac19(1,2)$ singularity (the singularity is not shown in \autoref{fig:topples}).

By construction, $\mathcal X'$ is $\Q$-factorial with canonical singularities.
In particular, both $\mathcal D'$ and $K_{\mathcal X'}$ are $\Q$-Cartier. By construction, the central fiber of $(\mathcal X', \mathcal D') \to \Delta$ is $(X', D')$.
The proof of \autoref{prop:type1flip} is now complete.

\end{asparaenum}

%%% TeX-command-extra-options: "-shell-escape"

%%% Local Variables:
%%% mode: latex
%%% TeX-master: "main"
%%% End:

\subsection{Flipping a $(-3)$ curve (Type II flip)}
\label{sec:typeIIflip}
Let $\Delta$ be the spectrum of a DVR.
Let $\mathcal X \to \Delta$ be a flat family of surfaces and $\mathcal D \subset \mathcal X$ a divisor flat over $\Delta$.
Assume that both $\mathcal X \to \Delta$ and $\mathcal D \to \Delta$ are smooth over $\Delta^\circ$.
Suppose the central fiber $(X, D)$ of $(\mathcal X, \mathcal D) \to \Delta$ has the following form: $X$ is reduced and has two non-singular irreducible components $S,T$, which meet transversely along a non-singular curve $B$, and $D = C \cup \sigma$, where $\sigma \subset T$ is a $(-3)$-curve that meets $B$ transversely at a point $p$ and $C \subset S$ is a non-singular curve that meets $B$ transversely at the same point $p$.
Recall that a $(-3)$-curve is a curve isomorphic to $\PP^1$ whose self-intersection is $-3$.
The left-most diagram in \autoref{fig:type2flip} shows a sketch of $(X, D)$.

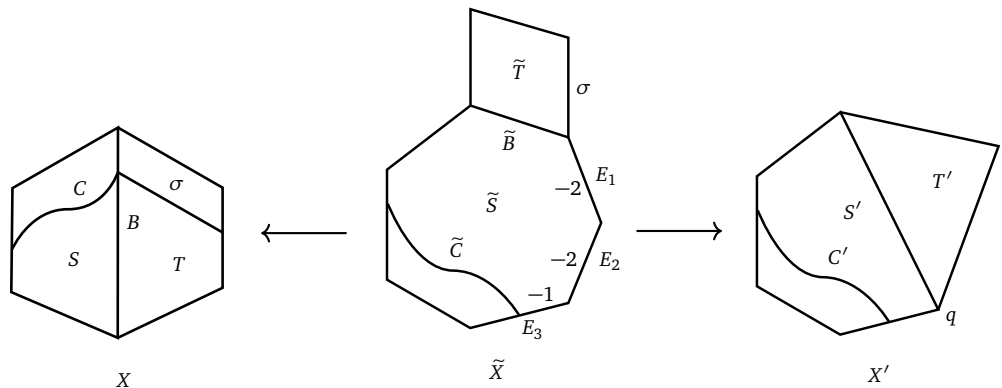
\begin{figure}[hb]
  \centering
  \begin{tikzpicture}[y=0.80pt, x=0.80pt, yscale=-1.000000, xscale=1.000000, inner sep=0pt, outer sep=0pt]
\begin{scope}[shift={(-77.33972,-5.95233)}]
  \begin{scope}[shift={(-36.0,0)}]
      \begin{scope}[cm={{0.80832,0.0,0.0,0.80734,(34.75491,23.28615)}}]
        \begin{scope}[shift={(0,7.43177)}]
          \path[draw=black,line join=miter,line cap=butt,miter limit=4.00,even odd
            rule,line width=0.990pt] (159.4594,162.3942) -- (159.5926,40.5422) --
            (98.5061,75.6556) -- (97.8437,135.9059) -- cycle;
          \path[draw=black,line join=miter,line cap=butt,miter limit=4.00,even odd
            rule,line width=0.990pt] (159.5926,40.5422) -- (220.1665,75.3794) --
            (220.1665,133.4059) -- (159.4594,162.3942);
          \path[xscale=1.102,yscale=0.908,draw=black,fill=black,line join=miter,line
            cap=butt,miter limit=4.00,line width=0.000pt] (118.7834,133.3072) node[above
            right] (text4193-62-0-0-6) {$S$};
          \path[xscale=1.102,yscale=0.908,draw=black,fill=black,line join=miter,line
            cap=butt,miter limit=4.00,line width=0.000pt] (173.0959,135.9668) node[above
            right] (text4193-62-0-0-6-1) {$T$};
          \path[xscale=1.102,yscale=0.908,draw=black,fill=black,line join=miter,line
            cap=butt,miter limit=4.00,line width=0.000pt] (149.7474,109.4270) node[above
            right] (text4193-62-0-0-6-3) {$B$};
          \path[draw=black,line join=miter,line cap=butt,even odd rule,line width=0.990pt]
            (159.5621,66.4791) .. controls (159.5621,66.4791) and (152.4429,88.0601) ..
            (130.5686,87.8813) .. controls (108.6944,87.7025) and (98.3320,111.0059) ..
            (98.3320,111.0059);
          \path[xscale=1.102,yscale=0.908,draw=black,fill=black,line join=miter,line
            cap=butt,miter limit=4.00,line width=0.000pt] (121.1740,87.6083) node[above
            right] (text4193-62-0-0-6-4) {$C$};
          \path[draw=black,line join=miter,line cap=butt,miter limit=4.00,even odd
            rule,line width=0.990pt] (159.5621,66.4791) -- (220.1360,101.3163);
          \path[xscale=1.102,yscale=0.908,draw=black,fill=black,line join=miter,line
            cap=butt,miter limit=4.00,line width=0.000pt] (171.7043,84.3450) node[above
            right] (text4193-62-0-0-6-4-6) {$\sigma$};
          \path[xscale=1.102,yscale=0.908,draw=black,fill=black,line join=miter,line
            cap=butt,miter limit=4.00,line width=0.000pt] (144.5224,210.3331) node[above
            right] (text4193-62-0-0-6-8) {$X$};
          \path[xscale=1.102,yscale=0.908,draw=black,fill=black,line join=miter,line
            cap=butt,miter limit=4.00,line width=0.000pt] (340.4890,204.8778) node[above
            right] (text4193-62-0-0-6-8-2) {$\widetilde X$};
          \path[xscale=1.102,yscale=0.908,draw=black,fill=black,line join=miter,line
            cap=butt,miter limit=4.00,line width=0.000pt] (539.3373,207.3702) node[above
            right] (text4193-62-0-0-6-8-2-8) {$X'$};
          \path[xscale=1.102,yscale=0.908,draw=black,fill=black,line join=miter,line
            cap=butt,miter limit=4.00,line width=0.000pt] (579.8512,170.8133) node[above
            right] (text4193-62-0-0-6-8-2-8-3) {$q$};
        \end{scope}
      \end{scope}
        \begin{scope}[shift={(0,6.0)}]
          \begin{scope}[cm={{0.85409,0.0,0.0,0.8543,(54.24428,13.26655)}}]
            \path[draw=black,line join=miter,line cap=butt,miter limit=4.00,even odd
              rule,line width=0.937pt] (321.2426,159.8411) -- (374.9867,146.1315) --
              (392.9867,102.1700) -- (374.9867,55.4302) -- (321.3758,37.9891) --
              (275.6268,73.1026) -- (275.6268,133.3528) -- cycle;
            \path[draw=black,line join=miter,line cap=butt,even odd rule,line width=0.937pt]
              (348.1146,152.9863) .. controls (348.1146,152.9863) and (333.8064,128.5730) ..
              (311.9321,128.3943) .. controls (290.0578,128.2155) and (276.0842,91.7559) ..
              (276.0842,91.7559);
            \path[draw=black,line join=miter,line cap=butt,miter limit=4.00,even odd
              rule,line width=0.937pt] (321.3758,37.9891) -- (321.3758,-14.8048) --
              (374.9867,0.8100) -- (374.9867,55.4302);
          \end{scope}
          \path[xscale=1.102,yscale=0.908,draw=black,fill=black,line join=miter,line
            cap=butt,miter limit=4.00,line width=0.000pt] (305.1556,105.3225) node[above
            right] (text4193-62-0-0-6-9) {$\widetilde S$};
          \path[xscale=1.102,yscale=0.908,draw=black,fill=black,line join=miter,line
            cap=butt,miter limit=4.00,line width=0.000pt] (315.7038,36.4965) node[above
            right] (text4193-62-0-0-6-1-8) {$\widetilde T$};
          \path[xscale=1.102,yscale=0.908,draw=black,fill=black,line join=miter,line
            cap=butt,miter limit=4.00,line width=0.000pt] (311.9483,73.1036) node[above
            right] (text4193-62-0-0-6-3-8) {$\widetilde B$};
          \path[xscale=1.102,yscale=0.908,draw=black,fill=black,line join=miter,line
            cap=butt,miter limit=4.00,line width=0.000pt] (289.2688,128.5853) node[above
            right] (text4193-62-0-0-6-4-0) {$\widetilde C$};
          \path[xscale=1.102,yscale=0.908,draw=black,fill=black,line join=miter,line
            cap=butt,miter limit=4.00,line width=0.000pt] (351.4066,91.5466) node[above
            right] (text4193-62-0-0-6-3-1) {$E_1$};
          \path[xscale=1.102,yscale=0.908,draw=black,fill=black,line join=miter,line
            cap=butt,miter limit=4.00,line width=0.000pt] (352.9962,135.1880) node[above
            right] (text4193-62-0-0-6-3-1-3) {$E_2$};
          \path[xscale=1.102,yscale=0.908,draw=black,fill=black,line join=miter,line
            cap=butt,miter limit=4.00,line width=0.000pt] (320.0418,170.7531) node[above
            right] (text4193-62-0-0-6-3-1-3-3) {$E_3$};
          \path[xscale=1.102,yscale=0.908,draw=black,fill=black,line join=miter,line
            cap=butt,miter limit=4.00,line width=0.000pt] (334.3130,96.6102) node[above
            right] (text4193-62-0-0-6-3-1-0) {$-2$};
          \path[xscale=1.102,yscale=0.908,draw=black,fill=black,line join=miter,line
            cap=butt,miter limit=4.00,line width=0.000pt] (332.3671,133.0239) node[above
            right] (text4193-62-0-0-6-3-1-0-6) {$-2$};
          \path[xscale=1.102,yscale=0.908,draw=black,fill=black,line join=miter,line
            cap=butt,miter limit=4.00,line width=0.000pt] (322.7314,151.4668) node[above
            right] (text4193-62-0-0-6-3-1-0-6-3) {$-1$};
          \path[draw=black,line join=miter,line cap=butt,miter limit=4.00,line
            width=0.800pt, <-] (230.7383,105.4023) -- (270.7612,105.4023);
          \path[draw=black,line join=miter,line cap=butt,miter limit=4.00,line
            width=0.800pt, <-] (446.4411,104.4023) -- (406.4183,104.4023);
        \end{scope}
        \path[xscale=1.102,yscale=0.908,draw=black,fill=black,line join=miter,line
          cap=butt,miter limit=4.00,line width=0.000pt] (343.2249,50.8778) node[above
          right] (text4193-62-0-0-6-4-6-6) {$\sigma$};
      \begin{scope}[shift={(173.12125,3.08073)}]
        \begin{scope}[shift={(0,6.0)}]
          \begin{scope}[cm={{0.85409,0.0,0.0,0.8543,(54.24428,13.26655)}}]
            \path[draw=black,line join=miter,line cap=butt,miter limit=4.00,even odd
              rule,line width=0.937pt] (321.2426,159.8411) -- (374.9867,146.1315) --
              (321.3758,37.9891) -- (275.6268,73.1026) -- (275.6268,133.3528) -- cycle;
            \path[draw=black,line join=miter,line cap=butt,even odd rule,line width=0.937pt]
              (348.1146,152.9863) .. controls (348.1146,152.9863) and (333.8064,128.5730) ..
              (311.9321,128.3943) .. controls (290.0578,128.2155) and (276.0842,91.7559) ..
              (276.0842,91.7559);
            \path[draw=black,line join=miter,line cap=butt,miter limit=4.00,even odd
              rule,line width=0.937pt] (321.3758,37.9891) -- (408.0179,56.5988) --
              (374.9867,146.1315);
            \path[xscale=1.000,yscale=1.000,fill=black,line join=miter,line cap=butt,line
              width=0.800pt] (-75.3524,192.5967) node[above right] (text6342) {$$};
          \end{scope}
          \path[xscale=1.102,yscale=0.908,draw=black,fill=black,line join=miter,line
            cap=butt,miter limit=4.00,line width=0.000pt] (299.7101,105.3225) node[above
            right] (text4193-62-0-0-6-9-4) {$S'$};
          \path[xscale=1.102,yscale=0.908,draw=black,fill=black,line join=miter,line
            cap=butt,miter limit=4.00,line width=0.000pt] (337.0318,89.3848) node[above
            right] (text4193-62-0-0-6-1-8-6) {$T'$};
          \path[xscale=1.102,yscale=0.908,draw=black,fill=black,line join=miter,line
            cap=butt,miter limit=4.00,line width=0.000pt] (292.8991,128.5853) node[above
            right] (text4193-62-0-0-6-4-0-1) {$C'$};
        \end{scope}
      \end{scope}
  \end{scope}
\end{scope}

\end{tikzpicture}
  \caption{The central fiber $X$ is replaced by $X'$ in a type 2 flip.}
  \label{fig:type2flip}
\end{figure}

Construct $(X', D')$ from $(X, D)$ as follows (see \autoref{fig:type2flip}).
Let $\widetilde S \to S$ be the blow up of $S$ three times, first at $p$, second at the intersection of the exceptional divisor of the first blow-up with the proper transform of $C$, and third at the intersection point of the exceptional divisor of the second blow up with the proper transform of $C$.
Equivalently, $\widetilde S$ is the minimal resolution of the blow-up of $S$ at the unique subscheme of $C$ of length 3 supported at $p$.
Denote by $\widetilde C$ and $\widetilde B$ the proper transforms of $C$ and $B$ in $\widetilde S$.
Let $\widetilde X$ be the union of $\widetilde S$ and $T$, glued along $\widetilde B \subset \widetilde S$ and $B \subset T$ via the canonical isomorphism $\widetilde B \to B$ induced by the identity on $B$.
Let $E_i$ be the proper transform in $\widetilde S$ of the exceptional divisor of the $i$th blowup, for $i = 1, 2, 3$.
Let $S'$ be obtained from $\widetilde S$ by contracting $\widetilde E_1$ and $\widetilde E_2$.
Let $T'$ be obtained from $T$ by contracting $\sigma$.
Let $X'$ be the union of $S'$ and $T'$ glued along the image of $\widetilde B$ in $S'$ and the image of $B$ in $T'$ via the isomorphism between the two induced by the identity on $B$.
Let $B' \subset X'$ be the image of either of these curves.
Let $C' \subset X'$ be the image of $\widetilde C$, and set $D' = C'$.
Let $\nu \subset X'$ be the image of $E_3 \subset \widetilde S$.

We impose an additional hypothesis on the structure of $(\mathcal X, \mathcal D)$ along $B$.
Assume that there exists a family of (not necessarily projective) curves $\mathcal P \to \Delta$, smooth over $\Delta^\circ$, and with a single node $p$ on the central fiber, an open subset $\mathcal U \subset \mathcal X$ containing $B$, and an isomorphism $\mathcal U \cong B \times \mathcal P$ over $\Delta$.
Assume, furthermore, that the first projection $\mathcal U \to \mathcal P$ restricts to an isomorphism $\mathcal D \cap \mathcal U\to \mathcal P$.

\begin{proposition}
\label{prop:typeIIflip}
  Let $(\mathcal X, \mathcal D) \to \Delta$ be a family of log surfaces as described above.
  There exists a flat family $(\mathcal X', \mathcal D') \to \Delta$ isomorphic to $(\mathcal X, \mathcal D)$ over $\Delta^\circ$ such that the central fiber of $(\mathcal X', \mathcal D') \to \Delta$ is $(X',D')$.
  Furthermore, $\mathcal X'$ is $\Q$-factorial and has canonical singularities.
\end{proposition}
\begin{remark}
  Note that $(X', D')$ is log canonical.
  Also note that it depends only on $(X, D)$, not on the family $(\mathcal X, \mathcal D) \to \Delta$.  
\end{remark}

Before proving \autoref{prop:typeIIflip}, we look at $X'$ and its two components $S'$ and $T'$ in more detail.
The contraction $\widetilde S \to S'$ results an $A_2 = \frac{1}{3}(1,2)$ singularity on $S'$ at the image point of the chain $E_1, E_2$.
The contraction $T \to T'$ results in a $\frac{1}{3}(1,1)$ singularity at the image point of the curve $\sigma$.
These two singularities are glued together in $X'$, say at a point $q$.
The complete local ring of $X'$ at $q$ is the ring 
\[ [\C\llbracket x,y,z \rrbracket/ (xy)]^{\mu_3}, \text{ where $\zeta \in \mu_3$ acts by }  \zeta \cdot (x, y, z) \mapsto  (\zeta x, \zeta^2y, \zeta z).\]
Finally, observe that the Picard ranks of the new surfaces are given by
\begin{align*}
  \rho(S') = \rho(S) + 1 \text{ and }   \rho(T') = \rho(T) - 1.
\end{align*}
On $S'$, we have the intersection numbers
\begin{align*}
  (C')^2 = C^2 - 3, \quad \nu^2 =  -\frac{1}{3}, \quad  B' \cdot \nu = \frac{1}{3}.
\end{align*}
On $S'$ and $T'$, we have the following intersection numbers of $B'$
\begin{align*}
  \left(B'|{S'}\right)^2 = \left( B|_S \right)^2 - \frac{1}{3} \text{ and }
  \left(B'|{T'}\right)^2 = \left( B|_T \right)^2 + \frac{1}{3}.
\end{align*}

\subsubsection{Proof of \autoref{prop:typeIIflip}: The non-singular case}
\label{subsec:pftypeIIflip}
Assume that $\mathcal P$ is non-singular.
Then both $\mathcal X$ and $\mathcal D$ are non-singular.

We construct $\mathcal X'$ from $\mathcal X$ by an explicit sequence of blow ups and blow downs.
We denote the intermediate steps in this process by ${\mathcal X}^{(i)}$.
Throughout, ${\mathcal D}^{(i)} \subset {\mathcal X}^{(i)}$ denotes the closure of $\mathcal D \big | _ {\Delta^\circ}$ in $\mathcal X^{(i)}$, or equivalently the proper transform of $\mathcal D$ in $\mathcal X^{(i)}$.
Let $\pi \from \mathcal X \to \mathcal Y$ be the contraction of $\sigma$.
All the ${\mathcal X}^{(i)}$ will be projective over $\mathcal Y$.

The first three steps consist of blow-ups; their central fibers are depicted in \autoref{fig:blowups}.
\begin{figure}[hb]
  \centering
  \begin{tikzpicture}[y=0.80pt, x=0.80pt, yscale=-1.000000, xscale=1.000000, inner sep=0pt, outer sep=0pt]
\begin{scope}[shift={(-31.39138,-83.03963)}]
  \begin{scope}[shift={(-82.07147,71.97058)}]
    \begin{scope}[cm={{0.85409,0.0,0.0,0.8543,(54.24428,19.26655)}}]
    \end{scope}
      \begin{scope}[shift={(0,6.0)}]
        \begin{scope}[cm={{0.80832,0.0,0.0,0.80734,(34.75491,23.28615)}}]
          \begin{scope}[shift={(0,7.43177)}]
          \end{scope}
        \end{scope}
      \end{scope}
        \begin{scope}[cm={{1.06678,0.0,0.0,1.0,(-30.52379,0.0)}}]
          \begin{scope}[cm={{1.22977,0.0,0.0,1.22977,(-30.98703,-42.8756)}}]
            \begin{scope}[cm={{0.60464,0.0,0.0,0.64322,(76.1701,45.93918)}}]
              \path[draw=black] (159.5621,66.4791) .. controls (159.5621,66.4791) and
                (152.4429,88.0601) .. (130.5686,87.8813) .. controls (108.6944,87.7025) and
                (98.3320,111.0059) .. (98.3320,111.0059);
              \path[draw=black] (159.4594,162.3942) -- (159.5926,40.5422) -- (98.5061,75.6556)
                -- (97.8437,135.9059) -- cycle;
              \path[draw=black] (159.5926,40.5422) -- (220.1665,75.3794) --
                (220.1665,133.4059) -- (159.4594,162.3942);
              \path[xscale=1.102,yscale=0.908,draw=black,fill=black] (118.7834,133.3072) node
                (text4193-62-0-0-6) {$S$};
              \path[xscale=1.102,yscale=0.908,draw=black,fill=black] (173.0959,135.9668) node
                (text4193-62-0-0-6-1) {$T$};
              \path[xscale=1.102,yscale=0.908,draw=black,fill=black] (152.5615,153.9651) node
                (text4193-62-0-0-6-3) {$B$};
              \path[xscale=1.102,yscale=0.908,draw=black,fill=black] (121.1740,84.1823) node
                (text4193-62-0-0-6-4) {$C$};
              \path[draw=black] (159.5621,66.4791) -- (220.1360,101.3163);
              \path[xscale=1.102,yscale=0.908,draw=black,fill=black] (171.7043,80.9190) node
                (text4193-62-0-0-6-4-6) {$\sigma$};
              \path[xscale=1.102,yscale=0.908,draw=black,fill=black] (170.4939,108.6843) node
                (text4193-62-0-0-6-4-6-3) {$-3$};
            \end{scope}
            \path[cm={{0.83489,0.0,0.0,0.73272,(61.19248,35.28615)}},draw=black,fill=black]
              (132.4541,205.2656) node (text4193-62-0-0-6-8) {$X$};
            \path[xscale=1.067,yscale=0.937,draw=black,fill=black] (385.1022,198.2119) node
              (text4193-62-0-0-6-8-3-1) {$X^{(2)}$};
            \path[xscale=1.067,yscale=0.937,draw=black,fill=black] (491.8693,198.2119) node
              (text4193-62-0-0-6-8-3-3) {$X^{(3)}$};
            \path[draw=black,<-] (216.6052,112.8493) -- (231.0882,112.8493);
            \begin{scope}[shift={(-10.67164,0)}]
              \path[cm={{0.66622,0.0,0.0,0.58377,(76.1701,45.93918)}},draw=black,fill=black]
                (343.7056,151.2028) node (text4193-62-0-0-6-4-6-3-7) {$-2$};
              \path[cm={{0.66622,0.0,0.0,0.58377,(76.1701,45.93918)}},draw=black,fill=black]
                (362.1289,106.2388) node (text4193-62-0-0-6-4-6-3-3) {$-3$};
              \path[cm={{0.66622,0.0,0.0,0.58377,(76.1701,45.93918)}},draw=black,fill=black]
                (362.0470,123.8676) node (text4193-62-0-0-6-4-6-3-3-2) {$2$};
                \path[cm={{0.83489,0.0,0.0,0.73272,(61.19248,35.28615)}},draw=black,fill=black]
                  (277.5539,205.2656) node (text4193-62-0-0-6-8-3) {$X^{(1)}$};
                \begin{scope}[cm={{0.79681,0.0,0.0,0.87894,(62.86953,17.42559)}}]
                  \path[draw=black] (287.1649,125.8696) -- (329.2733,126.4107);
                  \begin{scope}[cm={{0.95209,0.0,0.0,1.0,(-35.11859,0.0)}}]
                      \begin{scope}[cm={{0.85067,0.0,0.0,0.85066,(43.25314,16.39203)}}]
                        \path[cm={{0.85409,0.0,0.0,0.8543,(54.24428,19.26655)}},draw=black]
                          (321.3759,152.9604) -- (364.3421,103.2277) -- (321.3759,37.9891) --
                          (275.6269,73.1026) -- (275.6269,133.3528) -- cycle;
                        \path[draw=black] (347.0762,128.6971) .. controls (347.0762,128.6971) and
                          (340.2150,113.8827) .. (321.2026,113.8039) .. controls (302.1901,113.7251) and
                          (290.0446,97.6535) .. (290.0446,97.6535);
                        \path[cm={{0.85409,0.0,0.0,0.8543,(54.24428,19.26655)}},draw=black]
                          (321.3758,37.9891) -- (382.5819,37.5179) -- (425.8585,103.2277) --
                          (364.3421,103.2277);
                        \path[draw=black] (417.9652,107.4539) -- (381.2681,149.9404) --
                          (328.7277,149.9404);
                      \end{scope}
                  \end{scope}
                \end{scope}
            \end{scope}
            \begin{scope}[cm={{0.9374,0.0,0.0,1.0,(-136.63858,201.28647)}}]
              \begin{scope}[cm={{0.66142,0.0,0.0,0.66219,(473.82961,-160.01723)}}]
                \path[cm={{0.85409,0.0,0.0,0.8543,(54.24428,19.26655)}},draw=black]
                  (321.2426,159.8411) -- (374.9867,146.1315) -- (392.9867,102.1700) --
                  (374.9867,55.4302) -- (321.3758,37.9891) -- (275.6268,73.1026) --
                  (275.6268,133.3528) -- cycle;
                \path[draw=black] (350.2197,151.0117) .. controls (350.2197,151.0117) and
                  (347.8044,128.2224) .. (325.9159,128.0739) .. controls (304.0273,127.9255) and
                  (290.0446,97.6535) .. (290.0446,97.6535);
                \path[draw=black] (328.7277,51.7206) -- (328.7277,17.1374) -- (406.4763,41.5496)
                  -- (374.5161,66.6205);
                \path[draw=black] (406.4764,41.5496) -- (431.3587,106.1765) --
                  (389.8897,106.5503);
                \path[draw=black] (431.3587,106.1765) -- (402.3396,177.0675) --
                  (374.5161,144.1065);
              \end{scope}
            \end{scope}
            \begin{scope}[cm={{0.9374,0.0,0.0,1.0,(11.84338,4.0)}}]
              \begin{scope}[shift={(128.62844,-0.71227)}]
                \begin{scope}[cm={{1.12509,0.0,0.0,0.82147,(-24.47788,16.91376)}}]
                  \path[draw=black] (380.9774,151.2514) -- (375.0084,185.6322);
                  \path[draw=black] (375.0084,185.6322) -- (424.3193,168.3801);
                  \path[draw=black] (393.6790,147.3765) -- (392.1799,179.6245);
                \end{scope}
              \end{scope}
            \end{scope}
            \path[draw=black,<-] (339.7887,112.8493) -- (354.2717,112.8493);
            \path[draw=black,<-] (458.8682,112.8493) -- (473.3511,112.8493);
            \path[draw=black] (407.8949,128.1005) -- (408.3831,156.6022);
            \begin{scope}[cm={{1.01606,0.0,0.0,0.82147,(71.78118,20.20149)}}]
                \begin{scope}[cm={{0.85067,0.0,0.0,0.85066,(43.25314,16.39203)}}]
                    \path[draw=black] (328.2223,144.2163) -- (356.3862,118.1348) --
                      (356.4698,82.9016) -- (328.2223,58.8765) -- (294.2819,84.9402) --
                      (294.2819,129.6622) -- cycle;
                    \path[draw=black] (338.0255,135.1378) .. controls (338.0255,135.1378) and
                      (332.3530,122.8499) .. (317.9144,122.7745) .. controls (303.4757,122.6991) and
                      (294.2521,107.3238) .. (294.2521,107.3238);
                    \path[cm={{0.85409,0.0,0.0,0.8543,(54.24428,19.26655)}},draw=black]
                      (320.7841,46.3655) -- (321.3848,-4.1827) -- (382.4000,47.9615) --
                      (353.8573,74.4881)(381.5965,138.3257) -- (353.7594,115.7302);
                    \path[draw=black] (380.1615,137.4380) -- (328.6201,185.1553) --
                      (328.2223,144.2163);
                \end{scope}
            \end{scope}
            \path[draw=black] (444.9071,75.7624) -- (444.3140,129.7079);
            \path[xscale=1.103,yscale=0.907,draw=black,fill=black] (394.7697,129.8528) node
              (text4193-62-0-0-6-8-5-3) {$-2$};
            \path[xscale=1.103,yscale=0.907,draw=black,fill=black] (390.4662,143.3862) node
              (text4193-62-0-0-6-8-5-3-7) {$1$};
            \path[xscale=1.103,yscale=0.907,draw=black,fill=black] (375.7710,157.1168) node
              (text4193-62-0-0-6-8-5-3-7-3) {$-1$};
            \path[xscale=1.103,yscale=0.907,draw=black,fill=black] (395.1615,101.2468) node
              (text4193-62-0-0-6-8-5-3-9) {$2$};
            \path[xscale=1.103,yscale=0.907,draw=black,fill=black] (389.7962,88.3092) node
              (text4193-62-0-0-6-8-5-3-9-3) {$-1$};
            \path[xscale=1.103,yscale=0.907,draw=black,fill=black] (508.4041,118.3067) node
              (text4193-62-0-0-6-8-5-3-5) {$-2$};
            \path[xscale=1.103,yscale=0.907,draw=black,fill=black] (509.0568,131.7153) node
              (text4193-62-0-0-6-8-5-3-7-5) {$1$};
            \path[xscale=1.103,yscale=0.907,draw=black,fill=black] (500.9193,159.4966) node
              (text4193-62-0-0-6-8-5-3-7-3-3) {$-1$};
            \path[xscale=1.103,yscale=0.907,draw=black,fill=black] (501.2419,93.5458) node
              (text4193-62-0-0-6-8-5-3-9-8) {$2$};
            \path[xscale=1.103,yscale=0.907,draw=black,fill=black] (493.9792,81.0183) node
              (text4193-62-0-0-6-8-5-3-9-3-5) {$-1$};
            \path[xscale=1.103,yscale=0.907,draw=black,fill=black] (492.0253,165.6247) node
              (text4193-62-0-0-6-8-5-3-7-3-3-3) {$0$};
            \path[xscale=1.103,yscale=0.907,draw=black,fill=black] (465.2559,174.7626) node
              (text4193-62-0-0-6-8-5-3-7-3-3-3-3) {$0$};
          \end{scope}
        \end{scope}
  \end{scope}
\end{scope}

\end{tikzpicture}
  \caption{The central fibers $X^{(i)}$ of the first three blow ups $\mathcal X^{(i)}$ of $\mathcal X$.}
  \label{fig:blowups}
\end{figure}
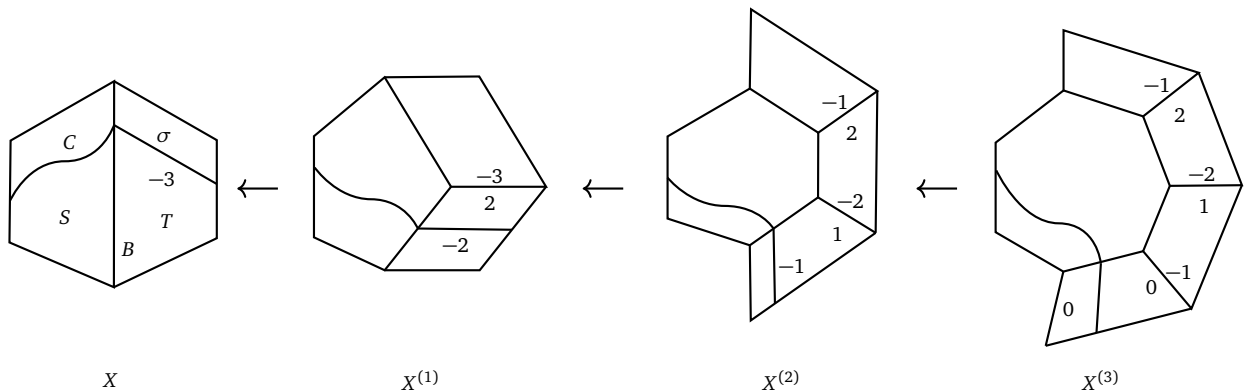
The first step ${\mathcal X}^{(1)} \to \mathcal X$ is the blow up at $\sigma$.
Let $E^{(1)} \subset {\mathcal X}^{(1)}$ be the exceptional divisor.
Note that the central fiber of ${\mathcal X}^{(1)} \to \Delta$ is the union of $E^{(1)}$ and the proper transform $X^{(1)}$ of $X$.
The surface $X^{(1)}$ has two smooth irreducible components, namely $\Bl_p S$ and $T$, which intersect transversely along the proper transform of $B$.
Set $\sigma^{(1)} = E^{(1)} \cap \mathcal D^{(1)}$.
The following lemma identifies the normal bundle of $\sigma$ and hence the isomorphism class of $E^{(1)}$.
\begin{lemma}
  \label{lem:uniquenormal}
  The normal bundle $N_{\sigma/\mathcal X}$ is given by
  \[
    N_{\sigma/\mathcal X} \cong \O(-1) \oplus \O(-3).
  \]
  As a result, we have $E^{(1)} \cong \F_2$, and $\sigma^{(1)}$ is the unique $-2$ curve on $E^{(1)}$.
\end{lemma}
\begin{proof}
  We have the exact sequence of bundles
  \[ 0 \to N_{\sigma/T} \to N_{\sigma/\mathcal X} \to N_{T/\mathcal X}\big|_\sigma \to 0,\]
  in which the kernel is $\O(-3)$ and the cokernel is $\O(-1)$.
  Therefore, the only possibilities for $N_{\sigma/\mathcal X}$ are $\O(-i) \oplus \O(-4+i)$ for $i =1,2$.

  Similar arguments from the proof of \autoref{prop:normalSigma} shows that $\mathcal D^{(1)}\cap E$ is an effective divisor on $E$ of self-intersection $(-2)$.
  The map $\mathcal{D}^{(1)} \to \mathcal D$ is the blow-up of $\mathcal D$ along $\sigma$.
  Since $\mathcal D$ is non-singular, this is an isomorphism.
  Therefore, the scheme-theoretic intersection $\mathcal D^{(1)} \cap E$ is a section of $E \to \sigma$.
  Among the two possibilities for $E$ given by $i = 1, 2$, only $i = 1$ yields a surface with a section of self-intersection $(-2)$.
  The result follows.
\end{proof}

The second step $\mathcal X^{(2)} \to \mathcal X^{(1)}$ is the blow up of $\mathcal X^{(1)}$ along $\sigma^{(1)}$.
Define $E^{(2)}$, $\mathcal D^{(2)}$, and $\sigma^{(2)}$ as before.
By similar computation as in the proof of \autoref{lem:uniquenormal}, we get that $E^{(2)} \cong \F_1$ and $\sigma^{(2)} \subset E^{(2)}$ is the unique curve of self-intersection $(-1)$.

The third step $\mathcal X^{(3)} \to \mathcal X^{(2)}$ is the blow up of $\mathcal X^{(3)}$ along $\sigma^{(2)}$.
Define $E^{(3)}$, $\mathcal D^{(3)}$, and $\sigma^{(3)}$ as before.
Again, by a similar computation as before, we get that $E^{(3)} \cong \PP^1 \times \PP^1$ and $\sigma^{(3)} \subset E^{(3)}$ is a line of a ruling, opposite to the fibers of $E^{(3)} \to \sigma^{(2)}$.

The next three steps consist of divisorial contractions.
Let $\mathcal X^{(3)} \to \mathcal X^{(4)}$ be the contraction in which the lines of the ruling of $\sigma^{(3)}$ are contracted.
This results in the contraction of $E^{(3)}$ in the opposite direction as compared with the contraction in $\mathcal X^{(3)} \to \mathcal X^{(2)}$.
Note that this is the contraction of the $K_{\mathcal X^{(3)}}$-negative extremal ray of $\NE(\pi)$ spanned by $\sigma^{(3)}$, and thus its existence is guaranteed by the contraction theorem.
Since this is a divisorial contraction, $\mathcal X^{(4)}$ is $\Q$-factorial with canonical singularities.
In fact, it turns out that the contraction does not introduce any new singularities (this is a case of \cite[3.3.1]{mor:82}). 
Let $\mathcal D^{(4)} \subset \mathcal X^{(4)}$ be the image of $\mathcal D^{(3)}$.
The images of $E^{(1)}$ and $E^{(2)}$ in $\mathcal X^{(4)}$ lie away from $\mathcal D^{(4)}$.
The image $\overline E^{(2)}$ of $E^{(2)}$ is isomorphic to $\PP^2$.
The image of is isomorphic to $E^{(1)}$; we denote it by the same notation.

Let $\mathcal X^{(4)} \to \mathcal X^{(5)}$ be the map that contracts $\overline E^{(2)}$ to a point.
This is the contraction of the $K_{\mathcal X^{(4)}}$-negative extremal ray of $\NE(\pi)$ spanned by a line in $\overline E^{(2)}$.
Again, $\mathcal X^{(5)}$ is $\Q$-factorial with canonical singularities.
The image $\overline E^{(1)}$ of $E^{(1)}$ is a surface of Picard rank 1; the only curve class on it is $[\tau]$, where $\tau$ is the image of $E^{(1)} \cap \Bl_p S$.
The only new singularity on $\mathcal X^{(5)}$ is at the image point of $\overline E^{(2)}$; it is the quotient singularity $\A^3/\Z_2$ where the generator of $\Z_2$ acts by $(x,y,z) \mapsto (-x,-y,-z)$ (this is case of \cite[3.3.5]{mor:82}).

Finally, let $\mathcal X^{(5)} \to \mathcal X'$ be the contraction that maps $\overline E^{(1)}$ to a point.
It is the contraction of the $K_{\mathcal X^{(5)}}$-negative extremal curve $[\tau]$ of $\NE(\pi)$.

Set $X' = \mathcal X'_0$ and $D' = \mathcal D'_0$.
Observe that the transformation from $X$ to $X'$ is exactly as described in \autoref{prop:typeIIflip}---on one component $S$, it is the result of two blow ups on $C$ followed by the contraction of a $(-2,-2)$ chain of $\PP^1$s, resulting in an $A_2$ singularity.
On the other component $T$, it is just the contraction of $\sigma$, resulting in a $\frac{1}{3}(1,1)$ singularity.
The proof of \autoref{prop:typeIIflip} is now complete, under the assumption that $\mathcal P$ is non-singular.

\subsubsection{Proof of \autoref{prop:typeIIflip}: The general case.}\label{sec:accordion}
Since the central fiber of $\mathcal P \to \Delta$ has a nodal singularity at $p$, the surface $\mathcal P$ has an $A_n$ singularity at $p$ for some $n \geq 0$.
We have already dispensed the case $n = 0$, so assume $n \geq 1$.
Let $\mathcal P^{(0)} \to \mathcal  P$ be the minimal resolution of singularities.
The exceptional divisor of $\mathcal P^{(0)}_0$ consists of a chain of $n$ rational curves.
Set
\begin{align*}
  \mathcal X^{(0)}
  &= \left(\mathcal U \times_{\mathcal P} \mathcal P^{(0)}\right) \bigcup_{\mathcal U \setminus B} \left(\mathcal X \setminus B\right) \\
  &= \left(B \times \mathcal P^{(0)}\right) \bigcup_{\mathcal U \setminus B} \left(\mathcal X \setminus B\right).
\end{align*}
Then we have a map $\mathcal X^{(0)} \to \mathcal X$, which is a resolution of singularities.
The central fiber $X^{(0)}$ of $\mathcal X^{(0)} \to \Delta$ is the union
\[ X^{(0)} = S^{(0)}\cup E_1^{(0)}\cup \dots \cup E_n^{(0)} \cup T^{(0)},\]
where $S^{(0)}$ and $T^{(0)}$ denote the strict transforms of $S$ and $T$, and each $E_i^{(0)}$ is isomorphic to $B \times \PP^1$.
Let $\mathcal D^{(0)}$ be the proper transform of $\mathcal D$, and $\sigma^{(0)}$ and $C^{(0)}$ the proper transforms of $\sigma$ and $C$, respectively.
Since the map $\mathcal D \to \mathcal P$ is an isomorphism over an open subset of $\mathcal P$ containing $p$, the map $\mathcal D^{(0)} \cap \mathcal U \to \mathcal P^{(0)}$ is an isomorphism over an open subset of $\mathcal P^{(0)}$ containing the pre-image of $p$ in $\mathcal P^{(0)}$.
In particular, $\mathcal D^{(0)}$ is non-singular.
Also, the intersection $D_i := \mathcal D^{(0)} \cap E_i^{(0)}$ is a section of $E_i^{(0)} \to \PP^1$.
See \autoref{fig:accordions1} for a picture of $(\mathcal X^{0}, \mathcal D^{(0)})$.
\begin{figure}[ht]
  \centering
  \begin{tikzpicture}[y=0.80pt, x=0.80pt, yscale=-1.000000, xscale=1.000000, inner sep=0pt, outer sep=0pt]
\begin{scope}[shift={(847.3197,900.30695)}]
  \begin{scope}[cm={{1.79394,0.0,0.0,1.36208,(-637.65428,-1189.806)}},miter limit=4.00,line width=1.600pt]
    \begin{scope}[cm={{0.7366,0.0,0.0,0.7366,(-29.29289,72.39284)}}]
        \begin{scope}[cm={{0.73854,0.0,0.0,0.95841,(358.62299,1042.2281)}}]
          \path[draw=black,line join=round,line cap=round,miter limit=4.00,line
            width=0.826pt] (-333.6424,-887.6329) -- (-266.1794,-876.3252) --
            (-266.1794,-803.4129) -- (-333.6424,-814.7206) -- (-333.6424,-887.6329);
          \path[draw=black,line join=round,line cap=round,miter limit=4.00,line
            width=0.826pt] (-333.6436,-867.8961) -- (-266.1806,-856.5883);
          \path[cm={{1.02468,0.0,0.0,1.03996,(-646.57979,-888.93835)}},fill=black,miter
            limit=4.00,line width=0.800pt] (20.4547,64.0021) node[above right] (text4088)
            {$$};
          \path[xscale=0.993,yscale=1.007,fill=black,miter limit=4.00,line width=0.826pt]
            (-632.3549,-821.0549) node[above right] (text4092) {$S^{(0)}$};
          \path[xscale=0.993,yscale=1.007,fill=black,miter limit=4.00,line width=0.826pt]
            (-572.0055,-822.9975) node[above right] (text4092-9) {$B \times \PP^1$};
          \path[xscale=0.993,yscale=1.007,fill=black,miter limit=4.00,line width=0.826pt]
            (-445.2860,-823.6805) node[above right] (text4092-9-6) {$B \times \PP^1$};
          \path[xscale=0.993,yscale=1.007,fill=black,miter limit=4.00,line width=0.826pt]
            (-381.6588,-824.8536) node[above right] (text4092-9-62) {$B \times \PP^1$};
          \path[xscale=0.993,yscale=1.007,fill=black,miter limit=4.00,line width=0.826pt]
            (-306.8173,-824.2672) node[above right] (text4092-9-1) {$T^{(0)}$};
          \path[xscale=0.993,yscale=1.007,fill=black,miter limit=4.00,line width=0.826pt]
            (-499.4819,-836.2325) node[above right] (text4092-9-7) {$\dots$};
          \path[xscale=0.993,yscale=1.007,fill=black,miter limit=4.00,line width=0.826pt]
            (-302.0891,-859.3889) node[above right] (text4092-9-7-2) {$\sigma^{(0)}$};
        \end{scope}
        \begin{scope}[miter limit=4.00,line width=0.695pt]
          \path[draw=black,line join=round,line cap=round,miter limit=4.00,line
            width=0.695pt] (-117.3584,202.4811) -- (-78.3858,192.9652) --
            (-78.3858,263.0547) -- (-117.8246,276.1895);
          \path[draw=black,line join=round,line cap=round,miter limit=4.00,line
            width=0.695pt] (-78.3858,192.9652) -- (-29.3277,203.8351);
          \path[draw=black,line join=round,line cap=round,miter limit=4.00,line
            width=0.695pt] (-78.3858,263.0547) -- (-29.3277,273.9246);
          \path[draw=black,line join=round,line cap=round,miter limit=4.00,line
            width=0.695pt] (-29.3277,203.8351) -- (-29.3277,273.9246);
          \path[draw=black,line join=round,line cap=round,miter limit=4.00,line
            width=0.695pt] (-78.3858,207.5454) -- (-29.3277,218.4154);
          \path[draw=black,line join=round,line cap=round,miter limit=4.00,line
            width=0.695pt] (14.0989,192.9567) -- (63.1570,203.8267) -- (63.1570,273.9161)
            -- (14.0989,263.0462) -- (14.0989,192.9567);
          \path[draw=black,line join=round,line cap=round,miter limit=4.00,line
            width=0.695pt] (63.1570,203.8267) -- (112.2152,191.5148);
          \path[draw=black,line join=round,line cap=round,miter limit=4.00,line
            width=0.695pt] (112.2152,191.5148) -- (112.2152,261.6043) --
            (63.1570,273.9162);
          \begin{scope}[shift={(0,-3.98681)},miter limit=4.00,line width=0.695pt]
            \path[draw=black,line join=round,line cap=round,miter limit=4.00,line
              width=0.695pt] (14.0989,214.4186) -- (63.1570,225.2885);
          \end{scope}
          \begin{scope}[cm={{-1.0,0.0,0.0,1.0,(126.31402,-3.98681)}},miter limit=4.00,line width=0.695pt]
            \path[draw=black,line join=round,line cap=round,miter limit=4.00,line
              width=0.695pt] (14.0989,214.4186) -- (63.1570,225.2885);
          \end{scope}
        \end{scope}
        \path[draw=black,line join=round,line cap=round,miter limit=4.00,line
          width=0.695pt] (-78.3858,207.5454) .. controls (-100.0038,240.7514) and
          (-103.8559,204.6477) .. (-118.0890,240.6737);
    \end{scope}
  \end{scope}
\end{scope}

\end{tikzpicture}
  \caption{The accordion-like central fiber of $(\mathcal X^{(0)}, \mathcal D^{(0)})$}
  \label{fig:accordions1}
\end{figure}
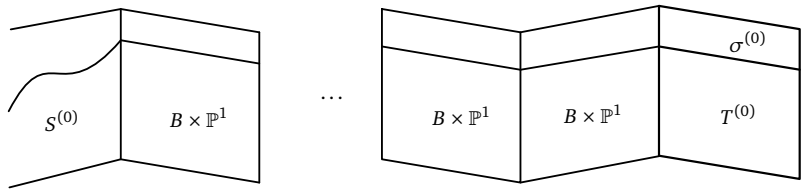

We now apply the non-singular case of \autoref{prop:typeIIflip} repeatedly to pair $\left( \mathcal X^{(0)}, \mathcal D^{(0)}\right)$.
First flip the $(-3)$ curve $\sigma^{(0)}$ by applying \autoref{prop:typeIIflip} to an open subset of $\mathcal X^{(0)}$ containing $\sigma^{(0)}$.
The role of $S$ and $T$is played by $E^{(0)}_n = B\times \PP^1$ and $T^{(0)}$, respectively.
The resulting threefold $\mathcal X^{(1)}$ (see \autoref{fig:accordions2}) has central fiber
\[ X^{(1)} = S^{(1)}\cup E_1^{(1)}\cup \dots \cup E_n^{(1)} \cup T^{(1)}\]
as shown in \autoref{fig:accordions2}, where $S^{(1)} = S^{(0)}$ and $E_i^{(1)} = E_i^{(0)}$ for $i = 1, \dots, n-1$, whereas $E_n^{(1)}$ is a surface with an $A_2$ singularity obtained by three blow ups and two blow downs from $E_n^{(0)}$, and $T^{(1)}$ is obtained from $T^{(0)}$ by contracting $\sigma^{(0)}$.
Note that the transformations $E_n^{(0)} \dashrightarrow E_n^{(1)}$ and $T^{(0)} \rightarrow T^{(1)}$ are simply the transformations $S \dashrightarrow S'$ and $T \rightarrow T'$ from \autoref{prop:typeIIflip}.
The proper transform of $D_n$ on $E_{n}^{(0)}$ is a $(-3)$ curve $\sigma^{(1)}$ on $E_n^{(1)}$.
Note that $\sigma^{(1)}$ lies in the non-singular locus of $E_n^{(1)}$ and $\mathcal X^{(1)}$, away from $T^{(1)}$.

\begin{figure}[ht]
  \centering
  \begin{tikzpicture}[y=0.80pt, x=0.80pt, yscale=-1.000000, xscale=1.000000, inner sep=0pt, outer sep=0pt]
\begin{scope}[shift={(847.3197,900.30695)}]
  \begin{scope}[cm={{1.79394,0.0,0.0,1.36208,(-637.65428,-1189.806)}},miter limit=4.00,line width=1.600pt]
    \begin{scope}[cm={{0.7366,0.0,0.0,0.7366,(-29.29289,72.39284)}}]
        \begin{scope}[cm={{0.73854,0.0,0.0,0.95841,(358.62299,1042.2281)}},miter limit=4.00,line width=0.826pt]
          \path[draw=black,line join=round,line cap=round,miter limit=4.00,line
            width=0.826pt] (-333.6436,-867.8961) -- (-274.3303,-847.7296) --
            (-333.6424,-814.7206) -- (-333.6434,-867.8961);
          \path[cm={{1.02468,0.0,0.0,1.03996,(-646.57979,-888.93835)}},fill=black,miter
            limit=4.00,line width=0.800pt] (20.4547,64.0021) node[above right] (text4088)
            {$$};
          \path[xscale=0.993,yscale=1.007,fill=black,miter limit=4.00,line width=0.826pt]
            (-499.4819,-836.2325) node[above right] (text4092-9-7) {$\dots$};
          \path[xscale=0.993,yscale=1.007,fill=black,miter limit=4.00,line width=0.826pt]
            (-371.3001,-858.2626) node[above right] (text4092-9-7-2) {$\sigma^{(1)}$};
          \path[xscale=0.993,yscale=1.007,fill=black,miter limit=4.00,line width=0.826pt]
            (-633.5479,-822.7681) node[above right] (text4092) {$S^{(0)}$};
          \path[xscale=0.993,yscale=1.007,fill=black,miter limit=4.00,line width=0.826pt]
            (-573.1984,-824.7107) node[above right] (text4092-9) {$B \times \PP^1$};
          \path[xscale=0.993,yscale=1.007,fill=black,miter limit=4.00,line width=0.826pt]
            (-446.4789,-825.3938) node[above right] (text4092-9-6) {$B \times \PP^1$};
          \path[xscale=0.993,yscale=1.007,fill=black,miter limit=4.00,line width=0.826pt]
            (-317.5922,-835.8350) node[above right] (text4092-3) {$T'$};
        \end{scope}
        \begin{scope}[miter limit=4.00,line width=0.695pt]
          \path[draw=black,line join=round,line cap=round,miter limit=4.00,line
            width=0.695pt] (-117.3584,202.4811) -- (-78.3858,192.9652) --
            (-78.3858,263.0547) -- (-117.8246,276.1895);
          \path[draw=black,line join=round,line cap=round,miter limit=4.00,line
            width=0.695pt] (-78.3858,192.9652) -- (-29.3277,203.8351);
          \path[draw=black,line join=round,line cap=round,miter limit=4.00,line
            width=0.695pt] (-78.3858,263.0547) -- (-29.3277,273.9246);
          \path[draw=black,line join=round,line cap=round,miter limit=4.00,line
            width=0.695pt] (-29.3277,203.8351) -- (-29.3277,273.9246);
          \path[draw=black,line join=round,line cap=round,miter limit=4.00,line
            width=0.695pt] (-78.3858,207.5454) -- (-29.3277,218.4154);
          \path[draw=black,line join=round,line cap=round,miter limit=4.00,line
            width=0.695pt] (14.0989,192.9567) -- (63.1570,203.8267) -- (63.1570,273.9161)
            -- (14.0989,263.0462) -- (14.0989,192.9567);
          \path[draw=black,line join=round,line cap=round,miter limit=4.00,line
            width=0.695pt] (63.1570,203.8267) -- (95.0159,180.1886) --
            (112.4312,210.6997);
          \path[draw=black,line join=round,line cap=round,miter limit=4.00,line
            width=0.695pt] (112.2152,210.4318) -- (112.2152,261.6043) --
            (63.1570,273.9162);
          \begin{scope}[shift={(0,-3.98681)},miter limit=4.00,line width=0.695pt]
            \path[draw=black,line join=round,line cap=round,miter limit=4.00,line
              width=0.695pt] (14.0989,214.4186) -- (63.1570,225.2885);
          \end{scope}
          \begin{scope}[cm={{-1.0,0.0,0.0,1.0,(126.31402,-3.98681)}},miter limit=4.00,line width=0.695pt]
            \path[draw=black,line join=round,line cap=round,miter limit=4.00,line
              width=0.695pt] (23.5148,197.8117) -- (63.1570,225.2885);
            \begin{scope}[miter limit=4.00,line width=0.695pt]
            \end{scope}
          \end{scope}
        \end{scope}
        \path[draw=black,line join=round,line cap=round,miter limit=4.00,line
          width=0.695pt] (-78.3858,207.5454) .. controls (-100.0038,240.7514) and
          (-103.8559,204.6477) .. (-118.0890,240.6737);
    \end{scope}
  \end{scope}
\end{scope}

\end{tikzpicture}
  \caption{A modified accordion after a $(-3)$ flip}
  \label{fig:accordions2}
\end{figure}
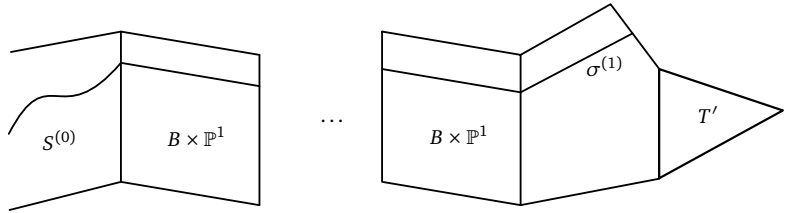

Once more, flip the $(-3)$ curve $\sigma^{(1)}$ by applying \autoref{prop:typeIIflip} to an open subset of $\mathcal X^{(1)}$ containing $\sigma^{(1)}$.
Now the role of $S$ and $T$ is played by $E^{(1)}_{n-1}$ and $E^{(1)}_{n}$, respectively.
The resulting threefold $\mathcal X^{(2)}$ has central fiber
\[ X^{(2)} = S^{(2)}\cup E_1^{(2)}\cup \dots \cup E_{n-1}^{(2)} \cup E_n^{(2)} \cup T^{(2)}, \]
where the only components that are different from their previous counterparts are $E_{n-1}^{(2)}$ and $E_n^{(2)}$.
The surface $E_{n-1}^{(2)}$ is obtained by three blow ups and two blow downs from $E_{n-1}^{(1)}$, and $E_n^{(2)}$ is obtained from $E_n^{(1)}$ by contracting $\sigma^{(1)}$.
The proper transform of $D_{n-1}$ on $E_{n-1}^{(1)}$ is a $(-3)$ curve $\sigma^{(2)}$ on $E_{n-1}^{(2)}$, which lies in the non-singular locus of $E_{n-1}^{(2)}$ and $\mathcal X^{(2)}$, and away from $E_n^{(2)}$.

Continue flipping the $(-3)$ curves $\sigma^{(i)}$, for $i = 2, 3, \dots, n$, resulting in a threefold $\mathcal X^{(n+1)}$ which has central fiber
\[ X^{(n+1)} = S^{(n+1)}\cup E_1^{(n+1)}\cup \dots \cup E_n^{(n+1)} \cup T^{(n+1)}.\]
Note that we now have $S^{(n+1)} \cong S'$,  obtained by three blow ups and two blow downs from $S$ as described in \autoref{prop:typeIIflip}, and $T^{(n+1)} \cong T'$, obtained by contracting the $(-3)$ curve $\sigma$ on $T$.
The intermediate components $E_i^{(n+1)}$ are obtained by 3 blow-ups and 3 blow-downs on $E_i^{(0)} \cong B \times \PP^1$.
Notice that the curves that are blown down are contracted under the map to $B$.
As a result, the projection map $E_i^{(0)} \to B$ survives as a regular map $E_i^{(n+1)} \to B$.

Note that $K_{X^{(n+1)}}+wD^{(n+1)}$ is nef but not ample on $E_i^{(n+1)}$; it is trivial on the fibers of $E_i^{(n+1)} \to B$.
Recall that $\mathcal X \to \mathcal Y$ is the contraction of $\sigma$, and we have a projective morphism $\psi \from \mathcal X^{(n+1)} \to \mathcal Y$.
The bundle $K_{\psi}+w\mathcal D$ is nef, hence semi-ample by the abundance theorem \cite[Thm 1.1]{kee.mat.mck:94}.
It gives a divisorial contraction $\mathcal X^{(n+1)} \to \mathcal X'$ in which all $E_i^{(n+1)}$ are contracted to $B$.

Let $\mathcal D'$ be the image of $\mathcal D^{(n+1)}$.
By construction $\mathcal X'$ is $\Q$-factorial with canonical singularities and the central fiber $(X', D')$ is as required in \autoref{prop:typeIIflip}.
The proof of \autoref{prop:typeIIflip} is now complete.

%%% Local Variables:
%%% mode: latex
%%% TeX-master: "main"
%%% End:

\section{Stable replacements of unstable pairs}\label{sec:stable-replacement}
% Use the section above to describe stable replacements of the unstable pairs.
In this section, we verify the valuative criterion of properness of the moduli stack of stable log quadrics $\mathfrak X$, which enhances \autoref{lem:valprop}.

Let $\Delta$ be a DVR and $(\mathcal X_{\eta},\mathcal D_{\eta}) \cong ((\PP^1 \times \PP^1)_{\eta},C)$ be a stable log surface over the generic point $\eta$ of $\Delta$ where $C$ is a smooth curve of bi-degree $(3,3)$. Possibly after a finite base change, $(\mathcal X_{\eta},\mathcal D_{\eta})$ extends to a family $(\overline{\mathcal X}, \overline{\mathcal D}) \rightarrow \Delta$ such that the central fiber $(\overline{X},\overline{D})$ is a stable log surface and both $K_{\overline{\mathcal{X}}/\Delta}$ and $\overline{\mathcal D}$ are $\Q$-Cartier by \autoref{lem:valprop}. We describe all possible $(\overline{X},\overline{D})$ in subsequent subsections and verify that they all satisfy the index condition \autoref{def:index}.
This confirms that $(\overline{\mathcal X}, \overline{\mathcal D}) \in \mathfrak{X}(\Delta)$, and shows the valuative criteria of properness of $\mathfrak X$.
We do this explicitly and independently of the proof sketched after \autoref{lem:valprop}.

Consider $\phi \from \mathcal D_{\eta} \to \PP^1_{\eta}$ induced by the first projection $(\PP^1 \times \PP^1)_{\eta} \rightarrow \PP^1_{\eta}$ as a $\eta$-valued point of $\mathcal H^3_4$.
Let $\mathcal C \to \mathcal P$ be its unique extension to a $\Delta$-valued point of $\overline{\mathcal H}^3_4(1/6+\epsilon)$, possibly after a base-change.
Note that $\mathcal P \to \Delta$ is an orbi-nodal curve of genus $0$.
Let $\mathcal E$ be the Tschirnhausen bundle of $\phi \from \mathcal C \to \mathcal P$.
By the procedure described in \autoref{sec:trigonal-surface}, $\phi$ gives a divisor $D(\phi)$ in $\PP \mathcal E$.
Let $(\mathcal X, \mathcal D)$ be the coarse space of $(\PP \mathcal E, \ell(\phi))$.
Let $(X, D)$ be the fiber of $(\mathcal X, \mathcal D)$ over the closed point $0 \in \Delta$.
By \autoref{prop:genus4}, the fibers of $(\mathcal X, \mathcal D) \to \Delta$ are semi-stable log quadric surfaces.
By construction, the general fiber is also stable, but the special fiber need not be.
Since $\mathcal X$ is $\Q$-factorial and the central fiber $(X,D)$ satisfies the index condition by \autoref{rem:Ssing}, the family $(\mathcal X, \mathcal D) \to \Delta$ is $\Q$-Gorenstein \autoref{lem:QgorDVR}.

\begin{theorem}[(Stabilization)]
  \label{prop:stab}
  Let $(\mathcal X, \mathcal D) \to \Delta$ be as above.
  There exists a $\Q$-Gorenstein family $(\overline {\mathcal X}, \overline{\mathcal D}) \to \Delta$ of stable log quadrics with generic fiber $(\mathcal X_{\eta},  \mathcal D_{\eta})$ on $\eta$.
  Furthermore, the central fiber $(\overline X, \overline D)$ of $(\overline {\mathcal X}, \overline {\mathcal D}) \to \Delta$ depends only on the central fiber $(X, D)$ of the original family $(\mathcal X, \mathcal D) \to \Delta$.
\end{theorem}
We highlight that, after obtaining the semi-stable family $(\mathcal X, \mathcal D)$, a further base change is not necessary to get to the stable family.

\begin{proof}[Outline of proof of \autoref{prop:stab}]
  If $K_X + (2/3 + \epsilon)D$ is ample for some $\epsilon > 0$, then $(\overline {\mathcal X}, \overline{\mathcal D}) = (\mathcal X, \mathcal D)$, and there is nothing to prove.
  The end of \autoref{sec:trigonal-surface} lists the possibilities for $C \to P$ for which $K_X+(2/3 + \epsilon)D$ fails to be ample for all $\epsilon > 0$.
  In all these cases, we construct $(\overline{\mathcal X}, \overline{\mathcal D})$ from $(\mathcal X, \mathcal D)$ by explicitly running a minimal model program on the threefold $\mathcal X$ using the birational transformations described in \autoref{sec:flips}.
  This program consists of the following two steps.
  \begin{asparaenum}
  \item[Step 1 (Flips): ]
    By a sequence of flips on the central fiber of $\mathcal X$, we construct $(\mathcal X', \mathcal D') \to \Delta$ with slc fibers and $\Q$-factorial total space $\mathcal X'$ such that $K_{\mathcal X'} + (2/3 + \epsilon) \mathcal D'$ is $\Q$-Cartier and nef for all sufficiently small $\epsilon > 0$.
    Our construction shows that the central fiber of $(\mathcal X', \mathcal D')$ depends only on the central fiber of $(\mathcal X, \mathcal D)$.
  \item[Step 2 (Contractions): ]
    Set $w = 2/3 + \epsilon$, where $\epsilon > 0$ is such that $K_{\mathcal X'} + (2/3 + \epsilon) \mathcal D'$ is nef.
    The divisor $K_{\mathcal X'} + w \mathcal D'$ is semi-ample.
    This follows from the log abundance theorem on threefolds \cite[Theorem 1.1]{kee.mat.mck:94}.
    For a more elementary argument, note that the pair $(\mathcal X', (2/3 + \epsilon) \mathcal D')$ is KLT since it is a locally stable pair over $\Delta$ with a KLT generic fiber (see \cite[Proposition~2.13]{kol:17}).
    We then apply the base-point free theorem \cite[Theorem~3.3]{kol.mor:98}.
    We set 
    \[\overline {\mathcal X} = \Proj \left( \bigoplus_{n \geq 0} H^0\left( \mathcal X', n\left( K_{\mathcal X'} + w\mathcal D' \right) \right) \right), \]
    and let $\overline{\mathcal D}$ be the image of $\mathcal D'$ in $\overline{\mathcal X}$.
    For this step, it is clear that the central fiber $(\overline X, \overline D)$ of $(\overline {\mathcal X}, \overline {\mathcal D}) \to \Delta$ depends only on the central fiber of $({\mathcal X'}, {\mathcal D'})$.
    We describe $(\overline X, \overline D)$ explicitly, culminating in the classification in \autoref{tab:list_of_all_pairs}.
    It is easy to check from the description that $(\overline X, w \overline D)$ is slc.
    We also observe that $\overline{D} \subset \overline{X}$ is a Cartier divisor that stays away from the non-Gorenstein singularities of $\overline{X}$. Hence, $(\overline{X},\overline{D})$ satisfies the index condition.
    Furthermore, by construction, both $K_{\overline {\mathcal X}}$ and $\overline {\mathcal D}$ are $\Q$-Cartier divisors, so the family $(\overline {\mathcal X}, \overline {\mathcal D}) \to \Delta$ is $\Q$-Gorenstein by \autoref{lem:QgorDVR}.
  \end{asparaenum}
  
  To complete the proof of \autoref{prop:stab}, we must carry out the two steps in each case listed at the end of \autoref{sec:trigonal-surface}.
  We do this in separate subsections that follow. 
\end{proof}  

\begin{remark}
  It is possible to show directly that $K_{\mathcal X'}+w\mathcal D'$ is semi-ample, without appealing to any sophisticated result in birational geometry, using the particular geometry of $(\mathcal X', \mathcal D')$.
  This is easy but tedious, so we have chosen to appeal to the general results instead.
\end{remark}

\subsection{Maroni special covers}\label{subsec:MaroniSpecial}
Suppose $C \to P$ is as in case \eqref{maroni4} of the unstable list from page~\pageref{maroni4}.
That is, $P \cong \PP^1$ and $C \to P$ is Maroni special.
In this case, $X \cong \F_2$ and $D \subset X$ is a divisor of class $3 \sigma + 6F$.

\subsubsection*{Step 1 (Flips):}
In this case, $K_{\mathcal X}+w{\mathcal D}$ is already nef, so we do not need any flips.
\subsubsection*{Step 2 (Contractions):}
The only $(K_X + w D)$-trivial curve is $\sigma$.
The contraction step contracts $\sigma \subset X$ to a point, resulting in $\overline X$ isomorphic to the weighted projective plane $\PP(1,1,2)$.

There are two possibilities on how the curve $\overline D$ interacts with the unique singular point $p \in \overline X$.
The first possibility is that $D \subset X$ is disjoint from $\sigma$.
In this case, $\overline D$ is away from the singularity.
The second possibility is that $D \subset X$ contains $\sigma$ as a component.
In this case $D = \sigma \cup E$, where $E$ does not contain $\sigma$ and $E \cdot \sigma = 2$.
Therefore, $\overline D = \overline E$; this passes through the singularity of $X$ and has either a node or a cusp there, depending on whether $E$ intersects $\sigma$ transversely at 2 points or tangentially at 1 point.
The two steps in required for the proof of \autoref{prop:stab} are thus complete.

\subsection{Hyperelliptic covers}\label{subsec:hypellip}
Suppose $C \to P$ is as in case \eqref{hyperell4} of the unstable list from page~\pageref{hyperell4}.
That is, $P \cong \PP^1$ and $C = \PP^1 \cup H$, where $H$ is a hyperelliptic curve of genus $4$ attached nodally to $\PP^1$ at one point.
In this case, $S \cong \F_4$ and $D$ is the union of $\sigma \cong \PP^1$ and a divisor of class $2 \sigma + 9F$ isomorphic to $H$; we denote the divisor also by the letter $H$.
Note that $H$ intersects $\sigma$ at a unique point, say $p$.

\subsubsection*{Step 1 (Flips):}
Let $(\mathcal X', \mathcal D')$ be the family obtained from $(\mathcal X, \mathcal D)$ by flipping the $-4$ curve $\sigma$;
this flip is constructed in \autoref{sec:flip1}.
Let $(X', D')$ be the central fiber of $(\mathcal X', \mathcal D') \to \Delta$.
Recall that the relationship between $X$ and $X'$ is given by the diagram
\begin{equation}\label{eqn:F4blowupdown}
	X \leftarrow \widetilde X \rightarrow X'
\end{equation}
where $\widetilde X$ is obtained from $X$ by blowing up the point $p$ and the intersection point $q$ of the proper transform of $H$ and the exceptional divisor of the first blowup.
Let $F$ be the fiber of $X \to \PP^1$ through $p$.
Denote the proper transforms of $\sigma$, $F$, and $H$ by the same letters, and denote by $E_1$ and $E_2$ be the exceptional divisors of the two blow-ups.
Then $\widetilde X \to X'$ is obtained by contracting $\sigma$ and $E_1$.

\subsubsection*{Step 2 (Contractions):}
There are three possibilities for the ramification behavior of $H \to \PP^1$ at $p$, which dictate the result of the contraction step.
To analyze the contracted curves, it is necessary to look at the configuration of the curves $\{\sigma, H, F, E_1, E_2\}$ on $\widetilde X$, which we encode by its dual graph.

\begin{asparaenum}[\sl {Case} 1:]
\item {$H \to \PP^1$ is unramified at $p$.}\label{subsubsec:F4unram}
  
  In this case, $K_{\mathcal X'} + w \mathcal D'$ is ample, and hence $(\overline {\mathcal X}, \overline {\mathcal D}) = (\mathcal X', \mathcal D')$.
  For the ampleness, let us compute the dual graph of $\{\sigma, H, F, E_1, E_2\}$ on $\widetilde X$.
  On $X$, the curve $H$ intersects $F$ transversely at two distinct points, one of which is $p$, and which is also the transverse intersection point of $\sigma$ with $H$ and $F$.
  This configuration of curves is changed as follows by the blow ups:
  
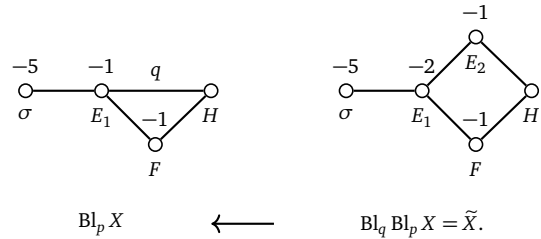
\[
    \begin{tikzpicture}
      \draw
      (0,0) node [curve, label=below:$\sigma$,label=above:$-5$] (sigma) {}
      (1,0) node [curve, label=below:$E_1$, label=above:$-1$] (E1) {}
      (1,0) + (-45:1)  node [curve, label=below:$F$, label=above:$-1$] (F) {}
      (1,0) + (0:1.44) node [curve, label=below:$H$] (H) {}
      (sigma) edge (E1) (E1) edge (F) (F) edge (H) (E1) edge node[above]{$q$} (H)
      (1,-1.75) node (1) {$\Bl_pX$};
      \begin{scope}[xshift=120]
        \draw
        (0,0) node [curve, label=below:$\sigma$,label=above:$-5$] (sigma) {}
        (1,0) node [curve, label=below:$E_1$, label=above:$-2$] (E1) {}
        (1,0) + (45:1) node [curve, label=below:$E_2$, label=above:$-1$] (E2) {}
        (1,0) + (-45:1)  node [curve, label=below:$F$, label=above:$-1$] (F) {}
        (1,0) + (0:1.44) node [curve, label=below:$H$] (H) {}
        (sigma) edge (E1) (E1) edge (E2) (E1) edge (F) (F) edge (H) (H) edge (E2)
        (1,-1.75) node (2) {$\Bl_q {\Bl_p X} = \widetilde X$.};
      \end{scope}
      \draw[shorten <= 1cm, shorten >= 1cm] (1) edge [<-] (2);
    \end{tikzpicture}
  \]
  Let $E_2'$ and $F'$ be the image in $X'$ of $E_2$ and $F$ on $\widetilde X$.
  From the dual graph above, we obtain the following intersection table on $X'$.
  \[
    \begin{array}{l | r r r r}
      & E'_2 & F' & K_{X'} & D'\\
      \hline
      {E}'_2 & -4/9 & 5/9 & -2/3 &1\\
      F' & 5/9 & -4/9 & -2/3 &1\\
    \end{array}
  \]
  Let us explain how this table is computed.
  Let $\beta \from \widetilde X \to X'$ the contraction of $E_1$ and $\sigma$.
  For a divisor class on $X'$, we first compute its pullback to $\widetilde X$, using that the pullback pairs to zero with $E_1$ and $\sigma$.
  We then use the dual graph above to compute the intersection numbers on $\widetilde X$.
  For example, we have
  \[ \beta^* E_2' = \frac{1}{9} \sigma + \frac{5}{9} E_1 + E_2,\]
  which gives $E_2' \cdot E_2' = \beta^* E_2' \cdot \beta^* E_2' = -4/9$.
  The other numbers are computed similarly.
  Since $X'$ is a $\Q$-factorial surface of Picard rank 2, and $E_2'$ and $F'$ have negative self-intersection, they must span $\overline{\NE}(X')$.
  We have
  \begin{align*}
    (K_{X'} + w D')  \cdot E'_2 = (K_{X'} + w D')  \cdot F' = \epsilon > 0,
  \end{align*}
  so $K_{\mathcal X'} + w \mathcal D'$ is ample.
  Note that the surface $\overline X = X'$  has a $\frac{1}{9}(1,2)$ singularity obtained by contracting the chain $(\sigma, E_1)$.
  The divisor $\overline D = D'$ stays away from the singularity.
  
  \begin{remark}\label{surf_from_hypellip_1}
  	Suppose in addition that $H$ is smooth. 
  	Then the hyperelliptic involution of $H$ extends to an automorphism of $\widetilde X$.
  	This automorphism fixes $\sigma$ point-wise, leaves $E_1$ invariant, interchanges $E_2$ and $F$, and sends $p$ to $r$.
  	Hence, it descends to an automorphism on $\overline X$ that interchanges the two extremal rays $\overline E_2$ and $\overline F$ of the $\overline{\mathrm{NE}}(\overline{X})$. Therefore, $(\overline{X},\overline{D})$ is determined from $H=\overline{D}$ with a hyperelliptic divisor $p+r=(\overline E_2+\overline F)|_H$, and $H$ is on the non-singular locus of $\overline{X}$.
  \end{remark}

\item{$H \to \PP^1$ is ramified at $p$.}\label{subsubsec:F4ram}

  In this case, the dual graph on $\widetilde X$ is
  
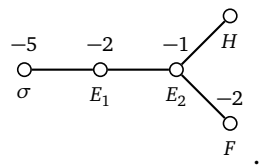
\[
  \begin{tikzpicture}
    \draw
    (0,0) node [curve, label=below:$\sigma$,label=above:$-5$] (sigma) {}
    (1,0) node [curve, label=below:$E_1$, label=above:$-2$] (E1) {}
    (2,0) node [curve, label=below:$E_2$, label=above:$-1$] (E2) {}
    (2,0) + (-45:1) node [curve, label=below:$F$, label=above:$-2$] (F) {}
    (2,0) + (45:1) node [curve, label=below:$H$] (H) {}
    (sigma) edge (E1) (E1) edge (E2) (E2) edge (F) (E2) edge (H);
  \end{tikzpicture}.
\]
As in case~\ref{subsubsec:F4unram}, we get that $\overline{\mathrm{NE}}(X')$ is spanned by the images ${E_2'}$ and $F'$ of $E_2$ and $F$, and we have
\begin{align*}
  (K_{X'} + w D')  \cdot E'_2   = \epsilon > 0, \text{ and } (K_{X'} + w D')  \cdot F' = 0.
\end{align*}
Therefore, the contraction step contracts $F'$, resulting in an $A_1$ singularity on $\overline X$.
The divisor $\overline D$ stays away from the singularity.

\begin{remark}\label{surf_from_hypellip_2}
	Similarly to \autoref{surf_from_hypellip_1}, when $H$ is smooth, $(\overline{X},\overline{D})$ is determined from a hyperelliptic curve $H$ with a hyperelliptic divisor $2p$, and $H$ is on non-singular locus of $\overline{X}$.
\end{remark}

\item{$H$ contains $F$ as a component.}\label{subsubsec:f4comp}

In this case, let $H = F \cup G$, where $G$ is the residual curve.
We have the dual graph on $\widetilde X$:
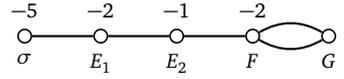
\[
  \begin{tikzpicture}
    \draw
    (0,0) node [curve, label=below:$\sigma$,label=above:$-5$] (sigma) {}
    (1,0) node [curve, label=below:$E_1$, label=above:$-2$] (E1) {}
    (2,0) node [curve, label=below:$E_2$, label=above:$-1$] (E2) {}
    (3,0) node [curve, label=below:$F$, label=above:$-2$] (F) {}
    (4,0) node [curve, label=below:$G$] (G) {}
    (sigma) edge (E1) (E1) edge (E2) (E2) edge (F) (F) edge[bend right=30] (G) (F) edge[bend left=30] (G);
  \end{tikzpicture}.
\]
For the same reason as in case~\ref{subsubsec:F4ram}, $K_{X'}+wD'$ is nef and it contracts the curve $F$, resulting in a surface $\overline X$ with an $A_1$ singularity.
Note, however, that the divisor $\overline D$--which is the image of $G$--passes through the $A_1$ singularity.
If $F$ intersects $G$ transversely in 2 distinct points, then $\overline D$ has a node at the $A_1$ singularity.
If $F$ intersects $G$ tangentially at 1 point, then $\overline D$ has a cusp at the $A_1$ singularity.
\end{asparaenum}
The two steps in required for the proof of \autoref{prop:stab} are now complete.

\begin{remark}\label{rmk:F4_surface}\label{rmk:nontoric}\label{rmk:P(9,1,2)}
  We record some properties of $\overline X$ and $\overline D$ obtained in each case.
  In case \ref{subsubsec:F4unram}, $\overline X$ is obtained from $X$ by two blow-ups and and two blow-downs.
  The blow-ups use the auxiliary data of the point $p$ on the hyperelliptic component $H$ of $D$ and the tangent direction to $H$ at $p$; the blow-downs do not require any auxiliary data.
    The automorphism group of $X$ acts transitively on the necessary auxiliary data, and therefore the isomorphism class of $\overline X$ is independent of $D$.
  The blow-downs result in a unique singular point on $\overline X$ corresponding to a $\frac19(1,2)$ singularity.
  It is easy to see that $\overline X$ is not a toric surface.
  In case~\ref{subsubsec:F4ram} and case~\ref{subsubsec:f4comp}, the tangent direction to $H$ at $p$ is along the fiber of $X$ through $p$.
  As a result, $\overline X$ is a toric surface.
  More precisely, it is easy to figure out that $\overline X$ is isomorphic to $\PP(1,2,9)$.
\end{remark}

\begin{remark}\label{note:QGorF4}
  Observe that the covers $C \to P$ in cases \ref{subsubsec:F4ram} and \ref{subsubsec:f4comp} are specializations of the covers in case~\ref{subsubsec:F4unram}.
  By considering the family of surfaces $\overline X$ in such a specialization, we see that the non-toric surface $\overline X$ in case~\ref{subsubsec:F4unram} specializes to $\PP(1,2,9)$.
  In other words, $\overline X$ is a smoothing of the $A_1$ singularity on $\PP(1,2,9)$.
  We can check that in this family of surfaces, both $K$ and $\mathcal D$ are $\Q$-Gorenstein, so the family is a $\Q$-Gorenstein family (\autoref{lem:QgorDVR}).
\end{remark}

\subsection{The $\F_3-\F_3$ case}\label{sec:replace-f33}
Suppose we are in case \eqref{f33} of the unstable list from page~\pageref{f33}.
That is, $C = C_1 \cup C_2$ mapping to $P = \PP^1 \cup \PP^1$, where $C_i$ is the disjoint union of $\PP^1$ and a hyperelliptic curve $H_i$ of genus 2.
In this case, $X = X_1 \cup X_2$, where $X_i \cong \F_3$ and $D_i \subset X_i$ is the disjoint union of the directrix $\sigma_i$ and a curve $H_i$ of class $2 \sigma_i + 6F$.
Since $C$ is connected, we note that $\sigma_1$ intersects $X_2$ and is disjoint from $\sigma_2$, and vice-versa.

\subsubsection*{Step 1 (Flips):}
Let $(\mathcal X', \mathcal D')$ be the family obtained from $(\mathcal X, \mathcal D)$ by flipping the $-3$ curves $\sigma_1$ and $\sigma_2$.
Let $(X', D')$ be the central fiber of $(\mathcal X' \to \mathcal D') \to \Delta$.
The surface $X'$ is the union of two components $X'_1 \cup X'_2$, where each $X_i'$ is related to $X_i$  by a diagram
\[ X_i \leftarrow \widetilde X_i \rightarrow X_i'.\]
This diagram is given by \autoref{fig:type2flip}; the role of $S$ and $T$ is played by $X_1$ and $X_2$ while flipping $\sigma_2$ and by $X_2$ and $X_1$ while flipping $\sigma_1$.
To recall, $\widetilde X_i \to X_i$ is the blow-up of $X_i$ three times, first at $D_i \cap \sigma_i$, and two more times at the intersection of proper transform of $D_i$ and
the most recent exceptional divisor.
Denote the exceptional divisor of the $j$th blowup by $E_{ij}$ for $j = 1, 2, 3$; use the same letters to denote proper transforms; and denote by $F$ the curve $X_1 \cap X_2$.
Then, $\widetilde X_i \to X_i'$ is the blow down of $E_{i1}$, $E_{i2}$, and $\sigma_i$.
Note that $X_i'$ has a $\mu_3$ singularity at the image point of $\sigma_i$, and an $A_2$ singularity at the image point of $E_1 \cup E_2$; it is smooth elsewhere.

\subsubsection*{Step 2 (Contractions):}
In this case, $K_{\mathcal X'}+w \mathcal D'$ is already ample, so that $(\overline {\mathcal X}, \overline {\mathcal D}) = (\mathcal X', \mathcal D')$.
For the ampleness, we must show that $K_{X'}+wD'$ is positive on $\overline{NE}(X_i')$ for $i = 1,2$.
As in \autoref{subsec:hypellip}, the dual graph of the configuration of curves $\{\sigma_i, E_1, E_2, E_3, H_i\}$ on $\widetilde X_i$ is 
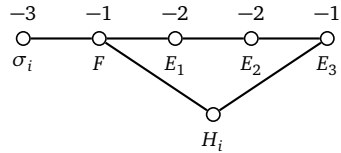
\[
  \begin{tikzpicture}
    \draw
    (0,0) node [curve, label=below:$\sigma_i$,label=above:$-3$] (sigma) {}
    (1,0) node [curve, label=below:$F$, label=above:$-1$] (F) {}
    (2,0) node [curve, label=below:$E_1$, label=above:$-2$] (E1) {}
    (3,0) node [curve, label=below:$E_2$, label=above:$-2$] (E2) {}
    (4,0) node [curve, label=below:$E_3$, label=above:$-1$] (E3) {}
    (2.5,-1) node [curve, label=below:$H_i$] (H) {}
    (sigma) edge (F) (F) edge (E1) (E1) edge (E2) (E2) edge (E3) (H) edge (F) (H) edge (E3)
    ;
  \end{tikzpicture}.
\]
Denote by $F'$ and $E_3'$ the images in $X_i'$ of $F$ and $E_3$, respectively.
Using the dual graph above, we get the following intersection table on $X_i'$:
\[
  \begin{array}{l | r r r r}
    & E'_3 & F' & K & D'\\
    \hline
    {E}'_3 & -1/3 & 1/3 & -1 & 1 \\
    F' & 1/3 & -5/6 & 1/3 & 1\\
  \end{array}
\]
Since $X_i'$ is of Picard rank 2, and the two curves $F'$ and $E_3'$ have negative self-intersection, they generate $\overline{NE}(X_i')$.
Now we compute
\begin{align*}
  (K_{X'} + w D') \cdot E'_3 = \epsilon > 0 \text{, and } (K_{X'} + w D') \cdot F' = 1/6 + \epsilon > 0.
\end{align*}
Hence $K_{X'} + w D'$ is ample on $X_i$ for $i = 1, 2$.

The two steps required in the proof of \autoref{prop:stab} are now complete.

\begin{remark}\label{rmk:F3F3}
  We record some properties of the $(\overline X, \overline D)$ we found above.

  First, note that $\overline X$ is determined from $X$ by two length 3 subschemes of $X_1$ and $X_2$, namely the curvilinear subschemes of length 3 on $H_1$ and $H_2$ supported at $\sigma_1 \cap H_2$ and $\sigma_2 \cap H_1$.
  All such data are equivalent modulo the action of the automorphism group of $X$.
  Therefore, the isomorphism type of $\overline X$ is uniquely determined.

  Second, note that the two components of $\overline X$ are toric.
  To see this, note that there are toric structures on the surfaces $X_i$'s such that the $\sigma_i$'s and the curvilinear subschemes of length 3 above are torus fixed.
  Tracing through the transformation of $X$ to $\overline{X}$, we see $\overline{X}$ may be represented as the a degenerate (non-normal) toric surface represented by the union of the quadrilaterals $\langle  (-3,-2), (-3,-1), (3,1), (3,-2) \rangle$ and $\langle  (-3,-1), (-3, 2), (3,2), (3,1) \rangle$ (see \autoref{fig:toricXbar}).
  \begin{figure}[hb]
    \centering
    \begin{tikzpicture}[scale=0.5]
      \draw[thick] (-3,-2) -- (-3,-1) -- (3,1) -- (3,-2) -- cycle;
      \draw[thick] (-3,-1) -- (-3,2) -- (3,2) -- (3,1) -- cycle;
      \foreach \i in {(-3,-2), (-3,-1), (-3,2)}
      \draw \i node [left] {\scriptsize $\i$};
      \foreach \i in {(3,-2), (3,1), (3,2)}
      \draw \i node [right] {\scriptsize $\i$};
    \end{tikzpicture}
    \caption{The non-normal toric surface $\overline X$ obtained in th $\F_3-\F_3$ case}
    \label{fig:toricXbar}
  \end{figure}

  Finally, let $p, q \in \overline X$ be the images of $\sigma_1, \sigma_2$, respectively.
  Then $\overline X$ has the singularity type $(xy = 0) \subset \frac13(1,2,1)$ at $p$, and $(xy=0) \subset \frac13(2,1,1)$ at $q$.
  \begin{comment}
  In these coordinates, the divisor $\overline D$ is given by $z = 0$.
  \todo{Check this.}
  \end{comment}
\end{remark}

It turns out that $\overline{X}$ also appears as a stable limit in a different guise.
\begin{proposition}\label{lem:F33asF_1/3_1/3}
  $(\overline{X},\overline{D})$ is isomorphic to a log surface appearing in \eqref{f_1/3_1/3} of the stable list from page \pageref{f_1/3_1/3}.
\end{proposition}
\begin{proof}
  Let $\overline{F}=\overline{X}_1\cap \overline{X}_2$, and let $x_i \in H_i$ to be the point of $X_i$ that gets blown up 3 times in the construction of $\overline X_i$ from $X_i$.

  Let $f$ be the class of in $\overline{X}_i$ of the image of the proper transform of a section $\tau$ of $X_i \cong \F_3$ which is triply tangent to $H_i$ at $x_i$ and satisfies $\tau^2 = 3$.
  Then we have
  \[ f^2 = 0, \quad f \cdot \overline F = 0, \text{ and } \overline D |_{\overline X_i} \cdot f = 3.\]
  Moreover, there is a 1-parameter families of such sections $\tau$, and the proper transforms of different sections yield disjoint images in $\overline X_i$.
  The section $\sigma_i + 3F$ is a particular such $\tau$.
  The image of its proper transform is the divisor $3 \overline F$.
  Thus, the line bundle associated to $f$ is base-point free.
  It induces a map
  \[ \pi_i \from \overline X_i \to \PP^1,\]
  which is generically a $\PP^1$-fibration.

  Define the stack $Y_i$ by 
  \[ Y_i = \left[ \spec\left( \bigoplus_{n \in \Z}\O_{\overline X_i}\left(n \overline F\right)\right) \big / \Gm \right].\]
  The natural map $Y_i \to \overline X_i$ is the coarse space map, and the divisorial pullback of $\O_{\overline X_i}(\overline F)$ to $Y_i$ is Cartier.
  A simple local calculation shows that over the two singular points of $\overline X_i$, the map $Y_i \to \overline X_i$ has the form
  \[ [\spec \k[x,y] \big/ \mu_3 ] \to \spec \k[x,y]/\mu_3.\]

  Let $0 \in \PP^1$ be the image of $\overline F \subset \overline X_i$.
  Set $P_i = \PP^1(\sqrt[3]{0})$.
  Since the scheme theoretic pre-image of $0$ is $3 \overline F$, which is 3 times a Cartier divisor on $Y_i$, the natural map $Y_i \to \PP^1$ gives a map $\pi \from Y_i \to P_i$.
  It is easy to check that $Y_i \to P_i$ is the $\PP^1$-bundle
\[ Y_i = \PP (\O(5/3) \oplus \O(4/3)).\]
Note that $\overline D_i \subset \overline X_i$ lies away from the singularities of $X_i$.
Hence, it gives a divisor on $Y_i$, which we denote by the same symbol.
We also see that $\overline D_i \to P_i$ is of the divisor class $\O(3) \otimes \pi^* \O(-3)$,
and therefore is obtained from the Tschirnhausen construction from a triple cover $C_i \to P_i$.
Putting together the triple covers for $i = 1, 2$, we obtain an element $\phi \from C_1 \cup C_2 \to P_1 \cup P_2$ of type \eqref{f_1/3_1/3} on page~\pageref{f_1/3_1/3}.
\end{proof}
The description of $\overline X$ as the degenerate toric surface given by the subdivided polytope in \autoref{fig:toricXbar} can also be obtained using the alternate description of $\overline X$ as the coarse space of $\PP(\O(4/3, 5/3) \oplus \O(5/3, 4/3))$ obtained in \autoref{lem:F33asF_1/3_1/3}.

\subsection{The $\F_1-\F_1$ case}\label{sec:f1-f1}
Suppose we are in case \eqref{f11} of the unstable list from page~\pageref{f11}.
That is, $C = C_1 \cup C_2$ mapping to $P = \PP^1 \cup \PP^1$, where $C_i$ is a connected curve of genus 1; $X_i \cong \F_1$; and $D_i \subset X_i$ is a divisor of class $3 \sigma_i + 3F$ intersecting the fiber $X_1 \cup X_2$ transversely.
\subsubsection*{Step 1 (Flips)}
In this case, we observe that $K_{\mathcal X'}+w{\mathcal D'}$ is nef.
Its restriction to each $\F_1$ is a multiple of the class $\sigma + F$.
Therefore, no flips are required; that is, $(\mathcal X', \mathcal D') = (\mathcal X, \mathcal D)$.

\subsubsection*{Step 2 (Contractions)}
The only $K_{\mathcal X'}+w \mathcal D'$ trivial curves are the directrices $\sigma_i$'s on the $X_i'$.
Therefore, $\overline X$ is the union of two copies of $\PP^2$ along a line, and $\overline D$ is $D_1 \cup D_2$.
Each $D_i$ is a cubic curve, intersecting the line of attachment transversely.

The two steps necessary for the proof of \autoref{prop:stab} are thus complete.

\subsection{The $\F_3-\F_1$ case}\label{f13-case}
Suppose we are in case \eqref{f13} of the unstable list from page~\pageref{f13}, which is a mixture of the two cases \eqref{f11} and \eqref{f33} treated before.
That is, $C = C_1 \cup C_2$ mapping to $P = \PP^1 \cup \PP^1$, where $C_1$ is the disjoint union of $\PP^1$ and a hyperelliptic curve of genus $2$, and $C_2$ is a connected curve of genus 1.
In this case $X = X_1 \cup X_2$ and $D = D_1 \cup D_2$, where $X_1 \cong \F_3$ and $X_2 \cong \F_1$; $D_1 \subset X_1$ is the disjoint union of the directrix $\sigma_1$ and a curve $H_1$ of class $2 \sigma_1 + 6F$, and $D_2 \subset X_2$ is a divisor of class $3 \sigma_2 + 3F$; both $D_1$ and $D_2$ intersect the fiber $X_1 \cap X_2$ transversely.

\subsubsection*{Step 1 (Flips):}
Let $(\mathcal X', \mathcal D')$ be obtained from $(\mathcal X, \mathcal D)$ by flipping the $-3$ curve $\sigma_1 \subset X_1$.
Let $(X', D')$ be the central fiber of $(\mathcal X' \to \mathcal D') \to \Delta$.
The surface $X'$ is the union of two components $X'_1 \cup X'_2$, where $X_i'$ is related to $X_i$ by a diagram
\[ X_i \leftarrow \widetilde X_i \rightarrow X_i',\]
given by \autoref{fig:type2flip} where $X_1$ corresponds to $T$ and $X_2$ corresponds to $S$.

Our next course of action differs substantially depending on the configuration of the directrices $\sigma_1$ and $\sigma_2$.

\subsubsection{Case (a): $\sigma_1$ and $\sigma_2$ do not intersect}\label{f13-1}
Let $p \in X_2$ be the intersection point of $\sigma_1$ with $X_2$.
Since $\sigma_1$ and $\sigma_2$ are disjoint, $\sigma_2$ does not pass through $p$.

\subsubsection*{Step 2a (Contractions):} \label{subsec:2acontractions}
We claim that $K_{\mathcal X'} + w \mathcal D'$ is nef. To see this, it suffices to show that the restriction of $K_{\mathcal X'} + w \mathcal D'$ to each component of $X'$ is nef.
Consider the component $X'_1$ of $X'$.
Note that $X_1'$ is obtained from $X_1 \cong \F_3$ by contracting the $(-3)$ curve $\sigma_1$.
Therefore, $X_1'$ is of Picard rank $1$, with $\Pic_\Q\left(X_1'\right)$ generated by $F$, the image of the fiber of $X_1$.
It is easy to calculate that
\[(K_{\mathcal{X}'}+w\mathcal{D}') \big|_{X'_1} = 6\epsilon F. \]
In particular, $K_{\mathcal{X}'}+w\mathcal{D}'$ is ample on $X'_1$.
As a result, it suffices to show that $(K_{\mathcal{X}'}+w\mathcal{D}')|_{X_2'}$ is nef.

Note that $\sigma_2 \subset X_2'$ is a $(-1)$-curve lying in the smooth locus of $X_2'$.
Let $X_2' \to X_2''$ be the contraction of $\sigma_2$.
It is easy to see that both $K_{\mathcal{X}'}|_{X_2'}$ and $D'$ are $\sigma_2$-trivial.
Hence, they descend to Cartier divisors on $X_2''$, which we denote by $K''$ and $D''$, respectively.
It now suffices to show that $K'' + w D''$ is nef on $X_2''$.

We now describe two extremal curves on $X_2''$.
For the first, recall that $\widetilde X_2 \to X_2$ is the composite of three successive blow-ups, and $\widetilde X_2 \to X_2'$ contracts the exceptional divisors introduced in the first two of these three blow-ups.
Let $E$ be the image in $X_2''$ of the exceptional divisor of the third blow-up.
For the second, note that there is a unique section $\tau$ of $X_2$ through $p$ that is tangent to $D_2$ at $p$.
Let $L$ be the image in $X_2''$ of the proper transform of this section in $\widetilde X_2$.

\begin{lemma}\label{prop:f13-1-cone}
  \mbox{}
  \begin{enumerate}
  \item The curves $E$ and $L$ generate the cone of curves $\overline{\NE}\left(X_2''\right)$. 
  \item  The divisor $K'' + wD''$ is nef. It is ample if $\tau$ is not triply tangent to $D_2$ at $p$.
\end{enumerate}
\end{lemma}
We say that $\tau$ is triply tangent to $D_2$ at $p$ if the unique subscheme of $D_2$ of length 3 supported at $p$ is contained in $\tau$.
In particular, if $\tau$ is a component of $D_2$, then it is triply tangent to $D_2$ at $p$.
\begin{proof}
  It is easy to calculate the intersection table of $E$ and $L$.
  If $\tau$ is not triply tangent to $D_2$ at $p$, we have the table
  \[
    \begin{array}{l | r r}
      & E & L \\
      \hline
      E & -\frac13 & 0  \\
      L & 0 & -\frac13
    \end{array},
  \]
  and otherwise we have
    \[
    \begin{array}{l | r r}
      & E & L \\
      \hline
      E & -\frac13 & 1  \\
      L & 1 & -2
    \end{array}.
  \]
  In either case, $E$ and $L$ represent effective classes of negative self-intersection, and therefore, they are extremal in $\overline{\NE}\left(X_2''\right)$.
  We also see that the classes of $L$ and $E$ are linearly independent.
  Since $\overline{\NE}\left( X_2'' \right)$ is two-dimensional, it follows that $L$ and $E$ span it.

  Let $F$ be the image in $X_2''$ of the class of a fiber of $X_2$.
  Then we have $E \cdot F = 0$ and $L \cdot F = 1$.
  A straightforward computation shows that we have
  \begin{align*}
    K'' &\equiv -2F + 2E \text{, and}\\
    D'' \big|_{X_2''} &\equiv 3F - 3E.\\
  \end{align*}
  Therefore, we get
  \begin{align*}
    (K''+ w D'') \cdot E &= \epsilon, \text{, and} \\
    (K'' + w D'') \cdot L &=
                            \begin{cases}
                              3 \epsilon & \text{ if $\tau$ is not triply tangent to $D_2$ at $p$,}\\
                              0 & \text{otherwise}.
                            \end{cases}
  \end{align*}
  The proof is now complete.  
\end{proof} 
With the proof of the lemma, we finish the proof that $K_{\mathcal X'}+ w \mathcal D'$ is nef, and hence the two steps required for the proof of \autoref{prop:stab} in sub-case (a) of the $\F_3-\F_1$ case.
We recall that the stable limit $(\overline X, \overline D)$ is obtained as the $\left( K_{\mathcal X'}+ w \mathcal D'\right)$-model of $(X',D')$.

\begin{remark}\label{rmk:F31}
  We record the geometry of $(\overline X, \overline D)$ obtained above.
  Recall that $\overline X$ is obtained from $X'$ by contracting the following curves:
  \begin{inparaenum} 
  \item the curve $\sigma_2 \subset X_2'$, and
  \item the curve $L \subset X_2'$ if $\tau$ is triply tangent to $D_2$ at $p$.
  \end{inparaenum}

  If $\tau$ is triply tangent to $D_2$ at $p$, then we see that $\overline X$ is the union of $\overline X_1 = \PP(3,1,1)$ and $\overline X_2 = \PP(3,1,2)$, where the $\mu_3$-singularity of $\overline X_1$ is glued to the $A_2$-singularity of $\overline X_2$ resulting in the (non-isolated) surface singularity $p$ given by $xy = 0 \subset \frac13(2,1,1)$.
  The $A_1$-singularity $q$ of $\overline X_2$ lies away from the double curve.
  The divisor $\overline D$ lies away from $p$.
  If $\tau$ is not a component of $D_2$, then $\overline D$ lies away from $q$.
  If $\tau$ is a component of $D_2$, then $\overline D$ passes through $q$ and has a node or a cusp there, depending on whether the residual curve $\overline{D_2 \setminus \tau}$ intersects $\tau$ transversely at two points or tangentially at one point.

  If $\tau$ is not triply tangent to $D_2$ at $p$, then $\overline X$ is a smoothing of $\PP(3,1,1) \cup \PP(3,1,2)$ at the isolated $A_1$-singularity $q$.
  As in \autoref{rmk:F4_surface}, it is easy to check that the isomorphism type of $\overline X$ does not depend on the divisor $D$, and $\overline{X}$ is not a union of toric surfaces along toric subschemes.
\end{remark}

\subsubsection{Case (b): $\sigma_1$ and $\sigma_2$ intersect}\label{f13-2}
In contrast with case (a), $K_{\mathcal X'} + w \mathcal D'$ is \emph{not} nef in this case, and a further flip is necessary.
\subsubsection*{Step 1 (a further flip) in case (b):}
To perform the flip, we must understand the configuration of the curves $\sigma_1$, $\sigma_2$, and $D_2$.
Let $p$ be the point of intersection of $\sigma_1$ and $\sigma_2$.
Since $\sigma_1 \subset D_1$, we must have $p \in D_2$.
However, we also have $D_2 \cdot \sigma_2 = 0$, and therefore, we conclude that $\sigma_2$ must be a component of $D_2$.
Let $D_2 = \sigma_2 \cup H$, where $H$ is the residual curve.
Then $H \subset X_2$ is a curve of class $2 \sigma_2 + 3f$.
Since $D_2$ is reduced, $H$ does not contain $\sigma_2$ as a component, and since $H \cdot \sigma_2 = 1$, it must intersect $\sigma_2$ transversely at a unique point $q$.
Also, since $D_2$ intersects the fiber through $p$ transversely, we have $q \neq p$.
Let $\sigma_2'$ be the proper transform of $\sigma_2$ in $X_2'$.
Then $\sigma_2'$ is a smooth rational curve of self-intersection $(-4)$ in the smooth locus of $X_2'$.
Let $\mathcal X' \dashrightarrow \mathcal X''$ be the type I flip along $\sigma_2'$.
Let $\mathcal D''$ be the proper transform of $\mathcal D'$ in $\mathcal X''$.

\subsubsection*{Step 2 (contractions) in case(b):}
We claim that $K_{\mathcal X''} + w \mathcal D''$ is nef.
The proof of the nefness of $K_{\mathcal X''} + w \mathcal D''$ closely resembles the proof of nefness in case (a).
As before, nefness on $X_1''$ is easy, using that $X_1''$ is of Picard rank 1.
For $X_2''$, we have the diagram
\[ \F_1 = X_2 \xleftarrow{a} \widetilde{X_2} \xrightarrow{b} X_2' \xleftarrow{a'} \widetilde{X_2'} \xrightarrow{b'} X_2'', \]
where the first transformation $X_2 \dashrightarrow X_2'$ is the result of a type II flip and the second transformation $X_2' \dashrightarrow X_2''$ is the result of a type I flip.
That is, the map $a$ consists of 3 successive blow-ups, $b$ consists of $2$ successive blow downs, $a'$ consists of two successive blow-ups, and $b'$ consists of 2 successive blow-downs.
We may perform all the blow-ups first, followed by all the blow-downs, obtaining a sequence
\[ X_2 \xleftarrow{\alpha} \Xi \xrightarrow{\beta} X_2''.\]
The exceptional locus of $\alpha$ consists of a chain of rational curves, whose dual graph is shown below.
\[
  \begin{tikzpicture}
    \draw
    node [curve, label=below:$\widetilde{\sigma_2}$,label=above:$-5$] (S) {}
    node [curve, label=below:$G_1$,label=above:$-2$, left of=S] (G1) {}
    node [curve, label=below:$G_2$,label=above:$-1$, left of=G1] (G2) {}
    node [curve, label=below:$E_1$,label=above:$-1$, right of=S] (E1) {}
    node [curve, label=below:$E_2$,label=above:$-2$, right of=E1] (E2) {}
    node [curve, label=below:$E_3$,label=above:$-2$, right of=E2] (E3) {}
    (G2) edge (G1) (G1) edge (S) (S) edge (E1) (E1) edge (E2) (E2) edge (E3);
  \end{tikzpicture}
\]
Here, $\widetilde{\sigma_2}$ is the proper transform of $\sigma_2$.
By contracting $E_2$, $E_3$, $G_1$ and $G_2$, we obtain $X_2'$; by contracting $G_1$, $\widetilde{\sigma_2}$, $E_2$ and $E_3$, we obtain $X_2''$.

Let $X_2'' \to X_2'''$ be the contraction of $\beta(E_1)$.
Equivalently, let $X_2'''$ be the surface obtained from $\Xi$ by contracting the chain $G_1,\widetilde{\sigma_2}, E_1, E_2, E_3$.
By contracting $E_1$ first, then $E_2$, then $E_3$, then the chain $G_1, \widetilde{\sigma_2}$, which are now both $(-2)$ curves, we see that $X_2'''$ has an $A_2$ singularity. 
It is easy to check that the divisors $K_{X''}''$ and $D''$ are both trivial on $\beta(E_1) \subset X_2''$, and hence both divisors are pull-backs of Cartier divisors from $X_2'''$, say $K'''$ and $D'''$.
It suffices to show that $K''' + wD'''$ is nef on $X_2'''$.

Denote by $G$ the image in $X_2'''$ of $G_2$.
Recall that $q \in X_2'$ is the intersection point of $\sigma_2'$ and $H$.
Let $F$ be the fiber of $X_2' \to P_2$ through $q$, and let $\widetilde F$ be the proper transform of $F$ in $X_2'''$.
We have the following analogue of \autoref{prop:f13-1-cone}
\begin{lemma}\label{prop:f13-2-cone}
  \mbox{}
  \begin{enumerate}
  \item The curves $G$ and $\widetilde F$ generate the cone of curves $\overline{\NE}\left(X_2'''\right)$. 
  \item  The divisor $( K''' + wD''')$ is nef.
    It is ample if $F$ is not tangent to $H$ at $q$.
\end{enumerate}  
\end{lemma}
We say that $F$ is tangent to $H$ at $q$ if the unique subscheme of length 2 of $H$ supported at $q$ is contained in $F$.
In particular, if $H$ contains $F$ as a component, then $F$ is tangent to $H$ at $q$.
\begin{proof}
  The proof is analogous to the proof of \autoref{prop:f13-1-cone}.
\end{proof}
With the proof of \autoref{prop:f13-2-cone}, the proof of nefness of $K_{\mathcal X''}+w\mathcal D''$ is complete, and so are the two steps necessary for the proof of \autoref{prop:stab} in case (b).
We recall that the stable limit $(\overline X, \overline D)$ is obtained as the $\left( K_{\mathcal X''}+ w \mathcal D''\right)$-model of $(X'',D'')$.

\begin{remark}\label{rem:f3-f1(b)}
  We observe that the pairs $(\overline X, \overline D)$ obtained in case (b) are the same as the pairs $(\overline X, \overline D)$ obtained in case (a).

  More specifically, consider a pair $(X', D')$ as in case (b).
  Let $(\PP^2, C)$ be the plane cubic obtained from $(X_2', D_2')$ by contracting $\sigma_2$, and let $L \subset \PP^2$ be the image of the double curve in $X_2'$.
  Let $X''_2$ be the blow up of $\PP^2$ at a general point of $L$, and let $D''_2$ be the proper transform of $C$ in $\widetilde X_2$.
  Construct $(X'', D'')$ by gluing $(X'_1, D'_1)$ and $(X_2'',D_2'')$ in the obvious way.
  Then $(X'', D'')$ is pair as in case (a) that leads to the same stable limit $(\overline X, \overline D)$ as in the pair $(X', D')$.
\end{remark}

\begin{remark}\label{note:ell_bridge}
  Observe that if $C_2$ is smooth, then $\sigma_1$ and $\sigma_2$ must be disjoint as treated in \autoref{f13-1}.
  In the resulting $(\overline X, \overline D)$, the divisor $\overline{D}$ meets the double curve $\overline{X}_1 \cap \overline{X}_2$ at 2 distinct points $q$, $r$.
  The divisor $q+r$ is the hyperelliptic divisor of $H_1$.
	
  To reconstruct $(X,D)$ from $(\overline{X},\overline{D})$ in this case, we must choose a point $t \in \overline{X}_1 \cap \overline{X}_2$ away from $\overline{D}$.
  The blow up of $t$ on $\overline{X}_2$ yields $X_2'$, and hence $X' = X_1' \cup X_2'$.
  We can then undo the transformations in the type 2 flip (\autoref{sec:typeIIflip}) to obtain $(X, D)$.

  If we do the same procedure starting with $t$ on $\overline{D}$, then then the corresponding $(X,D)$ is a surface with intersecting directrices as in \autoref{f13-2}.
\end{remark}

\subsection{Summary of stable replacements}\label{subsec:mainthm}
Thanks to the proof of \autoref{prop:stab} in \autoref{sec:stable-replacement}, we obtain an explicit list of stable log quadrics $(S, D)$, namely the points of $\mathfrak X$.

We first look at the surfaces $S$.
\autoref{tab:list_of_all_pairs} lists the possible surfaces $S$ along with its non-normal-crossing singularities.
If $S$ is reducible, then \autoref{tab:list_of_all_pairs} also describes the double curve on each component.
In the table, the divisor $H$ on a weighted projective space refers to the zero locus of a section of the primitive ample line bundle, and the divisor $F$ on (coarse space of) a projective bundle denotes the (coarse space of) a fiber.
The last column directs the reader to the relevant section in \autoref{sec:stable-replacement} where the stable reduction is obtained.

\begin{table}
  \centering
  \rowcolors{2}{gray!15}{white}
  \begin{tabular}{p{.25\textwidth}|p{.27\textwidth}|p{.20\textwidth}|p{.18\textwidth}}
    \hline
    \rowcolor{gray!25}
    $S$ & Singularities of $S$ & Double curve & Reference\\
    \hline
    $\PP^1 \times \PP^1$ & -- & -- & --\\
    $\PP(1,1,2)$ &$p: A_1$ & -- & \autoref{subsec:MaroniSpecial}\\
    $\Q$-Gorenstein smoothing of the $A_1$ singularity of $\PP(9,1,2)$
        &$p: \frac{1}{9}(1,2)$&--&\autoref{subsec:hypellip} Case~\eqref{subsubsec:F4unram}\\
    $\PP(9,1,2)$ & $p: \frac{1}{9}(1,2)$, $q: A_1$ &-- & \autoref{subsec:hypellip} Case \eqref{subsubsec:F4ram}, \eqref{subsubsec:f4comp}  \\
    Coarse space of \newline $\PP(\O(4/3,5/3) \oplus \O(5/3,4/3))$ & $p: (xy=0) \subset \frac{1}{3}(1,2,1)$, $q: (xy=0) \subset \frac{1}{3}(2,1,1)$ & $F$, $F$ & \autoref{sec:replace-f33}\\
        $\PP^2 \cup \PP^2$ & $(xy = 0) \subset \A^3 $ & $H$, $H$ & \autoref{sec:f1-f1}\\
    $\Q$-Gorenstein smoothing of the $A_1$ singularity of $\PP(3,1,2) \cup \PP(3,1,1)$ & $p: (xy=0) \subset \frac{1}{3}(2,1,1)$ & deformation of $2H$, deformation of $H$ & \autoref{f13-case} (non triply tangent case)\\
    $\PP(3,1,2) \cup \PP(3,1,1)$ & $p: (xy=0) \subset \frac{1}{3}(2,1,1)$, $q: A_1$ on $\PP(3,1,2)$ & $2H$, $H$ & \autoref{f13-case} (triply tangent case) \\
  \end{tabular}
  \caption{Surfaces $S$ that appear in stable log quadrics $(S, D)$}
  \label{tab:list_of_all_pairs}
\end{table}

\begin{remark}
  The surface $S$ described as the coarse space of $\PP(\O(4/3, 5/3) \oplus \O(5/3,4/3))$ has two alternate descriptions (see \autoref{rmk:F3F3}).
  First, it is obtained by gluing $\Bl_u\PP(3,1,1)$ and $\Bl_v \PP(3,1,1)$ along a $\PP^1$, where $u$ and $v$ are curvilinear subschemes of length 3.
  Second, it is a degenerate (non-normal) toric surface represented by the subdivided rectangle in \autoref{fig:toricXbar}.
\end{remark}

We now look at the divisors $D$.
By \autoref{rem:sing}, the curve $D$ is reduced and only admits $A_m$ singularities for $m \leq 4$.
We also observe that $D$ is a Cartier divisor.
In particular, the log quadrics $(S, D)$ satisfy the index condition.
To see that $D$ is Cartier, it suffices to examine it locally at the singular points of $S$.
We observe that whenever $D$ passes through an isolated singularity of $S$, it is an $A_1$ singularity; $D$ is either nodal or cuspidal at the singularity, and is cut out by one equation.
Whenever $D$ passes through a non-isolated singularity of $S$, it does so at the transverse union of two smooth surfaces; the local picture of $(S, D)$ is
\[ \left(\spec \k[x,y,t]/(xy), t \right).\]
Thus, $D$ is Cartier.
Furthermore, we can check directly that $(S, D)$ satisfies the definition of a stable log surface for all positive $\epsilon < 1/30$.

We collect the observations above in the following theorem.
\begin{theorem}\label{thm:geomofmoduli}
  Let $(S, D)$ be a stable log quadric.
  \begin{enumerate}
  \item \label{geomofmoduli:S} The isomorphism class of $S$ is one of the 8 listed in \autoref{tab:list_of_all_pairs}. 
    \item\label{geomofmoduli:D} The divisor $D$ is Cartier.
      In particular, $(S, D)$ satisfies the index condition (\autoref{def:index}).
    \item\label{geomofmoduli:fin_epsilon}
      $(S, D)$ satisfies \autoref{def:sls} for all positive $\epsilon < 1/30$.
\end{enumerate}
\end{theorem}

As a corollary, we obtain the following.
\begin{corollary}\label{prop:fin_type_proper}
  The stack $\mathfrak X$ is of finite type and proper over $\k$.
\end{corollary}
\begin{proof}
  From \autoref{thm:geomofmoduli} \eqref{geomofmoduli:fin_epsilon}, we get that $\mathfrak X$ is a locally closed substack of the finite type stack $\mathfrak F_{\epsilon, 8}$ for a positive $\epsilon < 1/30$ (see \autoref{prop:fintype}).
  The valuative criterion for properness follows from the valuative criterion for properness for $\overline {\mathcal H}^3_4(1/6 + \epsilon)$ \cite[Corollary~6.6]{deo:13} and stabilization (\autoref{prop:stab}).
\end{proof}

We take a closer look at the pairs $(S, D)$ where $D$ is smooth.
We see that these arise from a triple cover $f \from C \to \PP^1$ where $C$ is smooth non-hyperelliptic curve of genus 4, or from $g \from C \cup_p \PP^1 \to \PP^1$ where $C$ is a smooth hyperelliptic curve of genus 4.
\begin{corollary}\label{cor:sm_curve_loci}
  For all $(S, D)$ such that $D$ is smooth, we have the following classification.
  \begin{center}
    \rowcolors{2}{gray!15}{white}
    \begin{tabular}{p{.30\textwidth}|p{.20\textwidth}|p{.40\textwidth}}
      \hline
      \rowcolor{gray!25}
      $S$ & $D$ & Embedding $D \hookrightarrow S$\\
      \hline
      $\PP^1 \times \PP^1$ &Non-hyperelliptic, Maroni general & Induced by the canonical embedding\\
      $\PP(1,1,2)$ &Non-hyperelliptic, Maroni special& Induced by the canonical embedding\\
      $\Q$-Gorenstein smoothing of the $A_1$ singularity of $\PP(9,1,2)$
          &Hyperelliptic & Determined by a hyperelliptic divisor $p+q$ with $p \neq q$\\
      $\PP(9,1,2)$ &Hyperelliptic & Determined by a hyperelliptic divisor $2p$.
    \end{tabular}
  \end{center}
\end{corollary}

%%% Local Variables:
%%% mode: latex
%%% TeX-master: "main"
%%% End:

\section{Deformation theory}\label{sec:deformation}
% Use Q-Gorenstein deformation theory to show smoothness.

In this section, we study the $\Q$-Gorenstein deformations of pairs parametrized by $\mathfrak X$.
Since many of the results carry over from \cite[\S~3]{hac:04}, our treatment will be brief.

\subsection{The $\Q$-Gorenstein cotangent complex}\label{subsec:prelimdef}
Let $A$ be an affine scheme, and $S \to A$ a $\Q$-Gorenstein family of surfaces.
Denote by $p \from \mathcal S \to S$ the canonical covering stack of $S$.
By the definition of a $\Q$-Gorenstein family, $\mathcal S \to A$ is flat.
Let $L_{\mathcal S/A}$ be the cotangent complex of $\mathcal S \to A$ \cite{ill:72}.

\begin{definition}[($\Q$-Gorenstein deformation functors)]
  Let $M$ be a quasi-coherent  $\O(A)$-module.
  Define the $\O(A)$-module $T^i_{\Q\rm{Gor}}(S/A, M)$ and the $\O_{\mathcal S}$-module $\mathcal T^i_{\Q{\rm Gor}}(S/A, M)$ by
\begin{align*}
  T^i_{\Q\rm{Gor}}(S/A, M) &= \Ext^i(L_{\mathcal{S}/A},\mathcal{O}_{\mathcal{S}}\otimes_{A}M), \\
  {\mathcal T}^i_{\Q\rm{Gor}}(S/A, M) &= p_* \SExt^i(L_{\mathcal{S}/A},\mathcal{O}_{\mathcal{S}}\otimes_{A}M).
\end{align*}
\end{definition}
Recall that we also have the usual deformation functors $T^i(S/A, M)$ and $\mathcal T^i(S/A, M)$ defined using the cotangent complex of $S \to A$.
The usual functors, in general, differ from the $\Q$-Gorenstein ones (except for $i = 0$, see \autoref{thm:QGor_def_obs}).

The $\Q$-Gorenstein deformation functors play the expected role in classifying $\Q$-Gorenstein deformations and obstructions.
To make this precise, let $A \to A'$ be an infinitesimal extension of $A$.
A \emph{$\Q$-Gorenstein deformation} of $S \to A$ over $A'$ is a flat morphism $S' \to A'$ along with an isomorphism $S' \times_{A'} A \cong S$.

Let $A \to A'$ be a square zero extension of $A$ by a quasi-coherent $\O(A)$-module $M$.
Recall that this means we have a surjection $\O(A') \to \O(A)$ with kernel $M$ and $M^2 = 0$.
\begin{theorem}\label{thm:QGor_def_obs}
  Let $S \to A$ be a $\Q$-Gorenstein family of surfaces and let $A \to A'$ be a square zero extension by an $A$-module $M$.
  \begin{enumerate}
  \item There is a canonical element $o(S/A, A') \in T^2_{\Q{\rm Gor}}(S/A, M)$ which vanishes if and only if there exists of $\Q$-Gorenstein deformation of $S/A$ over $A'$.
  \item If $o(S/A, A') = 0$, then the set of isomorphism classes of $\Q$-Gorenstein deformations of $S/A$ over $A'$ is an affine space under $T^1_{\Q{\rm Gor}}(S/A, M)$.
  \item If $S'/A'$ is a $\Q$-Gorenstein deformation of $S/A$, then the group of automorphisms of $S'$ over $A'$ that restrict to the identity on $S$ is isomorphic to $T^0_{\Q {\rm Gor}}(S/A, M)$.
    Furthermore, we have an isomorphism
    \begin{equation}\label{eq:t0}
      \mathcal T^0_{\Q {\rm Gor}}(S/A, M) \cong \mathcal T^0(S/A, M).
    \end{equation}
  \end{enumerate}
\end{theorem}
\begin{proof}
  The isomorphism \eqref{eq:t0} is from \cite[Lemma~3.8]{hac:04}.
  The rest of the assertions are from \cite[Theorem~3.9]{hac:04}.
  The main point in the proof is an equivalence between $\Q$-Gorenstein deformations of $S$ and deformations of $\mathcal S$.
  Having established this equivalence, the theorem follows from the properties of the cotangent complex \cite[Theorem~1.7]{ill:72}.
\end{proof}

% Note that the obstruction space $T^2_{\Q G,S}$ for a slc surface $S$ can be computed by applying local-to-global spectral sequence for $R\Gamma \circ Rp_* \circ R\SHom$ (as $p_*$ is exact by definition of coarse moduli space) to get
% \begin{equation}\label{eqn:ss_T^i}
% 	E_2^{p,q}=H^p(\mathcal{T}^q_{\Q G,S}) \implies T^{p+q}_{\Q G,S}
% \end{equation}

\subsection{
  Deformations of pairs
}\label{subsec:canocov}

Having discussed deformations of surfaces, we turn to deformations of pairs.
The upshot of this discussion is \autoref{thm:divisor_no_worries}, which says that the deformations of pairs are no more challenging than the deformations of the ambient surfaces.

Let $(S,D)$ be a stable log quadric, that is, a $\k$-point of $\mathfrak X$.
The $\Q$-Gorenstein cotangent complex of a surface $S$ is determined by the canonical covering stack $p \from \mathcal{S} \rightarrow S$.
We collect the properties of $\mathcal{S}$ that we require for further analysis.
Set $D_{\mathcal S} = D \times_S \mathcal S$.
\begin{lemma}\label{lem:Cartierness}\label{prop:lciness}
  The stack $\mathcal S$ has lci singularities.
\end{lemma}
\begin{proof}
  Recall that $\mathcal S \to S$ is an isomorphism over the Gorenstein locus of in $S$.
  From \autoref{thm:geomofmoduli}, we see that the only non-Gorenstein singularities on $S$ are $\frac{1}{3}(1,1)$, $\frac{1}{9}(1,2)$, and $(xy=0) \subset \frac{1}{3}(2,1,1)$, and furthermore, all other singularities of $S$ are lci.
  The canonical covering stacks of the three non-Gorenstein singularities are
  \begin{align*}
    \left[\A^2 / \mu_3\right] &\to \frac{1}{3}(1,1),\\
    \left[\spec \k[x,y,z]/(xy-z^3) / \mu_3\right] &\to \frac{1}{9}(1,1), \text{ and }\\
    \left[ \spec \k[x,y,z]/(xy) / \mu_3 \right] &\to (xy=0) \subset \frac{1}{3}(2,1,1).
  \end{align*}
  All three stacks on the left have lci (in fact, hypersurface) singularities.
  The first assertion follows.  
\end{proof}

\begin{lemma}\label{lem:h10}
  Let $(S, D)$ be a stable log quadric.
  Then $H^1(\O_S(D)) = 0$.
\end{lemma}
\begin{proof}
  The assertion is analogous to \cite[Lemma~3.14]{hac:04}.
  The same proof goes through as long as we check that $-(K_S-D)$ is ample and the normalization $S^\nu$ is log terminal.
  From \autoref{thm:geomofmoduli}, we know that $D \cong -3/2 K_S$ and $K_S + (2/3 + \epsilon)D$ is ample.
  It follows that both $-K_S$ and $D$ are ample, and hence so is $-(K_S-D)$.
  From looking at the singularities of $S$ in \autoref{thm:geomofmoduli}, we see that $S^\nu$ is log terminal.
\end{proof}

\begin{proposition}\label{thm:divisor_no_worries}
  Let $A$ be an affine scheme and $(S, D)$ an object of $\mathfrak X(A)$.
  Let $A \to A'$ be an infinitesimal extension and $S' \to A'$ a $\Q$-Gorenstein deformation of $S/A$.
  Then there exists a $\Q$-Gorenstein deformation $(S', D')$ over $A'$ of $(S, D)$.
  That is, there exists an object of $\mathfrak X (A')$ restricting to $(S, D)$ over $A$.
\end{proposition}
\begin{proof}
  The assertion is analogous to \cite[Theorem~3.12]{hac:04}.
  The proof depends on two lemmas: \cite[Lemma~3.13]{hac:04}  and \cite[Lemma~3.14]{hac:04}.
  The analogue of the first is \autoref{thm:geomofmoduli}\eqref{geomofmoduli:D} and of the second is \autoref{lem:h10}.
\end{proof}

We now have all the tools to show that the $\Q$-Gorenstein deformations of stable log quadrics are unobstructed.
\begin{theorem}\label{subsec:smoothness}
  $\mathfrak X$ over $\k$ is a smooth stack.
\end{theorem}
\begin{proof}
  We use the infinitesimal lifting criterion for smoothness.
  Let $(S, D)$ be a stable log quadric.
  Let $A$ be the spectrum of an Artin local $\k$-algebra, $(\mathscr S, \mathscr D)$ be a $\Q$-Gorenstein deformation of $(S, D)$ over $A$, and $A \to A'$ an infinitesimal extension.
  We must show that $(\mathscr S, \mathscr D)$ extends to a deformation $(\mathscr S', \mathscr D')$ over $A'$.

  By induction on the length, it suffices to prove the statement when the kernel of $\O(A') \to \O(A)$ is $\k$.
  By \autoref{thm:divisor_no_worries}, it suffices to show the existence of $\mathscr S'$.
  By \autoref{thm:QGor_def_obs}, it suffices to show that $T^2_{\Q\rm{Gor}}(\mathscr S/A, k) = 0$.
  Note that we have $T^2_{\Q\rm{Gor}}(\mathscr S/A, k) = T^2_{\Q\rm{Gor}}(S/k, k)$.
  Henceforth, we abbreviate $T^i_{\Q\rm{Gor}}(S/k, k)$ by $T^i_{\Q\rm{Gor}}(S)$ and use similar abbreviations for $T^i$ and $\mathcal T^i$.

  To show that $T^2_{\Q\rm{Gor}} = 0$, it suffices to show by the Leray spectral sequence that $H^0\left(\mathcal T^2_{\Q\rm{Gor}}\right)$, $H^1\left(\mathcal T^1_{\Q\rm{Gor}}\right)$, and $H^2\left(\mathcal T^0_{\Q\rm{Gor}}\right)$ are all $0$.
  We do this one by one.

  Let $p \from \mathcal S \to S$ be the canonical covering stack.
  By \autoref{prop:lciness}, we know that $\mathcal S$ is lci.
  Therefore, $\mathcal T^2(\mathcal S) = 0$, and hence $\mathcal T^2_{\Q\rm{Gor}}(S) = 0$.

  The sheaf $\mathcal T^1_{\Q\rm{Gor}}(S)$ is supported on the singular locus of $S$.
  If the singular locus has dimension less than one, then $H^1(\mathcal T^1_{\Q\rm{Gor}}(S)) = 0$.
  From \autoref{thm:geomofmoduli}, we see that the only cases where the singular locus of $S$ has dimension $\geq 1$ have $S = S_1 \cup_B S_2$, where $S_1$ and $S_2$ are irreducible and meet along a curve $B \cong \PP^1$.
  More precisely, the local structure of $S$ is either $(xy = 0) \subset \A^3$  or its quotient by a $\mu_r$ where $\zeta \in \mu_r$ acts by
  \[ \zeta \cdot (x,y, z) \mapsto (\zeta x, \zeta^{-1} y, \zeta^a z ), \]
  with $\gcd(a, r) = 1$.
  Denote by $B_i$ the restriction of $B$ to $S_i$ for $i = 1, 2$.
  By \cite[Proposition~3.6]{has:99}, we get that in this case
  \[ \mathcal T^1_{\Q{\rm Gor}}  = \O_{S_1}(B_1)\big|_B \otimes \O_{S_2}(B_2)\big|_B.\]
  Thus, $\mathcal T^1_{\Q{\rm Gor}}$ is a line bundle on $B \cong \PP^1$ of degree $B_1^2 + B_2^2$.
  In the surfaces listed in \autoref{thm:geomofmoduli}, we see that $B_1^2 + B_2^2$ is either $0$, $1$, or $2$.
  We conclude that $H^1(\mathcal T^1_{\Q{\rm Gor}}) = 0$.

  By \autoref{thm:QGor_def_obs} equation~\eqref{eq:t0}, the sheaf $\mathcal T^0_{\Q \rm{Gor}}(S)$ is isomorphic to $\mathcal T^0(S)$.
  First, assume that $S$ is reducible with the notation as above.
  Then \cite[Lem 9.4]{hac:04} applies to $S$ as its proof is valid whenever $S$ is slc, $S_i$ only has quotient singularities, $S$ is not normal crossing along at most two points of $B$, the divisor $K_{S_i}+B_i$ anti-ample, and $h^1(\mathcal{O}_{\tilde S_i})=0$ where $\tilde S_i \rightarrow S_i$ are the minimal resolutions.
  Of these, the first three conditions follow from \autoref{thm:geomofmoduli}.
  The anti-ampleness of $K_{S_i}+B_i$ can be seen by noting that they are restrictions of the anti-ample $\Q$-divisor $K_S$ to $S_i$.
  Finally, since each $S_i$ is rational by \autoref{thm:geomofmoduli}, we have $h^1(\mathcal{O}_{\tilde S_i})=0$.
  This proves that $H^2(\mathcal{T}^0_S)=0$ whenever $S$ is reducible.

  When $S$ is irreducible, consider the minimal resolution $c \from \tilde S \rightarrow S$.
  Since $S$ only has quotient singularities by \autoref{thm:geomofmoduli}, the surface $\tilde S$ is rational as well by \cite[Prop 5.15]{kol.mor:98}.
  Therefore, we have $c_*\mathcal{O}_{\tilde S}=\mathcal{O}_S$ and $R^ic_*\mathcal{O}_{\tilde S}=0$ for any $i>0$.
  Furthermore, since $\tilde S$ is rational, we have $q(S):=h^1(\mathcal{O}_S)=h^1(c_*\mathcal{O}_{\tilde S})=0$ and $p_g(S):=h^2(\mathcal{O}_S)=h^2(c_*\mathcal{O}_{\tilde S})=0$.
  Since $-K_S$ is ample and effective as well, \cite[Prop III.5.3]{man:95} implies that $H^2(\mathcal{T}^0_S)=0$. The proof of \autoref{subsec:smoothness} is thus complete.
\end{proof}
%%% Local Variables:
%%% mode: latex
%%% TeX-master: "main"
%%% End:

\section{Geometry}\label{sec:geometry}
% More geometric statements about the moduli space: Picard group,
% canonical class, blow up of hyperelliptic locus, and other
% speculations.
In this section, we take a closer look at the geometry of $\mathfrak X$, and compare it with related moduli spaces.

\subsection{Comparison of $\mathfrak{X}$ with the spaces of weighted admissible covers $\overline{\mathcal{H}}_4^3(1/6+\epsilon)$}
Recall that $\overline{\mathcal H}_4^3(1/6 + \epsilon)$ is the moduli space of weighted admissible covers where up to 5 branch points are allowed to coincide.
Let $U \subset \overline{\mathcal H}_4^3(1/6 + \epsilon)$ be the open substack parametrizing $\phi \from C \to P$ where $P \cong \PP^1$, the curve $C$ is smooth, and the Tschirnhausen bundle $E_\phi$ of $\phi$ is $\O(3) \oplus \O(3)$.
We have a morphism $\Phi \from U \to \mathfrak X$ given by the transformation
\[ (\phi \from C \to P) \mapsto (\PP E_\phi, C).\]
\begin{theorem}\label{thm:phi_extends}
  The map $\Phi$ extends to a morphism of stacks $\Phi \from \overline{\mathcal{H}}_4^3(1/6+\epsilon) \to \mathfrak X$.
\end{theorem}
Since $\overline{\mathcal{H}}_4^3(1/6+\epsilon)$ is proper and $\mathfrak X$ is separated, the map $\Phi$ is also proper.

For the proof, we need extension lemmas for morphisms of stacks, extending some well-known results for schemes.
Let $\mathcal X$ and $\mathcal Y$ be separated Deligne--Mumford stacks of finite type over a field $\k$.
Let $\mathcal U \subset \mathcal X$ be a dense open substack.
\begin{lemma}\label{lem:unique_extension}
  Assume that $\mathcal X$ is normal.
  If $\Phi_1, \Phi_2 \from \mathcal X \to \mathcal Y$ are two morphisms whose restrictions to $\mathcal U$ are equal (2-isomorphic), then  $\Phi_1$  and $\Phi_2$ are equal (2-isomorphic).
\end{lemma}
\begin{proof}
  We have the following diagram where the square is a pull-back
  \[
  \begin{tikzpicture}
    \node (Y) {$\mathcal Y$};
    \node[below of=Y] (YY) {$\mathcal Y \times \mathcal Y$};
    \draw (YY) ++ (-2cm,0) node (X) {$\mathcal X$};
    \node (U) [left of=X] {$\mathcal U$};
    \node[above of=X] (Z) {$\mathcal Z$};
    \draw[->] (Z) edge (Y) (Z) edge (X) (U) edge (X) (X) edge node[above] {\tiny $(\Phi_1, \Phi_2)$} (YY) (Y) edge node[right] {\tiny $\Delta$} (YY);
  \end{tikzpicture}
\]
Since $\mathcal Y$ is a separated Deligne--Mumford stack, the diagonal map $\Delta$ is representable, proper, and unramified.
Therefore, so is the pullback $\mathcal Z \to \mathcal X$.
Since $\Phi_1$ and $\Phi_2$ agree on $\mathcal U$, the inclusion $\mathcal U \to \mathcal X$ lifts to $\mathcal U \to \mathcal Z$.
Since $\mathcal Z \to \mathcal X$ is unramified and $\mathcal U \to \mathcal X$ is an open immersion, so is the lift $\mathcal U \to \mathcal Z$.
Let $\overline {\mathcal U} \subset \mathcal Z$ be the closure of $\mathcal U$ and $\overline{\mathcal U}^\nu \to \overline{\mathcal U}$ its normalization.
Since $\mathcal X$ is normal, Zariski's main theorem implies that $\overline {\mathcal U}^\nu \to \mathcal X$ is an isomorphism.
Hence the map $\mathcal Z \to \mathcal X$ admits a section $\mathcal X \to \mathcal Z$.
In other words, the map $(\Phi_1, \Phi_2) \from \mathcal X \to \mathcal Y \times \mathcal Y$ factors through the diagonal $\mathcal Y \to \mathcal Y \times \mathcal Y$.
\end{proof}
\begin{example}
  In \autoref{lem:unique_extension}, we can drop the normality assumption on $\mathcal X$ if $\mathcal Y$ is an algebraic space, but not otherwise.
  An example of distinct maps that agree on a dense open substack can be constructed using twisted curves (see \cite[Proposition~7.1.1]{abr.cor.vis:03}).
  Let $\mathcal X$ be the stack
  \[ \mathcal X = \left[ \spec \C[x,y]/xy / \mu_n \right],\]
  where $\zeta \in \mu_n$ acts by $(x, y) \mapsto (\zeta x, \zeta^{-1} y)$.
  Every $\zeta \in \mu_n$ defines an automorphism of $t_\zeta \from \mathcal X \to \mathcal X$ given by $(x,y) \mapsto (x, \zeta y)$.
  The map $t_\zeta$ is the identity map on the complement of the node of $\mathcal X$, but not the identity map on $\mathcal X$ if $\zeta \neq 1$.
\end{example}

The map $\Phi \from \mathcal U \to \mathcal Y$ induces a map $|\Phi| \from |\mathcal U| \to |\mathcal Y|$ on the set of points.
Let $\phi \from |\mathcal X| \to |\mathcal Y|$ be an extension of $|\Phi|$.
We say that $\phi$ is \emph{continuous in one-parameter families} if for every DVR $\Delta$ and a map $i \from \Delta \to \mathcal X$ that sends the generic point $\eta$ of $\Delta$ to $\mathcal U$, the map $\Phi \circ i \from \eta \to \mathcal Y$ extends to a map $\Delta \to \mathcal Y$ and agrees with the map $\phi$ on the special point.

\begin{lemma}\label{lem:exists_extension}
  Suppose $\mathcal X$ is smooth, and $\Phi \from \mathcal U \to \mathcal Y$ is a morphism.
  Let $\phi \from |\mathcal X| \to |\mathcal Y|$ be an extension of $|\Phi| \from |\mathcal U| \to |\mathcal Y|$.
  If $\phi$ is continuous in one-parameter families, then it is induced by a morphism $\Phi \from \mathcal X \to \mathcal Y$ that extends $\Phi \from \mathcal U \to \mathcal Y$.
  \end{lemma}
By \autoref{lem:unique_extension}, the extension is unique.

\begin{proof}
  Consider the map $(\id, \Phi) \from \mathcal U \to \mathcal X \times \mathcal Y$.
  Let $\mathcal Z \subset \mathcal X \times \mathcal Y$ be the scheme theoretic image of $\mathcal U$ (see \cite[Tag 0CMH]{sta:20}), and let let $\mathcal Z^\nu \to \mathcal Z$ be the normalization.
  By construction, the map $\mathcal Z \to \mathcal X$ is an isomorphism over $\mathcal U$ \cite[Tag 0CPW]{sta:20}.
  Since $\mathcal U$ is smooth (and hence normal), the map $\mathcal Z^\nu \to \mathcal X$ is also an isomorphism over $\mathcal U$.
  Our aim is to show that $\mathcal Z^\nu \to \mathcal X$ is in fact an isomorphism.
  
  Let $Z^\nu \to X$ be the  morphism on coarse spaces induced by $\mathcal Z^\nu \to \mathcal X$.
  Let $(x, y) \in |\mathcal X| \times |\mathcal Y|$ be a point from $Z^\nu$.
  Since the image of $\mathcal U$ is dense in $\mathcal Z^\nu$, there exists a DVR $\Delta$ with a map $\Delta \to \mathcal Z^\nu$ whose generic point maps into the image of $\mathcal U$ and whose special point maps to $(x,y)$.
  The continuity of $\phi$ in one-parameter families implies that $y = \phi(x)$.
  As a result, $Z^\nu \to X$ is a bijection on points.
  As $Z^\nu$ and $X$ are normal spaces, $Z^\nu \to X$ must be an isomorphism.

  By hypothesis, for every DVR $\Delta$, a map $\Delta \to \mathcal X$ that sends the generic point to $\mathcal U$ lifts to a map $\Delta  \to \mathcal Y$, and hence to a map $\Delta \to \mathcal Z^\nu$.
  This implies that $\mathcal Z^\nu \to \mathcal X$ is unramified in codimension 1.
  It follows by the same arguments as in \cite[Corollary~6]{ger.sat:17} that $\mathcal Z^\nu \to \mathcal X$ is an isomorphism.
  Since \cite[Corollary~6]{ger.sat:17} is stated with slightly stronger hypotheses, we recall the proof.
  Let $V$ be a scheme and $V \to \mathcal X$ an \'etale morphism.
  Set $\mathcal W = \mathcal Z^\nu \times_{\mathcal X} V$ and $U = \mathcal U \times_{\mathcal X} V$.
  Let $\mathcal W \to W$ be the coarse space.
  The map $W \to V$ is an isomorphism over the dense open subset $U \subset V$, and is a quasi-finite map between normal spaces.
  By Zariski's main theorem, it is an isomorphism.
  Furthermore, as $\mathcal W \to V$ is unramified in codimension 1, so is $\mathcal W \to W$.
  Since $\mathcal W$ is normal and $W$ is smooth, purity of the branch locus \cite[Tag 0BMB]{sta:20} implies that $\mathcal W \to W$ is \'etale.
  As $\mathcal W \to V$ is an isomorphism over $U$, we see that $\mathcal W$ contains a copy of $U$ as a dense open substack.
  In particular, $\mathcal W$ has trivial generic stabilizers.
  By \cite[Lemma~4]{ger.sat:17}, we conclude that $\mathcal W \to W$ is an isomorphism.
  Since both $\mathcal W \to W$ and $W \to V$ are isomorphisms, their composite  $\mathcal W \to V$ is an isomorphism.
  We have proved that $\mathcal Z^\nu \to \mathcal X$ is an isomorphism \'etale locally on $\mathcal X$.
  We conclude that $\mathcal Z^\nu \to \mathcal X$ is an isomorphism.

  The composite of the inverse of $\mathcal Z^\nu \to \mathcal X$, the map $\mathcal Z^\nu \to \mathcal X \times \mathcal Y$, and the projection onto $\mathcal Y$ gives the required extension $\Phi \from \mathcal X \to \mathcal Y$.
\end{proof}
% Is smoothness of \mathcal X necessary? Does normality suffice. It certainly does if \mathcal Y is an algebraic space.

\begin{proof}[Proof of \autoref{thm:phi_extends}]
  Define a map $\phi \from |\overline{\mathcal {H}}_4^3(1/6+\epsilon)| \to |\mathcal X|$, consistent with the map induced by $\Phi$ on $U$ as follows.
  Let $k$ be an algebraically closed field and $a \from \spec k \to \overline{\mathcal{H}}_4^3(1/6 + \epsilon)$ a map.
  Let $\Delta$ be a DVR with residue field $k$ and $\alpha \from \Delta \to \overline{\mathcal{H}}_4^3(1/6 + \epsilon)$ a map that restricts to $a$ at the special point and maps the generic point $\eta$ to $U$.
  By \autoref{prop:stab}, there exists an extension $\beta \from \Delta \to \mathfrak X$ of $\Phi \circ \alpha |_\eta$.
  Let $b \from \spec k \to \mathfrak X$ be the central point of the extension.
  Set $\phi(a) = b$.
  \autoref{prop:stab} guarantees that $b$ depends only on $a$, and not on $\alpha$; so the map $\phi$ is well-defined.
  \autoref{prop:stab} also guarantees that $\phi$ is continuous in one-parameter families.
  Since $\overline{\mathcal H}^3_4(1/6 + \epsilon)$ is smooth, \autoref{lem:exists_extension} applies and yields the desired extension.
\end{proof}

%%% Local Variables:
%%% mode: latex
%%% TeX-master: "main"
%%% End:

\subsection{The boundary locus of $\mathfrak{X}$}
\label{subsec:boundaryloci}
Let $\mathfrak U \subset \mathfrak X$ be the open subset that parametrizes $(S, D)$ where $S \cong \PP^1 \times \PP^1$ and $D \subset S$ is a smooth curve of degree $(3,3)$.
The \emph{boundary} of $\mathfrak X$ refers to the complement $\mathfrak X \setminus \mathfrak U$.
Let $\mathfrak T \subset \overline{\mathcal H}_4^3(1/6+\epsilon)$ be the open subset that parametrizes $f \from C \to P$ where $P \cong \PP^1$, the curve $C$ is smooth, and the Tschirnhausen bundle of $f$ is $\O(3) \oplus \O(3)$.
We see that $\mathfrak T = \Phi^{-1}(\mathfrak U)$.
To understand the boundary of $\mathfrak X$, we are led to understanding $\overline{\mathcal{H}}_4^3(1/6 + \epsilon) \setminus \mathfrak T$.

We define some closed subsets of $\overline{\mathcal{H}}_4^3(1/6 + \epsilon) \setminus \mathfrak T$.
Before we do so, let us extend the notion of the Maroni invariant of a triple cover of $\PP^1$ to a triple cover of $P = \PP^1(\sqrt[a]p)$.
Recall that vector bundles on $P$ are direct sums of line bundles, and the line bundles are given by $\O(n)$ for $n \in \frac{1}{a} \Z$, where the generator $\O(-1/a)$ is the ideal sheaf of the unique stacky point $p$ \cite{mar.tha:12}.
\begin{definition}[Maroni invariant]
  \label{def:Maroni_rat'l_stack}
  Let $P = \PP^1(\sqrt[a]p)$, and let $f \from C \to P$ be a representable, finite, flat morphism of degree 3.
  Suppose $f_* \O_C / \O_P \cong \O(-m) \oplus \O(-n)$ for some $m, n \in \frac{1}{a} \Z$.
  Then the \emph{Maroni invariant} of $f$, denoted by $M(f)$, is the difference $|m-n|$.
\end{definition}

We now define various boundary loci of $\overline{\mathcal{H}}_4^3(1/6+\epsilon)$ based on the Maroni invariant and the singularities.

\begin{definition}[Tschirnhausen loci]
\label{def:boundary_loci}
Let $a, b, c$ be positive rational numbers.
Define the following closed subsets of $\overline{\mathcal{H}}_4^3(1/6+\epsilon)$.
\begingroup
\allowdisplaybreaks
\begin{align*}
  \mathfrak{Y}_0&:=\overline{\left\{ [f \from C \rightarrow P] \; | \; P \cong \mathbb{P}^1, \; M(f)=0, \; C' \text{ is singular}\right\}}\\
  \mathfrak{Y}_a&:=\overline{\left\{[f \from C \rightarrow P]\; | \; P \cong \mathbb{P}^1, \; M(f)=a \right\}}\\
  \mathfrak{Y}_{b,c}&:=\overline{\left\{[f \from C \rightarrow P]\; | \; P \; \text{is a rational chain of length}\; 2,  \; M(f) \text{ on each component is } b,c\right\}}
\end{align*}
\endgroup
Define $\mathfrak{Z}_a$, $\mathfrak{Z}_{b,c}$ to be the image of $\mathfrak{Y}_a$, $\mathfrak{Y}_{b,c}$ under $\Phi$, respectively.
\end{definition}

Since $\Phi$ is a proper map,  $\mathfrak Z_a$ and $\mathfrak Z_{b,c}$ are closed subsets of $\mathfrak X$.
We have seen in \autoref{sec:stable-replacement} that the corresponding cases there describe general members of $\mathfrak Z_a$ and $\mathfrak Z_{b,c}$ for suitable $a,b,c$.
By construction, the various $\mathfrak Z_a$ and $\mathfrak Z_{b,c}$ cover the boundary of $\mathfrak X$ as $a,b,c$ vary.

Let us identify $a, b, c$ that lead to non-empty loci in $\overline{\mathcal H}_4^3(1/6 + \epsilon)$.
Let us start with $\mathfrak Y_a$, keeping in mind that $a$ must be even.
Taking $a = 2$ yields the classical Maroni divisor $\mathfrak Y_2$.
Taking $a = 4$ yields the hyperelliptic divisor $\mathfrak Y_4$.
A generic point of $\mathfrak Y_4$ corresponds to $f \from C \cup \PP^1 \to \PP^1$, where $C$ is a smooth hyperelliptic curve of genus 4 attached to $\PP^1$ nodally at one point.
For $a > 4$, we have $\mathfrak Y_a = \varnothing$ (otherwise we would find $\mathbb{F}_a$ in the list of log quadric surfaces in pages~\pageref{f0} and \pageref{maroni4}).

Let us now consider $\mathfrak Y_{b,c}$.
First, observe that the node on the rational chain $P$ of length 2 has automorphism group of order 1 or 3.
In the case of trivial automorphism group, the non-empty cases are $\mathfrak Y_{1,1}$, $\mathfrak Y_{1,3}$, and $\mathfrak Y_{3,3}$.
A generic point of $\mathfrak Y_{1,1}$ corresponds to $f \from C_1 \cup C_2 \to P_1 \cup P_2$ where each $C_i$ is a smooth curve of genus 1.
A generic point of $\mathfrak Y_{1,3}$ corresponds to $f \from C_1 \cup C_2 \to P_1 \cup P_2$ where $C_1$ is a smooth curve of genus 1 and $C_2$ is the disjoint union of a smooth curve of genus 2 and $\PP^1$.
A generic point of $\mathfrak Y_{3,3}$ corresponds to $f \from C_1 \cup C_2 \to P_1 \cup P_2$ where each $C_i$ is a disjoint union of a smooth curve of genus 2 and $\PP^1$; they are attached so that the union $C_1 \cup C_2$ is connected.
In the case of an automorphism group of order 3, the only non-empty case is $\mathfrak Y_{1/3, 1/3}$.
A generic point of $\mathfrak Y_{1/3, 1/3}$ corresponds to $f \from C_1 \cup C_2 \to P_1 \cup P_2$ where $C_i$ is smooth of genus 2, and on the level of coarse spaces, $C_i \to P_i$ is a triple cover totally ramified over the node point on $P_i$.

From the discussion above, we see that the non-empty $\mathfrak Y_{a}$ and $\mathfrak Y_{b,c}$, namely $\mathfrak Y_0$, $\mathfrak Y_2$, $\mathfrak Y_4$, $\mathfrak Y_{1,1}$, $\mathfrak Y_{1,3}$, $\mathfrak Y_{3,3}$, and $\mathfrak Y_{1/3,1/3}$, are all irreducible of codimension 1 in $\overline {\mathcal H}_4^3(1/6 + \epsilon).$
Set
\[
  I = \{0, 2, 4, (1,1), (1,3), (3,3), (1/3, 1/3)\}.
\]
This is the set of possible subscripts of the $\mathfrak Y$'s.

\begin{proposition}
\label{prop:boundarydim}
For all $i \in I$ except $i = (1,1)$ and $i = (1,3)$, the loci $\mathfrak Z_i$ are of codimension 1 in $\mathfrak X$.
The locus $\mathfrak Z_{1,1}$ is of codimension 3, and $\mathfrak Z_{1,3}$ of codimension 2.
\end{proposition}

\begin{proof}
  For all $i \in I$, the $\mathfrak Y_i$ are irreducible, and hence so are the $\mathfrak Z_i$.
  Notice that $\mathfrak Z_0$ is of codimension one, since having a singular point for curves of class $(3,3)$ in $\PP^1 \times \PP^1$ induces a codimension 1 condition.

  For the rest, we find the dimension of the general fiber of $\Phi$ on $\mathfrak Y_i$.
  We first treat the cases of irreducible $S$.
  Given a generic $(S, D) \in \mathfrak Z_2$, we obtain the Tschirnhausen embedding $D \subset \mathbb{F}_2$ by taking the minimal resolution of the $A_1$ singularity of $S$.
  Similarly, for a general $(S, D) \in \mathfrak Z_4$, we can obtain the Tschirnhausen embedding $D \subset \F_4$ by undoing the transformation described in \autoref{subsec:hypellip}.
  To do so, we first take the minimal resolution of $S$ and contract one of the two $-1$ curves on the resolution.
  Therefore, we get that $\mathfrak Z_2$ and $\mathfrak Z_4$ are of codimension 1.
  
  We now consider the cases of reducible $S$.
  By considering the two components separately, we can reconstruct the Tschirnhausen embedding from a general $(S, D)$ in $\mathfrak Z_{1/3,1/3}$ and $\mathfrak Z_{3,3}$, up to finitely many choices.
  Therefore, we get that $\mathfrak Z_{1/3, 1/3}$ and $\mathfrak Z_{3,3}$ are of codimension 1.

  Let us now look at the two exceptional cases.
  For a general $(S,D) \in \mathfrak Z_{1,1}$, we have $S = S_1 \cup S_2$ where both components are isomorphic to $\PP^2$.
  To construct the Tschirnhausen surface $\mathbb{F}_1 \cup \mathbb{F}_1$ from $S$, we need to choose two points $p,q \in D:=S_1 \cap S_2$ to blow up on $S_1$ and $S_2$ respectively (the curve in $\mathbb{F}_1 \cup \mathbb{F}_1$ is simply the pre-image of $D$).
  Since $p$ and $q$ can be any points in $D \cong \PP^1$, a general fiber of $\Phi \from \mathfrak Y_{1,1} \to \mathfrak Z_{1,1}$ has dimension 2.
  Therefore, $\mathfrak Z_{1,1}$ is of codimension 3.
	
  For the a general $(S, D) \in \mathfrak Z_{1,3}$,  we have $S=S_1 \cup S_2$ where $S_1 \cong \PP^2$ and $S_2$ is the cone over twisted cubic.
  To construct the Tschirnhausen surface $\mathbb{F}_1 \cup \mathbb{F}_3$, we must choose a point on the double curve of $S_1 \cup S_2$ to blow up on $S_1$ (see \autoref{note:ell_bridge}).
  Hence, a general fiber of $\Phi \from \mathfrak Y_{1,3} \to \mathfrak Z_{1,3}$ has dimension 1.
  Therefore, $\mathfrak Z_{1,3}$ is of codimension 2.
\end{proof}

Although the indices $i \in I$ correspond bijectively with the divisors $\mathfrak Y_i$, when we pass to the $\mathfrak Z_i$, we have one coincidence.
\begin{proposition}
\label{lem:distinctness}
For indices $i \neq j$ in $I$, we have $\mathfrak Z_i \neq \mathfrak Z_j$
Nonempty $\mathfrak Z_a$'s and $\mathfrak Z_{b,c}$'s are distinct except $\mathfrak Z_{3,3}=\mathfrak Z_{\frac{1}{3},\frac{1}{3}}$ in $\mathfrak{X}$.
\end{proposition}
\begin{proof}
  For $i \neq j$ in $I \setminus \{(3,3), (1/3,1/3)\}$, the surfaces parametrized by the general points of $\mathfrak Z_i$ and $\mathfrak Z_j$ are non-isomorphic, as they have non-isomorphic singularities.
  That $\mathfrak Z_{3,3}$ is the same as $\mathfrak Z_{(1/3,1/3)}$ follows from \autoref{lem:F33asF_1/3_1/3}.
\end{proof}

\autoref{lem:distinctness} shows that boundary of $\mathfrak{X}$ has 4 divisorial components, namely $\mathfrak Z_0$, $\mathfrak Z_2$, $\mathfrak Z_4$ and $\mathfrak Z_{3,3}$.
The next proposition shows that they cover the entire boundary.

\begin{proposition} \label{thm:divisorial}
  $\mathfrak Z_{1,1}$ is contained in $\mathfrak Z_2$ and $\mathfrak Z_{1,3}$ is contained in $\mathfrak Z_4$.
  Therefore, the boundary of $\mathfrak{X}$ is divisorial.
\end{proposition}
\begin{proof}
  Let us first consider the case of $\mathfrak Z_2$ and $\mathfrak Z_{1,1}$.
  Recall that a general $(S, D) \in \mathfrak Z_{1,1}$ arises as the stabilization of a Tschirnhausen pair $(\F_1 \cup \F_1, D)$, and a general $(S, D)$ in $\mathfrak Z_2$ as the stabilization of a Tschirnhausen pair $(\F_2, D)$.
  We construct a family of Tschirnhausen pairs with generic fiber $\F_2$ degenerating to $\F_1 \cup \F_1$.

  Take a one parameter family $\pi \from B \to \Delta$ of smooth $\PP^1$'s degenerating to a nodal rational curve $P_1 \cup_p P_2$ over a DVR $\Delta$.
  Call $0 \in \Delta$ the special point and $\eta \in \Delta$ the generic point. Then $B_{\eta} \cong \PP^1$ and $B_0=P_1 \cup_p P_2$. Choose sections $s_i:\Delta \rightarrow B$ of $\pi$ for $i=1,2$ with $s_i(0) \in P_i\setminus\{p\}$ for $i = 1, 2$.
  Consider the vector bundle
  \[E =\mathcal{O}(s_1+s_2)\oplus \mathcal{O}(2s_1+2s_2)\]
  over $B$.
  Notice that the generic fiber $E_{\eta}$ is $\mathcal{O}(2)\oplus\mathcal{O}(4)$ and the special fiber $E_0$ is $\mathcal{O}(1,1)\oplus\mathcal{O}(2,2)$; here $\O(a,b)$ denotes the line bundle on $P_1 \cup P_2$ of degree $a$ on $P_1$ and degree $b$ on $P_2$.
  Let $\mathcal D \subset \mathcal \PP E$ be a general divisor of class $\pi^*\O(3) \otimes \pi^* \det E^\vee$.
  We can check that
  \[h^0(\pi^*\O(3) \otimes \pi^* \det E^\vee|_\eta) = h^0(\pi^*\O(3) \otimes \pi^* \det E^\vee|_0),\]
  so the divisor $\mathcal D_0 \subset \mathcal \PP(E_0)$ is general in its linear series.
  The covering $\mathcal D \to B$ over $\Delta$ gives a map $\mu \from \Delta \to \overline{\mathcal H}_4^3(1/6 + \epsilon)$.
  The composite $\Phi \circ \mu \from \Delta \to \mathfrak X$ maps $0$ to a general point of $\mathfrak Z_{1,1}$ and $\eta$ to a point of $\mathfrak Z_2$.
  It follows that $\mathfrak Z_{1,1}$ is contained in $\mathfrak Z_2$.
  
  The case of $\mathfrak Z_{1,3}$ and $\mathfrak Z_4$ is proved similarly by taking $E=\mathcal{O}(s_1)\oplus\mathcal{O}(2s_1+3s_2)$.
\end{proof}

%%% Local Variables:
%%% mode: latex
%%% TeX-master: "main"
%%% End:

\subsection{Comparison of $\mathfrak X$ with ${\mathcal M}_4$}
\label{subsec:birgeom}
In this section, we describe the relationship between $\mathfrak X$ and the moduli stack $\mathcal M_4$ of smooth curves of genus 4.

Denote by $\mathfrak M_4$ the (non-separated) moduli stack of all curves (proper, connected, reduced schemes of dimension 1) of arithmetic genus 4.
We have a forgetful map $\mathfrak X \to \mathfrak M_4$ that sends $(S, D)$ to $D$.
Let $\mathfrak X_0 \subset \mathfrak X$ be the open subset where $D$ is smooth.
\autoref{cor:sm_curve_loci} describes the surfaces appearing on $\mathfrak X_0$.
The forgetful map restricts to a map
\[ F \from \mathfrak X_0 \to \mathcal M_4.\]
\begin{proposition}\label{prop:repr_forgetful}
  $F$ is
  \begin{inparaenum}
  \item representable,
  \item proper, and
  \item restricts to an isomorphism
    \[ F \from \mathfrak X_0 \setminus \mathfrak Z_4 \to \mathcal M_4 \setminus \mathcal H_4,\]
    where $\mathcal H_4 \subset \mathcal M_4$ is the hyperelliptic locus.
  \end{inparaenum}
\end{proposition}
\begin{proof}
  \begin{inparaenum}
  \item 
    By \cite[Lemma~4.4.3]{abr.vis:02}, it suffices to show that $F \from \mathfrak{X}_0(\k) \to \mathcal{M}_4(\k)$ is a faithful map of groupoids.
    In other words, given any $(S,C) \in \mathfrak{X}_0(\k)$, we need to show that any automorphism $f$ of $(S,C)$ restricting to identity on $C$ is the identity on $S$.
    We break this into two cases.

    In the first case, suppose $C$ is not hyperelliptic.
    Then $C$ has a canonical embedding $C \subset \PP^3$.
    The linear series $|K_S + C|$ gives an embedding of $S$ in $\PP^3$ as a quadric surface.
    So $S$ is realized as the unique quadric surface in $\PP^3$ containing $C$.
    Note that every automorphism of $S$ extends uniquely to an automorphism of $\PP^3$.
    That is, we have an injection
    \[ \Aut(S) \subset \PGL_4(\k).\]
    Likewise, every automorphism of $C$ extends uniquely to an automorphism of $\PP^3$, so we also have an injection
    \[ \Aut(C) \subset \PGL_4(\k).\]
    It follows that every automorphism of $S$ that is the identity on $C$ is the identity on $S$.
    
    In the second case, suppose $C$ is hyperelliptic.
    Let $\widetilde S \to S$ be the minimal desingularization of $S$.
    Recall that $S$ has a $\frac 19(1,2)$ singularity and possibly an additional $A_1$ singularity.
    The map $\widetilde S \to S$ resolves the $\frac19(1,2)$ singularity to produce a chain of rational curves of self-intersection $(-5, -2)$.
    We have a unique fibration $\widetilde S \to \PP^1$ whose generic fiber is $\PP^1$.
    The $-5$ curve $\sigma$ obtained in the resolution is a section of this fibration.
    An automorphism $f$ of $S$ induces an automorphism $\widetilde f$ of $\widetilde S$.
    Note that $\widetilde f$ must preserve the fibration $\widetilde S \to \PP^1$ and the section $\sigma$.
    If $f$ also fixes $C$, then $\widetilde f$ fixes three points in a generic fiber of $\widetilde S \to \PP^1$, namely the point of $\sigma$, and the two points of $C$.
    It follows that $f$ is the identity on $S$.

  \item Since $\mathfrak X_0$ is separated and of finite type, so is $F$.
    For properness, we check the valuative criterion.
    Let $\pi \from C \to \Delta$ be a smooth proper curve of genus $4$.
    We may assume that the generic fiber $C_\eta$ is non-hyperelliptic and Maroni general.
    Let $(S_\eta, C_\eta)$ be an object of $\mathfrak X$ over the generic point $\eta$.
    We must extend it to an object of $\mathfrak X$ over $\Delta$ that gives $C \to \Delta$ under the map $F$.
    
    Since $C_\eta$ is a smooth, non-hyperelliptic curve, $S_\eta$ is the unique quadric surface containing $C_\eta$ in its canonical embedding.
    Possibly after a base change on $\Delta$, we have a line bundle $L$ on $C$ such that for all $t \in \Delta$, we have $\deg L_t = 3$ and $h^0(L_t) = 2$.
    If the central fiber $C_0$ is non-hyperelliptic, then $L_0$ is base-point free.
    In that case, we have a finite, flat, degree 3 map
    \[f \from C \to \PP^1_\Delta = \PP(\pi_* L).\]
    If $C_0$ is hyperelliptic, then $L_0$ is given by the hyperelliptic line bundle twisted by $\O(p)$ for some $p \in C_0$ and has a base point at $p$.
    After finitely many blow-ups and contractions of $(-2)$ curves centered on $p$, we obtain a family $\pi' \from C' \to \Delta$ and a finite, flat, degree 3 map
    \[ f \from C' \to \PP^1_\Delta = \PP(\pi_* L).\]
    The central fiber of $C' \to \Delta$ is the nodal union of $C_0$ and $\PP^1$ at $p$.
    In either case, $f$ yields a map $\Delta \to \overline{\mathcal H}_4^3(1/6 + \epsilon)$.
    Its composition with $\Phi \from \overline{\mathcal H}_4^3(1/6 + \epsilon) \to \mathfrak X$ gives a map $\Delta \to \mathfrak X$.
    From the description of stabilization for the central fiber of $f$ (see \autoref{subsec:MaroniSpecial} and \autoref{subsec:hypellip}), we see that $\Delta$ maps to $\mathfrak X_0$ and provides the necessary extension of $\eta \to \mathfrak X$ given by $(S_\eta, C_\eta)$.

  \item
    Let $p \from \spec \k \to \mathcal{M}_4 \setminus \mathcal H_4$ be a point given by a smooth, non-hyperelliptic curve $C$.
    The fiber of
    \begin{equation}\label{eqn:Frest}
      F \from \mathfrak X_0 \setminus \mathfrak Z_4 \to \mathcal M_4 \setminus \mathcal H_4
    \end{equation}
    over $p$ is a unique point, represented by the isomorphism class of $(S, C)$ where $S$ is the unique quadric surface containing the canonical image of $C$. Since $(S,C)$ is isomorphic to the pair of unique quadric surface containing the canonical image of $C$ (in $\PP^3$) and $C$, any automorphism of the pair $(S,C)$ in $\mathfrak X_0$ is given by an unique extension of an automorphism of $C$. Thus, $F$ induces an isomorphism of automorphism groups of $(S,C)$ and $C$.
    By Zariski's main theorem, we conclude that \eqref{eqn:Frest} is an isomorphism.
  \end{inparaenum}
\end{proof}

Using \autoref{prop:repr_forgetful}, we immediately deduce the following.
\begin{theorem}\label{thm:blowup_hypellip}
  The map $F$ induces an isomorphism of stacks
  \[ \mathfrak X_0 \xrightarrow{\sim} \Bl_{\mathcal H_4}\mathcal M_4.\]
\end{theorem}
\begin{proof}
  It suffices to check the statement \'etale locally on $\mathcal M_4$.
  So let $U$ be a scheme and $U \to \mathcal M_4$ an \'etale map.
  Let $H \subset U$ be the pre-image of $\mathcal H_4$.
  Likewise, let $X \to U$ be the pullback of $\mathfrak X_0 \to \mathcal M_4$ and $Z \subset X$ the pre-image of $\mathcal Z_4$.
  Note that $U$ and $H$ are smooth.
  We may assume that they are also connected (hence irreducible).
  
  Let $p$ be a point of $H$ whose image in $\mathcal H_4$ corresponds to the hyperelliptic curve $C$.
  By \autoref{cor:sm_curve_loci}, the (set-theoretic) fiber of $Z \to H$ over $p$ is $\PP^1$, given by the elements of the hyperelliptic linear series on $C$.
  Since $H$ is irreducible, and the fibers of $Z \to H$ are irreducible of the same dimension, $Z$ is also irreducible.
  We also know that $X \to U$ is an isomorphism over the complement of $H$.
  Since $H \subset U$ is smooth, $X$ is smooth, and $Z$ is irreducible, $X \to U$ is the blow-up at $H$ by \cite[Corollary]{sch:76}.
\end{proof}

Using \autoref{prop:repr_forgetful}, we also obtain the Picard group of $\mathfrak X$.
\begin{proposition} \label{cor:picrank}
  The rational Picard group $\Pic_\Q(\mathfrak X)$ of $\mathfrak X$ is of rank 4, and is generated by the classes of the four boundary divisors.
\end{proposition}
\begin{proof}
  We have a surjective map
  \[ \Pic_\Q(\mathfrak X) \to \Pic_\Q(\mathfrak X_0)\]
  given by pull-back, whose kernel is generated by the irreducible components of $\mathfrak X \setminus \mathfrak X_0$, namely $\mathfrak Z_0$ and $\mathfrak Z_{3,3}$.
  Since $\Pic_\Q(\mathcal M_4) = \langle  \lambda \rangle$ and $\mathcal H_4 \subset \mathcal M_4$ is of codimension 2, we have
  \[ \Pic_\Q(\mathfrak X_0 \setminus \mathfrak Z_4) = \Pic_\Q(\mathcal M_4 \setminus \mathcal H_4) = \Pic_\Q(\mathcal M_4) = \Q \langle  \lambda \rangle.\]
  The image of $\mathfrak Z_2$ in $\mathcal M_4$ is the Maroni divisor, which is linearly equivalent to a rational multiple of $\lambda$ (precisely, $\mathfrak Z_2 \sim 17\lambda$ by \cite[Theorem~IV]{sta:00}).
  Therefore, we get
  \[ \Pic_\Q\left(\mathfrak X_0 \setminus (\mathfrak Z_2 \cup \mathfrak Z_4)\right) = 0.\]
  Hence, $\Pic_\Q(\mathfrak X)$ is generated by $\mathfrak Z_0$, $\mathfrak Z_2$, $\mathfrak Z_4$, and $\mathfrak Z_{3,3}$.

  We now show that the 4 boundary divisors are linearly independent by test-curve calculations.
  Take 3 curves $C_1,C_2,C_3$ in $\mathfrak{X}$ as follows:
	\begin{align*}
		C_1&:=\text{a general pencil of }(3,3)\text{ curves in }\PP^1 \times \PP^1, \\
		C_2&:=\text{a curve meeting }\mathfrak Z_{3,3}\text{ but not contained in it},\\
		C_3&:=\text{a fiber of } \mathrm{Bl}_{\mathcal{H}_4}\mathcal{M}_4 \to \mathcal M_4 \text{ over a general point in } \mathcal H_4.
	\end{align*}
        The intersection matrix of $C_1,C_2,C_3$ and $\mathfrak Z_0,\mathfrak Z_{3,3},\mathfrak Z_4$ is as follows, where $*$ denotes a non-zero number and $?$ an unknown number.
	\[
	\begin{array}{l | c c c}
		& \mathfrak Z_0 & \mathfrak Z_{3,3} & \mathfrak Z_4\\
		\hline
		C_1 & 34 & 0 & 0\\
		C_2 & ? & * & ?\\
		C_3 & 0 & 0 & -1
	\end{array}
      \]
      We briefly explain how to obtain this matrix.
      The intersection points of $C_1$ with $\mathfrak Z_0$ correspond to the singular members of the pencil of $(3,3)$ curves on $\PP^1 \times \PP^1$.
      Since the pencil has $18$ base points and consists of curves of genus $g = 4$, the number of singular members is 
      \[ 18 + \chi(\PP^1 \times \PP^1) - \chi(\PP^1) \cdot (2-2g) = 34.
      \]
      The curve $C_1$ is visibly disjoint from $\mathfrak Z_{3,3}$ and $\mathfrak Z_4$.
      The second row follows from the choice of $C_2$.
      The exceptional divisor of $\mathrm{Bl}_{\mathcal{H}_4}\mathcal{M}_4 \to \mathcal M_4$ is disjoint from $\mathfrak Z_0$ and $\mathfrak Z_{3,3}$.
      Since $\mathfrak Z_4$ restricted to  $\mathrm{Bl}_{\mathcal{H}_4}$ is the exceptional divisor, its intersection with a non-trivial fiber is $-1$.
      
      Since the intersection matrix is invertible, we conclude that $\mathfrak Z_0,\mathfrak Z_{3,3},\mathfrak Z_4$ are linearly independent.
      It remains to show that $\mathfrak Z_2$ is linearly independent of these three.
      If $\mathfrak Z_2$ were a linear combination of $\mathfrak Z_0$, $\mathfrak Z_{3,3}$, and $\mathfrak Z_4$, then its restriction to $\mathfrak X_0$ would be a rational multiple of $\mathfrak Z_4$.
        But $\mathfrak Z_2$ and $\mathfrak Z_4$ are clearly linearly independent on $\mathcal X_0 = \Bl_{\mathcal H_4}\mathcal M_4$.
        Indeed, $\mathfrak Z_4$ is the exceptional divisor of the blow up and $\mathfrak Z_2$ is the pullback of a non-trivial divisor on $\mathcal M_4$.
\end{proof}

\autoref{thm:blowup_hypellip} implies that $\mathfrak{X}$ is a compactification of $\mathrm{Bl}_{\mathcal{H}_4}\mathcal{M}_4$.
We may ask whether $\mathfrak{X}$ is the blow up of the closure of $\mathcal{H}_4$ in $\overline{\mathcal M}_4$.
The answer is ``No.''
In fact, we can see that $F$ does not even extend to a morphism from $\mathfrak X$ to $\overline{\mathcal{M}}_4$.

To see this, observe that there is a stable log quadric $(\PP^1 \times \PP^1,C)$ where $C$ is an irreducible curve with a cuspidal singularity.
Let $p \in \mathfrak X$ be the point represented by this stable log quadric.
Then the rational map $F \from \mathfrak X \dashrightarrow \overline M_4$ is undefined at $p$.
Let $C \to \Delta$ be a one parameter family of $(3,3)$ curves on $\PP^1 \times \PP^1$ with central fiber $C$ and smooth general fiber.
The stable limit of such a family in $\overline M_4$ is $C^\nu \cup E$, where $C^\nu$ is the normalization of $C$ and $E$ is an elliptic curve attached nodally to $C^\nu$ at the pre-image of the cusp.
Furthermore, it is easy to see that we obtain all possible elliptic curves $E$ by making different choices of the one parameter family $\Delta$.
Hence, it is impossible to define $F$ at $p$.

The next natural question is whether there is a map from $\mathfrak{X}$ to an existing alternative compactification of $\mathcal M_4$?
Let us consider the alternative compactifications of $\mathcal M_4$ constructed in the Hassett--Keel program \cite{fed.smy:13}, which we now recall.
Let $\alpha \in [0,1]$ be such that $K_{\overline{\mathcal M}_4}+\alpha\delta$ is effective (here $\delta$ is the class of the boundary divisor of $\overline{\mathcal M}_4$), we have the space
\[
  \overline{M}_4(\alpha) = \Proj \bigoplus_{m \ge 0}H^0\left(\overline{\mathcal M}_4,m(K_{\overline{\mathcal M}_4}+\alpha\delta)\right).
\]
We restrict ourselves to $\alpha > 2/3 - \epsilon$ for a small enough $\epsilon$.
For such $\alpha$, the spaces $\overline{M}_4(\alpha)$ can be described as the good moduli spaces of various open substacks of the stack of all curves $\mathfrak M_4$ \cite{alp.fed.smy.ea:17}.
The answer, however, still turns out to be negative.

\begin{proposition} \label{prop:newcpct}
  For any value of $\alpha \in (2/3-\epsilon,1] \cap \Q$, the map $F$ does not extend to a morphism from $\mathfrak{X}$ to $\overline{M}_4(\alpha)$.
\end{proposition}

\begin{proof}
  There is a stable log quadric $(\PP^1 \times \PP^1, C)$ where $C$ is irreducible with an $A_4$ (rhamphoid cusp) singularity.
  Let $p$ be the point of $\mathfrak X$ corresponding to $(\PP^1 \times \PP^1, C)$.
  But $\overline M_4(\alpha)$ contains a point representing a curve with a rhamphoid cusp only if $\alpha \leq 2/3$.
  We conclude that for $\alpha > 2/3$, the rational map $F \from \mathfrak X \dashrightarrow \overline M_4(\alpha)$ must be undefined at $p$.
  Indeed, for $\alpha > 2/3$, the limit in $\overline M_4(\alpha)$ of a one parameter family of generically smooth $(3,3)$ on $\PP^1 \times \PP^1$ curves limiting to $C$ is $C^\nu \cup T$, where $C^\nu$ is the normalization of $C$ and $T$ is a genus 2 curve attached to $C^\nu$ at the pre-image of the rhamphoid cusp on $C$ and at a Weierstrass point of $T$ \cite[6.2.2]{has:00}.
  Furthermore, we can see that multiple Weierstrass genus 2 tails $T$ arise (in fact, all of them do) by different choices of the family.
  So $F$ cannot be defined at $p$.

  It remains to show that $F$ does not extend to a map to $\overline M_4(\alpha)$ for $2/3-\epsilon < \alpha \leq 2/3$.
  The culprit here is the locus $\mathfrak Z_{1,3}$.
  Let $p \in \mathfrak X$ be a generic point of $\mathfrak Z_{1,3}$.
  Recall that the curve in the pair corresponding to $p$ is a genus 2 curve with an elliptic bridge.
  We will show that the elliptic bridge causes $F \from \mathfrak X \dashrightarrow \overline M_4(\alpha)$ to be undefined at $p$.
  On one hand, $p$ lies in the closure of the hyperelliptic locus $\mathfrak Z_4$ by \autoref{thm:divisorial}.
  Therefore, if $F$ is defined at $p$, then $F(p)$ must lie in the closure of the hyperelliptic locus in $\overline M_4(\alpha)$.
  On the other hand, we construct a one parameter family $\Delta \to \mathfrak X$ with central fiber $p$ whose stable limit in $\overline M_4(\alpha)$ does not lie in the closure of the hyperelliptic locus.
  This will show that $F$ cannot be defined at $p$.

  To construct $\Delta$, start with a family $\mathcal P \to \Delta$ whose generic fiber $\mathcal P_\eta$ is $\PP^1$, whose special fiber $\mathcal P_0$ is a nodal rational chain of length 2, and whose total space $\mathcal P$ is non-singular.
  Take a vector bundle $\mathcal E$ on $\mathcal P$ such that $\mathcal E_\eta \cong \O(3) \oplus \O(3)$ and $\mathcal E_0 \cong \O(1,0) \oplus \O(2,3)$.
  Let $\mathcal C \subset \PP \mathcal E$ be a general divisor in the linear series $\O_{\PP E}(3) \otimes \det \mathcal E^\vee$.
  Observe that the central fiber $\PP \mathcal E_0$ is $\F_1 \cup \F_3$.
  The divisor $\mathcal C_0 \cap \F_1$ is the pre-image of a general plane cubic and is disjoint from the directrix.
  The divisor $\mathcal C_0 \cap \F_3$ is the disjoint union of the directrix and a hyperelliptic curve $H$ of genus $2$.
  The curve $H$ meets the elliptic curve nodally at two points, say $q$ and $r$, which are hyperelliptic conjugate.
  We have seen that the stabilization of the central fiber $(\PP \mathcal E_0, \mathcal C_0)$ is a point of $\mathfrak Z_{1,3}$ (\autoref{f13-case}).

  We now find the stable limit of the family $\mathcal C \to \Delta$ in $\overline M_4(\alpha)$.
  To do so, we must contract the rational tail and the elliptic bridge of $\mathcal C_0$.
  It will be useful to achieve this contraction in the family of surfaces $\PP \mathcal E \to \Delta$.
  Let $\mathcal X_1 \to \PP \mathcal E$ be the blow up of the directrix $\sigma \subset \F_1 \subset \PP \mathcal E_0$.
  From the sequence
  \[0 \rightarrow \O(-1) = N_{\sigma/\F_1} \rightarrow N_{\sigma/\mathcal X} \rightarrow N_{\F_1/ \mathcal X}\big|_{\sigma} = \O(-1)\rightarrow 0 \]
  we see that the normal bundle of $\sigma$ in $\PP \mathcal E$ is $\O(-1) \oplus \O(-1)$.
  Hence the exceptional divisor of the blow up is $\PP^1 \times \PP^1$ and it is disjoint from the proper transform of $\mathcal C$.
  The proper transform of $\F_1 \subset \PP \mathcal E_0$ is a copy of $\F_1$.
  The proper transform of $\F_3 \subset \PP \mathcal E_0$ is $\Bl_p \F_3$ where $p = \sigma \cap \F_3$.
  We contract the exceptional divisor $\PP^1 \times \PP^1 \subset \mathcal X_1$ in the other direction, namely along the fibers opposite to the fibers of the projection $\PP^1 \times \PP^1 \to \sigma$, obtaining a threefold $\mathcal X_2$ (this is a contraction of type \cite[3.3.1]{mor:82}).
  The central fiber of $\mathcal X_2 \to \Delta$ is $\PP^2 \cap \Bl_p \F_3$.
  We next contract the $\PP^2$ in the central fiber to obtain $\mathcal X_3$ (this is a contraction of type \cite[3.3.2]{mor:82}).
  The central fiber of $\mathcal X_3 \to \Delta$ is $\F_2$.
  On this $\F_2$, the central fiber of the proper transform of $\mathcal C$ is a divisor of class $-3/2 K$.
  More precisely, it is the union of the directrix $s$ and a curve of class $2s + 6f$ with a node on the directrix.
  Finally, let $\mathcal X_3 \to \mathcal X_4$ be the small contraction of $s$ (obtained by taking the anti-canonical model).
  The central fiber of $\mathcal X_4 \to \Delta$ is the cone over a plane conic, namely a singular quadric surface in $\PP^3$.
  Let $\mathcal C_4 \subset \mathcal X_4$ be the proper transform of $\mathcal C$.
  The central fiber $C$ of $\mathcal C_4 \to \Delta$ is a tacnodal curve whose normalization is $H$; the pre-image of the tacnode is the hyperelliptic conjugate pair $\{q, r\}$.
  Most importantly, however, we have $C \subset Q$ where $Q \subset \PP^3$ is a quadric surface.
  As a result, we see that $C$ has a canonical embedding in $\PP^3$.
  Therefore, it is not in the closure of the hyperelliptic locus.
  This observation completes the proof of the assertion that $F$ cannot be defined at a generic point of $\mathfrak Z_{1,3}$.
\end{proof}

Denote by $X$ the coarse space of $\mathfrak X$.
\autoref{prop:newcpct} says that the relationship between $X$ and the known modular compactifications of $M_4$ is complicated.

We close with some questions.
\begin{question} \label{Q:bir_rel}
  How does the birational map $ X \dashrightarrow \overline M_4(\alpha)$ decompose into more elementary birational transformations (divisorial contractions and flips)?
  Is $X$ a log canonical model of $\Bl_{\mathcal H_4}\overline {\mathcal M}_4$?
\end{question}

Recall that $X$ can be interpreted as the KSBA compactification of weighted pairs $(S, wC)$ with weight $w = 2/3 + \epsilon$ for sufficiently small $\epsilon<\frac{1}{30}$.
An answer to the following question will be interesting in itself, and also potentially useful for \autoref{Q:bir_rel}.
\begin{question}\label{Q:chamber_rel}
  How does the KSBA compactification change as the weight $w$ varies in $(2/3, 1]$?
\end{question}
%%% Local Variables:
%%% mode: latex
%%% TeX-master: "main"
%%% End:

\bibliography{math}
\bibliographystyle{amsalpha}
\end{document}